\documentclass[a4paper, reqno, twoside]{amsart}
\usepackage{amsfonts,amssymb,graphics,amsthm,amsmath,amsxtra,amscd,mathrsfs}
\usepackage{latexsym}
\usepackage[all]{xy}
\usepackage{enumitem}
\usepackage{a4wide}
\usepackage[bookmarks=false]{hyperref}
\usepackage{cite}
\usepackage[normalem]{ulem}
\usepackage[usenames]{color}
\usepackage[mathscr]{euscript}
\usepackage[usenames,dvipsnames,svgnames,table]{xcolor}
\usepackage{lmodern}
\usepackage{txfonts}

\newtheorem{theorem}{\sc Theorem}[section]
\newtheorem{proposition}[theorem]{\sc Proposition}

\newtheorem{lemma}[theorem]{\sc Lemma}
\newtheorem{corollary}[theorem]{\sc Corollary}
\theoremstyle{definition}
\newtheorem{definition}[theorem]{\sc Definition}

\newtheorem{example}[theorem]{\sc Example}

\theoremstyle{remark}
\newtheorem{notations}[theorem]{\sc Notations}
\newtheorem{remark}[theorem]{\sc Remark}

\IfFileExists{tcilatex.tex}{\input{tcilatex}}{}

\newcommand{\tensor}[1]{\otimes_{\scriptscriptstyle{#1}}}

\newcommand{\Sf}[1]{\mathsf{#1}}

\newcommand{\cat}[1]{\mathcal{#1}}
\newcommand{\rmod}[1]{\Sf{Mod}_{\Sscript{#1}}}

\renewcommand{\hom}[3]{\mathrm{Hom}_{\Sscript{#1}}\left(#2,\,#3\right)}

\newcommand{\bara}[1]{\overline{#1}}

\newcommand{\End}[2]{\mathrm{End}_{\Sscript{#1}}(#2)}

\newcommand{\rhom}[3]{\mathrm{Hom}_{\scriptscriptstyle{\text{-}#1}}(#2,#3)}

\newcommand{\bcirc}{{\boldsymbol{\circ}}}

 \newcommand{\id}{\mathrm{Id}}

\newcommand{\Bim}[2]{{}_{\Sscript{#1}}\mathsf{Bim}{}_{\Sscript{#2}}}
\newcommand{\FBim}[2]{{}^{}_{\Sscript{#1}}\mathsf{Bim}{}^{\mathsf{flt}}_{\Sscript{#2}}}
\newcommand{\CBim}[2]{{}^{}_{\Sscript{#1}}\mathsf{Bim}{}^{\mathsf{c}}_{\Sscript{#2}}}

\newcommand{\varfun}[3]{{#1} \colon {#2} \rightarrow {#3}}

\newcommand{\lfun}[5]{{#1} \colon {#2} \longrightarrow {#3};\quad  \Big({#4} \longmapsto {#5}\Big)}
\newcommand{\sfun}[5]{{#1} \colon {#2} \rightarrow {#3};\, \Big({#4} \mapsto {#5}\Big)}

\newcommand{\prlimit}[2]{\varprojlim_{#1}\left({#2}\right)}
\newcommand{\injlimit}[2]{\varinjlim_{#1}\left({#2}\right)}

\newcommand{\CHom}[3]{\mathsf{Hom}{}^{\mathsf{c}}_{\Sscript{#1}}\left({#2},{#3}\right)}
\newcommand{\Hom}[3]{\mathsf{Hom}_{\Sscript{#1}}\left({#2},{#3}\right)}
\newcommand{\HA}{\mathsf{CHAlgd}_{\K}}
\newcommand{\CHA}{\mathsf{CHAlgd}_{\K}^{\text{c}}}
\newcommand{\Img}[1]{\mathsf{Im}\left({#1}\right)}

\newcommand{\FHom}[3]{\mathsf{Hom}^{\mathsf{flt}}_{\Sscript{#1}}\left({#2},{#3}\right)}
\newcommand{\CEnd}[2]{\mathsf{End}^{\mathsf{cnt}}_{#1}\left({#2}\right)}
\newcommand{\diff}{\mathsf{Diff}}
\newcommand{\Calg}{\mathrm{CAlg}_{\K }}
\newcommand{\proj}[1]{\mathsf{proj}(#1)}
\newcommand{\ann}[1]{\mathsf{Ann}\left({#1}\right)}
\renewcommand{\ker}[1]{\mathrm{Ker}\left({#1}\right)}
\newcommand{\what}[1]{\widehat{#1}}
\newcommand{\cmptens}[1]{~\what{\otimes}_{\Sscript{#1}}~}

\newcommand{\gr}{\mathrm{gr}}
\definecolor{bostonuniversityred}{rgb}{0.8, 0.0, 0.0}

\newcommand{\ntau}[2]{\tau_{\Sscript{#1}}^{\Sscript{#2}}}
\newcommand{\greekn}[3]{{#1}_{\Sscript{#2}}^{\Sscript{#3}}}
\newcommand{\eg}{e.g.~}

\newcommand{\N}{\mathbb{N}}
\newcommand{\K}{\Bbbk}
\newcommand{\C}{\mathbb{C}}
\newcommand{\R}{\mathbb{R}}

\newcommand{\sS}{\mathscr{S}}

\newcommand{\cB}{{\mathcal B}}
\newcommand{\cD}{{\mathcal D}}
\newcommand{\cE}{{\mathcal E}}
\newcommand{\cF}{{\mathcal F}}

\newcommand{\cH}{{\mathcal H}}
\newcommand{\cI}{{\mathcal I}}
\newcommand{\cJ}{{\mathcal J}}
\newcommand{\cK}{{\mathcal K}}
\newcommand{\cL}{{\mathcal L}}
\newcommand{\cM}{{\mathcal M}}

\newcommand{\cR}{{\mathcal R}}
\newcommand{\cS}{{\mathcal S}}

\newcommand{\cU}{{\mathcal U}}
\newcommand{\cV}{{\mathcal V}}

\newcommand{\Sscript}[1]{\scriptscriptstyle{#1}}

\newcommand{\mnm}{\mu_{\Sscript{n,m}}}

\newcommand{\taun}{\tau_{\Sscript{n}}}

\newcommand{\taup}{\tau_{\Sscript{p}}}
\newcommand{\tauq}{\tau_{\Sscript{q}}}
\newcommand{\thetan}{\theta_{\Sscript{n}}}

\newcommand{\limn}{\underset{n\to\infty}{\lim}}

\newcommand{\limk}{\underset{k\to\infty}{\lim}}
\newcommand{\liml}{\underset{l\to\infty}{\lim}}
\newcommand{\limxn}{\underset{n\to\infty}{\lim}(x_n)}

\newcommand{\umas}{u_{\Sscript{+}}}
\newcommand{\umin}{u_{\Sscript{-}}}
\newcommand{\tuau}{\cU_{\Sscript{A}}\tensor{A} \cU_{\Sscript{A}}}
\newcommand{\tfuafu}[2]{F^{#1}\cU_{\Sscript{A}}\tensor{A}F^{#2}\cU_{\Sscript{A}}}
\newcommand{\ltuau}{\cU_{\Sscript{A}}\tensor{A} {}_{\Sscript{A}}\cU}

\newcommand{\Rep}[1]{\mathrm{Rep}_{\Sscript{\cM}}(\mathcal{#1})}

\begin{document}
\allowdisplaybreaks

\title[Topological tensor product of bimodules, complete Hopf Algebroids and convolution algebras.]{Topological tensor product of bimodules, complete Hopf Algebroids and convolution algebras}

\author{Laiachi El Kaoutit}
\address{Universidad de Granada, Departamento de \'{A}lgebra and IEMath-Granada. Facultad de Educaci\'{o}n, Econon\'ia y Tecnolog\'ia de Ceuta. Cortadura del Valle, s/n. E-51001 Ceuta, Spain}
\email{kaoutit@ugr.es}
\urladdr{http://www.ugr.es/~kaoutit/}

\author{Paolo Saracco}
\address{University of Turin, Department of Mathematics ``Giuseppe Peano'', via Carlo Alberto 10, I-10123 Torino, Italy}
\email{p.saracco@unito.it}
\urladdr{sites.google.com/site/paolosaracco}

\date{\today}
\thanks{This paper was written while P. Saracco was member of the ``National Group for Algebraic and Geometric Structures, and their Applications'' (GNSAGA - INdAM). His stay as visiting researcher at the campus of Ceuta of the University of Granada was financially supported by IEMath-GR. Research supported by the Spanish Ministerio de Econom\'{\i}a y Competitividad  and the European Union FEDER, grant MTM2016-77033-P}

\begin{abstract}
Given a finitely generated and projective Lie-Rinehart algebra, we show that there is a continuous homomorphism of complete commutative Hopf algebroids between the completion of the finite dual of its universal enveloping Hopf algebroid and the associated convolution algebra. The topological Hopf algebroid structure of this convolution algebra is here clarified, by providing an  explicit description of its topological antipode as well as of its other structure maps. Conditions under which that homomorphism becomes an homeomorphism are also discussed.  These results, in particular, apply to the smooth global sections of any Lie algebroid over a smooth (connected) manifold and they lead a new formal groupoid scheme to enter into the picture. 
In the Appendix we develop the necessary machinery  behind complete Hopf algebroid constructions, which involves also the topological tensor product of filtered bimodules over filtered rings.
\end{abstract}

\maketitle

\vspace{-0.8cm}
\begin{small}
\tableofcontents
\end{small}

\pagestyle{headings}

\section{Introduction}

\subsection{Motivation and overviews} \label{ssec:muchocaldo}

Let $\cM$ be a smooth connected real manifold and denote by $A=C^{\infty}(\cM)$ its smooth $\R$-algebra \cite[\S 4.1]{nestruev}. All vector bundles considered below are over $\cM$ and by definition they are locally trivial with constant rank \cite[\S 11.2]{nestruev}. For a given Lie algebroid $\cL$ with anchor map $\omega:\cL\,\to\,T\cM$ (see Example \ref{exam:VayaCon}), we consider the category $\mathrm{Rep}_{\Sscript{\cM}}(\cL)$ consisting of those vector bundles $\cE$ with a (right) $\cL$-action. That is, an $A$-module morphism $\varrho_{-}: \Gamma(\cL) \to \End{\R}{\Gamma(\cE)}$ which is a Lie algebra map satisfying $\varrho_{X}(f s) =f \varrho_{X}(s) +\Gamma(\omega)_{X}(f) s$, for any section $s \in \Gamma(\cE)$ and any function $f \in A$. Morphisms in the category $\mathrm{Rep}_{\Sscript{\cM}}(\cL)$  are morphisms of vector bundles $\varphi: \cE \to \cF$ which commute with the actions, that is, such that $\Gamma(\varphi) \circ \varrho_{X} \, =\, \varrho_{X}' \circ \Gamma(\varphi)$, for every section $X \in \Gamma(\cL)$. Thanks to the properties of the smooth global sections functor $\Gamma$, expounded in \cite[Theorems 11.29, 11.32 and 11.39]{nestruev}, the category $\mathrm{Rep}_{\Sscript{\cM}}(\cL)$ is a (not necessarily abelian) symmetric rigid monoidal category, which is endowed with a fiber functor $\boldsymbol{\omegaup}: \mathrm{Rep}_{\Sscript{\cM}}(\cL) \to \proj{A}$ to the category of finitely generated and projective $A$-modules. The identity object in this monoidal structure is the line bundle $\cM \times \R$ with action  $\Gamma(\omega)$ (composed with the canonical injection $\mathrm{Der}_{\Sscript{\R}}(A) \subset \End{\R}{A}$), and the action on the tensor product of two objects $\cE$ and $\cF$ in $\Rep{L}$, is given by 
$$
(\varrho \tensor{A} \varrho')_{X}: \Gamma(\cE)\tensor{A}\Gamma(\cF) \longmapsto \Gamma(\cE)\tensor{A}\Gamma(\cF), \quad \Big(  s\tensor{A}s' \longmapsto \varrho_{X}(s)\tensor{A} s' + s\tensor{A}\varrho_{X}'(s')\Big), 
$$
for every $X \in \Gamma(\cL)$. The symmetry is the one given by the tensor product over $A$.  Up to the natural $A$-linear isomorphism $\Gamma(\cE^{*}) = \Gamma\big(\mathsf{Hom}(\cE,\cM\times \R)\big) \cong \mathsf{Hom}_{\mathcal{C}^{\infty}(\cM)}\big(\Gamma\left(\cE\right),\mathcal{C}^{\infty}(\cM)\big) = \Gamma(\cE)^{*}$ (see \eg \cite[Theorem 11.39]{nestruev}), the action on the dual vector bundle $\cE^{*}$ is provided by the following $A$-module and Lie algebra map $\varrho^{*}: \Gamma(\cL) \to \End{\R}{\Gamma(\cE)^{*}}$, sending any section $X \in \Gamma(\cL) $ to the $\R$-linear map 
$$
\varrho^{*}_{X} : \Gamma(\cE)^{*} \longrightarrow \Gamma(\cE)^{*}, \quad \Big( s^{*} \longmapsto \Gamma(\omega)_{X} \circ s^{*}- s^{*} \circ \varrho_{X} \Big)
$$
(see also, for example, \cite[\S1.4]{Crainic}, \cite[page 731]{ModularClasses}).

The Tannaka reconstruction process shows  then that the pair $(\Rep{L}, \boldsymbol{\omegaup})$ leads to a (universal) commutative Hopf algebroid which we denote by $(A, \cU^{\circ})$ and then to a complete commutative (or topological) Hopf algebroid $(A, \what{\cU^{\circ}})$, where $A$ is considered as a discrete topological ring. It turns out that $(A, \cU^{\circ})$ is the finite dual Hopf algebroid, in the sense of \cite{LaiachiGomez},  of the (co-commutative) universal enveloping Hopf algebroid $(A, \cU:=\cV_{\Sscript{A}}(\Gamma(\cL))$  of the Lie-Rinehart algebra $(A,\Gamma(\cL))$, because in this case ($\cM$ connected) the category $\Rep{L}$ can be identified with the category of (right) $\cU$-modules with finitely generated and projective underlying $A$-module structure.  In this way, we end up with a continuous $(A\tensor{\R}A)$-algebra map $\zeta: \cU^{\circ} \to \cU^{*}$, where the latter is the convolution algebra of $\cU$. This algebra admits a topological Hopf algebroid structure, rather than just a topological bialgebroid one as it is known in the literature, see \cite{Kapranov:2007}. It is worthy to point out that, even if the Tannaka reconstruction process may be applied in this context, the pair $(\Rep{L}, \boldsymbol{\omegaup})$ does not necessarily form a Tannakian category in the sense of \cite{Deligne:1990}.

The main motivation of this paper is to set up, in a self-contained and comprehensive way,  the basic notions and tools behind the theory of complete commutative (or topological) Hopf algebroids and the connection of this theory with Lie algebroids as above. 
Our aim is to provide explicitly the topological Hopf algebroid structure on $\cU^{*}$ mentioned previously and to show that the completion $\what{\zeta}$ of the map $\zeta$ leads not only to a continuous map, but also to a continuous morphism of topological Hopf algebroids. Besides, we will provide some conditions under which $\what{\zeta}$ becomes an homeomorphism, as well. 

Now, knowing that the map $\what{\zeta}$ is always a morphism of topological Hopf algebroids, one may analyse the distinguished  case  when $\cL=T \cM$ is the tangent bundle with its obvious Lie algebroid structure. It can be shown, by using for instance \cite[Corollary 15.5.6]{MR1811901},  that in this case the universal enveloping algebroid $\cU$ can be identified with the algebra $\mathrm{Diff}(A)$ of all differential operators on $A$. On the other hand, one can extend the duality between $l$-differential operators and $l$-jets of the bundle of $l$-jets of functions on $\cM$ (see e.g. \cite[Theorem 11.64]{nestruev}) to an homeomorphism between the convolution algebra $\cU^{*}$ and the algebra of infinite jets $\cJ(A):=\what{A\tensor{\R}A}$ (see Example \ref{exam:Sacarrelli} below). These are isomorphic not only  as  complete algebras, but also as topological Hopf algebroids.  Summing up, we have a commutative diagram 
\begin{equation}
\begin{gathered}\label{eq:maindiagram}
\xymatrix@R=15pt@C=30pt{ \what{\cU^{\bcirc}} \ar@{->}^-{\what{\zeta}}[rr]   &  & \mathrm{Diff}(A)^{*}  \\ & \cJ(A) \ar@{->}^-{\cong}_-{\what{\vartheta}}[ru]  \ar@{->}^-{\what{\etaup}}[lu] & }
\end{gathered}
\end{equation}
of topological Hopf algebroids, where the algebra maps $\vartheta: A\tensor{\R}A \to \cU^{*}$  and $\etaup:A\tensor{\R}A \to \cU^{\circ}$ define the source and the target of $\cU^{*}$ and $\cU^{\circ}$ respectively. 
Furthermore, in view of \cite[Proposition A.5.10]{Kapranov:2007}, if the map $\what{\zeta}$ is an homeomorphism, then  both  ${\rm Spf}(\what{\cU^{\circ}})$ and ${\rm Spf}(\mathrm{Diff}(A)^{*})$ can be seen as formal groupoids  that integrate the Lie algebroid $\cL$.   Nevertheless, it is not always guaranteed that $\what{\zeta}$ is an homeomorphism, even for the simplest case, see \cite{LaiachiPaolo} for a counterexample and Remark \ref{rem:laventanadelfrente} for more details.

\subsection{Description of the main results}

{\renewcommand{\thetheorem}{{A}}

Recall that a complete Hopf algebroid in the sense of \S\ref{ssec:CCHAlgds} is a co-groupoid object in the category of  complete commutative algebras, where the latter is endowed with a suitable topological tensor product (see Appendix \ref{sec:CBCF}). 

Let $(A,L)$ be a Lie-Rinehart algebra whose underlying module $L_{\Sscript{A}}$ is finitely generated and projective  and consider $\cU:=\cV_{\Sscript{A}}(L)$, its universal enveloping Hopf algebroid (see \S\ref{ssec:ULA} for details). This is a co-commutative (right) Hopf algebroid to which one can associate two complete commutative (or topological) Hopf algebroids, where the base algebra $A$ is trivially filtered. The first one of these is the completion $(A, \what{\cU^{\circ}})$ of the finite dual $(A,\cU^{\circ})$, the commutative Hopf algebroid constructed as in \cite{LaiachiGomez},
see  \S\ref{ssec:Fdual} for more details on this construction.  The second one is $(A, \cU^{*})$, where $\cU^{*}$ is  the commutative convolution algebra of $\cU$. The topological Hopf structure of this algebroid is explicitly retrieved in \S\ref{sec:Ustar}. It is noteworthy to mention that, in general, even if a given co-commutative Hopf algebroid  possesses an admissible filtration (see \S\ref{ssec:FUstra} and \S\ref{ssec:AFiltration}), it is not clear, at least  to us,  how to endow its convolution algebra with a topological antipode.  Namely, one needs the translation map to be an homeomorphism, or at least a filtered algebra map. A condition which we show that is always fulfilled for $\cV_{\Sscript{A}}(L)$ with $L$ as above.

As we already mentioned, at the algebraic level we have a canonical $(A\tensor{}A)$-algebra map $\zeta: \cU^{\circ} \to \cU^{*}$. The following result, which concerns the properties of its completion $\what{\zeta}$, is our main theorem (stated below as Theorem \ref{thm:triangle} and its Corollary \ref{coro:Equivalence}).

\begin{theorem}\label{thm:A}
Let $(A,\cU)$ be a co-commutative (right) Hopf algebroid with an admissible filtration (see \S\ref{ssec:FUstra}) and  assume that the translation map $\delta$ of $\cU$, defined in \eqref{eq:amore}, is a filtered algebra map. Then the algebra map $\zeta: \cU^{\bcirc} \to \cU^*$ factors through a continuous morphism $\what{\zeta}: \what{\cU^{\bcirc}} \to \cU^*$ of complete Hopf algebroids. 

In particular, this applies to $\cU=\cV_{\Sscript{A}}(L)$ for any Lie-Rinehart algebra $(A,L)$ with $L_{\Sscript{A}}$ a finitely generated and projective module. Moreover, 
if $\cU^{\circ}$ is an Hausdorff topological space with respect to its canonical adic topology and $\what{\zeta}$ is an homeomorphism, then $\zeta$ is injective and therefore there is an equivalence of symmetric rigid monoidal categories between the category of right  $L$-modules  and the category of right $\cU^{\circ}$-comodules, with finitely generated and projective underlying $A$-module structure.
\end{theorem}

As a consequence, we also discuss conditions under which the map $\what{\zeta}$ is an homeomorphism, for instance when $\what{\zeta}$ is injective and $\what{\vartheta}$ is a filtered isomorphism (as in the example at the end of \S\ref{ssec:muchocaldo}). See Proposition \ref{prop:zeroneveropen} for further conditions.}

\subsection{Notation and basic notions}

Throughout this paper and if not stated otherwise, $\K$ will denote a fixed  field (eventually a commutative ring in the appendices), all algebras, coalgebras and Lie algebras will be assumed to be over $\K$ (eventually, all modules will have an underlying structure of central $\K$-modules. That is, the left and right actions satisfy $k\cdot m=m\cdot k$ for every $m\in M$ and $k\in \K$. Such a category of central $\K$-modules will be denoted by $\rmod{\K}$).
Given two algebras $R$ and $S$, for an \emph{$(S, R)$-bimodule} $M$ we denote the underlying module structures by ${}_{\Sscript{S}}M$, $M_{\Sscript{R}}$ and ${}_{\Sscript{S}}M_{\Sscript{R}}$.  The \emph{right and left duals} stand for the bimodules $M^*:= \hom{-R}{M}{R}$ and ${}^*M:=\hom{S-}{M}{S}$, which we consider canonically as $(R, S)$-bimodule and $(S, R)$-bimodule, respectively. 
The unadorned tensor product $\otimes$ is understood to be over $\K$.
For  an algebra $A$, an \emph{$A$-(co)ring} is a (co)monoid inside the monoidal category of $A$-bimodules. 

We denote by $\Calg$  the category of commutative algebras.
Given an algebra $A$ in $\Calg$, and two bimodules ${}_{\Sscript{A}}M_{\Sscript{A}}$ and ${}_{\Sscript{A}}N_{\Sscript{A}}$. The notations like $M_{\Sscript{A}} ~\tensor{A}~ {}_{\Sscript{A}}N$,  ${}_{\Sscript{A}}M ~\tensor{A}~ N_{\Sscript{A}}$ or ${}_{\Sscript{A}}M ~\tensor{A}~ {}_{\Sscript{A}}N$, are used for different tensor products over $A$, in order to specify the sides on which we are making such a tensor product.
The subscript joining bimodules will be omitted in that notation whenever the sides are clear from the context.  For more details on the tensor product construction, see for instance \cite[A II.50, \S3]{MR}.

Recall finally that algebras and bimodules over algebras form a bicategory $\cB im_{\K}$. For the notion of bicategories and 2-functors (i.e., morphisms of bicategories) we refer the reader to \cite{Benabou}.

\section{Complete commutative Hopf algebroids.}\label{sec:CHA}

In this section we recall explicitly the structure maps involved in the definition of complete commutative Hopf algebroids. All algebras are assume to be commutative, in particular, all Hopf algebroids are understood to be so. We will freely use the notations and  notions expounded in  Appendix \ref{sec:CBCF}. Following the standard literature on the subject, we will assume $\K$ to be a field, even if the results presented here can be stated for a commutative ring in general.

\subsection{Commutative Hopf algebroids: Definition and examples}\label{ssec:CHAlgds}
We recall here from \cite[Appendix A1]{ravenel} the definition of commutative Hopf algebroid. We also expound some examples which will be needed in the forthcoming sections. 

A \emph{Hopf algebroid} is a \emph{cogroupoid} object in the category $\mathrm{CAlg}_{\K }$ (equivalently a groupoid in the category of affine schemes). Thus, a Hopf algebroid consists of a  pair of algebras $\left( A,\mathcal{H}\right) $ together with a diagram of algebra maps:
\begin{equation}\label{Eq:loop}
\xymatrix@C=45pt{ A \ar@<1ex>@{->}|(.4){\scriptstyle{s}}[r] \ar@<-1ex>@{->}|(.4){\scriptstyle{t}}[r] & \ar@{->}|(.4){ \scriptstyle{\varepsilon}}[l] \ar@(ul,ur)^{\scriptstyle{\cS}} \cH \ar@{->}^-{\scriptstyle{\Delta}}[r] & \cH \tensor{A}\cH , }
\end{equation}
where to perform the tensor product, $\cH$ is considered as an $A$-bimodule of the form ${}_{\scriptstyle{s}}\cH_{\scriptstyle{t}}$, i.e., $A$ acts on the left through $s$ while it acts on the right through $t$.
The maps $s,t:A\rightarrow \mathcal{H}$ are called the \emph{source} and the \emph{target} respectively,  $\varepsilon :\mathcal{H}\rightarrow A$ the \emph{counit},  $\Delta :\mathcal{H}\rightarrow \mathcal{H}\tensor{A} \mathcal{H}$ the \emph{comultiplication} and $\cS:\mathcal{H}\rightarrow \mathcal{H}$ the \emph{antipode}. These have to satisfy the following compatibility conditions.
\begin{itemize}
\item The datum $(\cH, \Delta, \varepsilon)$ has to be a coassociative and counital comonoid in $\Bim{A}{A}$, i.e., an $A$-coring. At the level of groupoids, this encodes a unitary and associative composition law between arrows.
\item The antipode has to satisfy $\cS \circ s=t$, $\cS \circ t = s$ and $\cS^2=\id_{\cH}$, which encode the fact that the inverse of an arrow interchanges source and target and that the inverse of the inverse is the original arrow.
\item The antipode has to satisfy also $\cS(h_1)h_2=(t\circ \varepsilon)(h)$ and $h_1\cS(h_2)=(s\circ \varepsilon)(h)$, which encode the fact that the composition of a morphism with its inverse on either side gives an identity morphism (the notation $h_1\otimes h_2$ is a variation of the Sweedler's Sigma notation, with the summation symbol understood, and it stands for $\Delta(h)$).
\end{itemize}
Note that there is no need to require that $\varepsilon \circ s= \id_A= \varepsilon \circ t$, as it is implied by the first condition.

A \emph{morphism of Hopf algebroids} is a pair of algebra maps
$\left( \phi _{\Sscript{0}},\phi _{\Sscript{1}}\right) :\left( A,\mathcal{H}\right)
\rightarrow \left( B,\mathcal{K}\right)$
such that
\begin{eqnarray*}
\phi _{1}\circ s =s\circ \phi _{0},&\qquad& \phi _{1}\circ t=t\circ \phi
_{0}, \\
\Delta \circ \phi _{1} =\chi \circ \left( \phi _{1}\otimes _{A}\phi
_{1}\right) \circ \Delta ,&\qquad& \varepsilon \circ \phi _{1}=\phi _{0}\circ
\varepsilon , \\
\cS\circ \phi _{1} =\phi _{1}\circ \cS &&
\end{eqnarray*}
where $\chi :\mathcal{K}\tensor{A}\mathcal{K\rightarrow }\mathcal{K} \tensor{B}\mathcal{K}$ is the obvious map induced by $\phi_0$, that is $\chi \left(h\tensor{A}k\right) =h\tensor{B}k$.

\begin{example}\label{exm:Halgd}
Here are some common examples of Hopf algebroids (see  \cite{ElKaoutit:2015} for more):
\begin{enumerate}
\item Let $A$ be an algebra. Then the pair $(A, A\otimes_{}A)$ admits a Hopf algebroid structure given by $s(a)=a\otimes_{}1$, $t(a)=1\otimes_{}a$, $S(a\otimes_{}a')=a'\otimes_{}a$, $\varepsilon(a\tensor{}a')=aa'$ and $\Delta(a\otimes_{}a')= (a\otimes_{}1) \tensor{A} (1\otimes_{}a')$, for any $a, a' \in A$.
\item Let $(B,\Delta, \varepsilon, \sS)$ be a Hopf algebra and $A$ a right $B$-comodule algebra with coaction $A \to A\otimes_{}B$, $a \mapsto a_{\Sscript{(0)}} \otimes_{} a_{\Sscript{(1)}}$. This means that $A$ is a right $B$-comodule and the coaction is an algebra map (see e.g. \cite[\S 4]{Montgomery:1993}).   Consider  the algebra  $\cH= A\otimes_{}B$ with  algebra extension $ \eta: A\otimes_{}A \to  \cH$, $a'\otimes_{}a \mapsto a'a_{\Sscript{(0)}}\otimes_{}a_{\Sscript{(1)}}$. Then $(A,\cH)$ has  a structure of Hopf algebroid, known as a \emph{split Hopf algebroid}:
$$
\Delta(a\otimes_{}b) = (a\otimes_{}b_{\Sscript{1}}) \otimes_{A} (1_{\Sscript{A}}\otimes_{}b_{\Sscript{2}}), \;\;\varepsilon(a\otimes_{}b)=a\varepsilon(b),\;\; \cS(a\otimes_{}b)= a_{\Sscript{(0)}}\otimes_{}  a_{\Sscript{(1)}}\sS(b).
$$
\item Let $B$ be as in part $(2)$ and $A$ any algebra. Then $(A, A\otimes_{}B\otimes_{}A)$ admits in a canonical way a structure of Hopf algebroid.  For $a,a'\in A$ and $b\in B$, its structure maps are given as follows
\begin{gather*}
s(a)=a\otimes 1_{\Sscript{B}}\otimes 1_{\Sscript{A}}, \quad t(a)=1_{\Sscript{A}}\otimes 1_{\Sscript{B}}\otimes a, \quad \varepsilon(a\otimes b\otimes a')=aa'\varepsilon(b), \\
\Delta(a\otimes b\otimes a')= \big(a\otimes b_1\otimes 1_{\Sscript{A}}\big) \tensor{A} \big(1_{\Sscript{A}}\otimes b_2\otimes a'\big), \quad \cS(a\otimes b\otimes a')=a'\otimes \sS(b)\otimes a.
\end{gather*}
 \end{enumerate}
\end{example}

\subsection{Complete commutative Hopf algebroids: Definition and examples}\label{ssec:CCHAlgds}
We keep the notation from Appendix \ref{sec:CBCF} and we refer to the same section for all definitions and results as well.

In the spirit of \cite[Appendix A]{quillen}, our next aim is to show that the completion functor from Appendix \ref{sec:CBCF} now induces a functor from the category of Hopf algebroids $\HA$ to that of complete Hopf algebroids $\CHA$  in the sense of \cite[\S 1]{MR1320989}, for example. To our knowledge this is not immediately clear. To perform such a construction one needs tools from the $I$-adic topology  \cite[chap~0~\S~7]{MR0163908} as well as from linear topology in bimodules and their topological tensor product as retrieved in the Appendices.

Assume we are given a diagram $S \leftarrow A \rightarrow R$ of filtered algebras and consider the filtered $(R,S)$-bimodule $R\tensor{A}S$ with filtration given as in Equation \eqref{eq:filtrations}:

$$
\cF_{n}(R\tensor{A}S)= \sum_{p+q=n} \Img{F_pR \tensor{A} F_qS}.
$$
This tensor product  becomes then a filtered algebra over $A$, i.e., $R\tensor{A}S$ is a filtered $A$-algebra.

Using this notion of filtered tensor product,  the definition of Hopf algebroid as given in \S\ref{ssec:CHAlgds} can be canonically adapted to the filtered context. Thus,  a pair $(A,\cH)$  of filtered algebras  is said to be a \emph{filtered Hopf algebroid} provided that there is a diagram as in equation~\eqref{Eq:loop}, where the involved maps satisfy the compatibility conditions in the filtered sense.  Morphisms between filtered Hopf algebroids are easily understood and the category so obtained  will be denoted by $\HA^{\Sf{flt}}$. The following example shows that there is a canonical functor $\HA \to \HA^{\Sf{flt}}$.

\begin{example}\label{exam:Hflt}
Consider a Hopf algebroid $(A,\cH)$. Assume that $A$ is trivially filtered and endow $\mathcal{H}$ with the
augmented filtration $F_{0}\mathcal{H=H}$ and $F_{n}\mathcal{H}:=K^{n}$ for
every $n\geq 1$, where $K:=\ker{ \varepsilon :\mathcal{H}
\rightarrow A} $ and $\mathcal{H}\tensor{A}\mathcal{H}$ with the usual tensor product filtration
$$
\cF_{0}\left( \mathcal{H}\tensor{A}\mathcal{H}\right) =\mathcal{H
}\tensor{A}\mathcal{H},\quad \cF_{n}\left( \mathcal{H}\tensor{A}\mathcal{H}
\right) =\sum_{p+q=n}\Img{K^{p}\tensor{A}K^{q}}
$$
for every $n\geq 1$.  Then $(A,\cH)$ is a filtered Hopf algebroid.
\end{example}

As it is already known, the usual (algebraic) tensor product  of complete bimodules  is not necessarily a complete bimodule. Thus, in order to introduce complete Hopf algebroids, one needs to use the topological (or complete) tensor product (see Definition \ref{def:CTP} and Theorem \ref{th:adjunction} for more details). In this way,  a \emph{complete Hopf algebroid} is a pair $(A,\cH)$ of complete algebras together with a diagram of filtered algebra maps similar to the one of equation \eqref{Eq:loop}, with the exception that $\Delta: \cH \to\cH\cmptens{A}\cH$, and where all the maps satisfy compatibility conditions analogous to that of a Hopf algebroid. In different but equivalent words, a complete Hopf algebroid is a cogroupoid object in the category of complete algebras. Henceforth, by complete Hopf algebroid we will always mean a \emph{commutative} complete Hopf algebroid.

\begin{remark}\label{rem:cmptensalgebra}
A detail has to be highlighted here. Assume that $A$ and $\cH$ are complete algebras and that $\cH$ is also an $A$-bimodule via two structure maps $s,t:A\to\cH$ as above. Then $\cH\cmptens{A}\cH$ is a complete algebra. Indeed, it can be checked that the obvious multiplication
\begin{equation*}
\sfun{\mu_{\Sscript{\cH\tensor{A}\cH}}}{\left(\cH\tensor{A}\cH\right)\otimes\left(\cH\tensor{A}\cH\right)}{\cH\tensor{A}\cH}{(x\tensor{A} y)\otimes (x'\tensor{A}y')}{xx'\tensor{A} yy'}
\end{equation*}
is well-defined and makes of $\cH\tensor{A}\cH$ a filtered algebra.
Since $\what{(-)}$ is monoidal (cf.~Corollary  \ref{prop:moncatcomplbimod}), we have that $\what{\cH\tensor{A}\cH}=\cH\cmptens{A}\cH$ is a complete algebra with $\what{\mu}$ as a multiplication.
\end{remark}

The following result can be seen as a consequence of Theorem \ref{thm:Athm}. Nevertheless, for the convenience of reader, we outline here the proof.
\begin{proposition}\label{coro:CHc}
The completion functor induces a functor
$$
\xymatrix{  \HA^{\Sf{flt}} \ar@{->}^-{\what{(-)}}[rr] & & \CHA     }
$$
to the category of complete Hopf algebroids with continuous morphisms of Hopf algebroids. In particular, we have the following composition of functors
$$
\xymatrix@R=15pt{  \HA^{\Sf{flt}} \ar@{->}^-{\what{(-)}}[rr] & & \CHA   \\ \HA \ar@{->}[u]   \ar@/_1ex/@{-->}[urr] & &   }
$$
\end{proposition}
\begin{proof}
Let $(A,\cH)$ be a filtered Hopf algebroid with filtered algebra maps $s,t,\varepsilon,\Delta,\cS$. In particular, $\cH$ is an object in $\FBim{A}{A}$. Consider $\what{A}$ and $\what{\mathcal{H}}$, which are complete modules as well as complete algebras (see Theorem \ref{thm:Athm}). We have that $\what{\cH}$  is an object in $\CBim{\what{A}}{\what{A}}$ and we have complete algebra maps
\begin{gather*}
\varfun{\what{s}~,~\what{t}}{\what{A}}{\what{\cH}},\qquad \varfun{\what{\varepsilon }}{\what{\cH}}{
\what{A}}, \\
\varfun{\what{\Delta }}{\what{\cH}}{\what{\left(
\cH\tensor{A}\cH\right) }\stackrel{\eqref{Eq:Psi}}{\cong}\what{\cH}~\cmptens{\Sscript{\what{A}}}~\what{\cH}},\quad \text{and} \quad \varfun{\what{\cS}}{\what{\cH}}{\what{\cH}}.
\end{gather*}
These maps satisfy the same axioms as the original ones, because $\what{\left( -\right) }$ is a monoidal functor by Corollary \ref{prop:moncatcomplbimod}. The unique detail that needs perhaps a few words more is the antipode condition.  Consider the maps ${c_l}:{\cH\tensor{A}\cH}\to{\cH}$ and ${c_r}:{\cH\tensor{A}\cH}\to{\cH}$ such that $c_l\left(x\tensor{A} y\right)=\cS(x)y$ and  $c_r\left({x\tensor{A} y}\right)={x\cS(y)}$ respectively, for all $x,y\in\cH$. These allow us to write the antipode conditions as the commutativity of the diagram
\begin{equation*}
\xymatrix @C=55pt @R=15pt {
\cH & \cH \otimes \cH \ar[d]^-{p} \ar[r]^-{\mu~\circ~\left(\cH~\otimes~ \cS\right)} \ar[l]_-{\mu~\circ~\left(\cS~\otimes~\cH\right)} & \cH \\
 & \cH\tensor{A} \cH \ar[ul]^-{c_l} \ar[ur]_-{c_r} &  \\
A \ar[uu]^-{t} & \cH \ar[l]_-{\varepsilon} \ar[r]^-{\varepsilon} \ar[u]^-{\Delta} & A \ar[uu]_-{s}
}
\end{equation*}
One can check that $p$, $c_l$ and $c_r$ are all filtered. Indeed, for $p$ it is immediate and for $c_l$ and $c_r$ it follows from the fact that $\cS$ is filtered. We can now apply the functor $\what{(-)}$ to get a commutative diagram
\begin{equation}\label{Eq:TopAntip}
\xymatrix @C=55pt @R=15pt {
\what{\cH} & \what{\cH}~\what{ \otimes}~ \what{\cH} \ar[d]^-{\what{p}} \ar[r]^-{\what{\mu}~\circ~\left(\what{\cH}~\what{\otimes}~ \what{\cS}\right)} \ar[l]_-{\what{\mu}~\circ~\left(\what{\cS}~\what{\otimes}~\what{\cH}\right)} & \what{\cH} \\
 & \what{\cH}\cmptens{\what{A}} \what{\cH} \ar[ul]^-{\what{c_l}} \ar[ur]_-{\what{c_r}} &  \\
\what{A} \ar[uu]^-{\what{t}} & \what{\cH} \ar[l]_-{\what{\varepsilon}} \ar[r]^-{\what{\varepsilon}} \ar[u]^-{\what{\Delta}} & \what{A} \ar[uu]_-{\what{s}}
}
\end{equation}
which shows that $\what{\cS}$ is the antipode of $\what{\cH}$.
Therefore,  $\left( \what{A},\what{
\mathcal{H}}\right) $ is a complete  Hopf algebroid.

Let  $\left( \phi _{\Sscript{0}},\phi _{\Sscript{1}}\right) :\left( A,\mathcal{H}\right)
\rightarrow \left( B,\mathcal{K}\right)$ be a morphism of filtered Hopf algebroids. Hence we can consider $\varfun{\what{\phi_0}}{\what{A}}{\what{B}}$ and $\varfun{\what{\phi_1}}{\what{\mathcal{H}}}{\what{\mathcal{K}}}$ and these are morphisms of complete algebras. Since $\chi:\cK\tensor{A}\cK\to \cK\tensor{B}\cK$ is filtered,
$\left(\what{\phi _{0}},\what{\phi _{1}}\right) $ becomes a morphism of complete Hopf algebroids by the functoriality of $\what{(-)}$. In light of Example \ref{exam:Hflt}, $\what{(-)}$ restricts to a functor
\begin{equation*}
\varfun{\what{(-)}}{\HA}{\CHA},
\end{equation*}
and this finishes the proof.
\end{proof}

\begin{example}[The complete Hopf algebroid of infinite jets]\label{exam:Sacarrelli}
Let $A$ be an algebra and consider the pair $(A,A\tensor{}A)$ as an Hopf algebroid in a canonical way.  Then the pair $(A, \what{A\tensor{}A})$ is a complete Hopf algebroid by Proposition \ref{coro:CHc},  where $A$ is trivially filtered and $A\tensor{}A$ is given the $K$-adic topology where $K:=\ker{\mu_{A}:A\tensor{} A\to A}$ is the kernel of the multiplication. The complete algebra $\what{A\tensor{}A}$ is known under the name \emph{algebra of infinite jets}. This terminology is justified by looking  at the case when $A=C^{\infty}(\cM)$. Namely, in this case,  $\what{A\tensor{}A}$  coincides with the inverse limit of the smooth global sections $\Gamma(\cJ^{l}(\cM))$ of the bundle of $l$-jets of functions  on $\cM$, see \cite[\S 11.46]{nestruev}, \cite[Proposition 9.4 (iv)]{Krasilshchik:1997}.  
\end{example}

\section{The complete commutative Hopf algebroid structure of the convolution algebra}\label{sec:Ustar}

The main task of this section is to show that the convolution algebra of a given (right) co-commutative Hopf algebroid,  which is endowed with an admissible filtration (see subsection \ref{ssec:AFiltration}) and  whose translation map is a filtered algebra map, is a complete Hopf algebroid in the sense of subsection \ref{ssec:CCHAlgds}. At the level of topological (right) bialgebroids this was mentioned in \cite[A.5]{Kapranov:2007}. However, it seems that the literature is lacking a precise construction for a topological antipode.  Here we will supply this by providing the explicit description of all the involved maps in the complete Hopf algebroid structure of the convolution algebra. Such a description will be crucial in proving the results of the forthcoming section. The prototype  example, which we have in mind and which fulfils the above assumptions, is the convolution algebra of the universal Hopf algebroid of  a finitely generated and projective Lie-Rinehart algebra.

\subsection{Co-commutative Hopf algebroids: Definition and examples}\label{ssec:CoHAlg}
Next, we  recall the definition of a (right) co-commutative Hopf algebroid (see for instance \cite[A.3.6]{Kapranov:2007}, compare also with \cite[Definition 2.5.1]{Kowalzig}).  
A \emph{(right) co-commutative Hopf algebroid} over a commutative algebra is the datum of a commutative algebra $A$,  a possibly  noncommutative algebra $\cU$ and an algebra map $s=t: A\to \cU$  landing not necessarily in the center of $\cU$, with the following additional structure maps:
\begin{itemize}
\item A  morphism of right $A$-modules $\varepsilon: \cU \to A$ which satisfies $\varepsilon(uv)=\varepsilon(\varepsilon(u)v)$, for all $u, v \in \cU$;
\item An $A$-ring map $\Delta: \cU \to \cU \times_{\Sscript{A}} \cU$, where the module
$$
\cU\times_{\Sscript{A}} \cU:=\left\{ \sum_{\Sscript{i}} u_{\Sscript{i}}\tensor{A} v_{\Sscript{i}} \in \cU_{\Sscript{A}}~ \tensor{A} ~\cU_{\Sscript{A}} \mid~ \sum_{\Sscript{i}} au_{\Sscript{i}}\tensor{A} v_{\Sscript{i}}= \sum_{\Sscript{i}} u_{\Sscript{i}}\tensor{A} av_{\Sscript{i}}  \right\}
$$
is endowed with the algebra structure
\begin{equation*}
\sum_{\Sscript{i}} u_{\Sscript{i}}\times_{\Sscript{A}} v_{\Sscript{i}} ~ .~ \sum_{\Sscript{j}} u'_{\Sscript{j}}\times_{\Sscript{A}}v'_{\Sscript{j}} = \sum_{\Sscript{i,j}} u_{\Sscript{i}}u'_{\Sscript{j}} \times_{\Sscript{A}} v_{\Sscript{i}}v'_{\Sscript{j}}, \qquad 1_{\cU\times_{\Sscript{A}} \cU}=1_{\cU}\tensor{A}1_{\cU}
\end{equation*}
and the $A$-ring structure given by the algebra map $1: A \rightarrow \cU \times_{\Sscript{A}}\cU,~ \Big(  a \mapsto a \times_{\Sscript{A}} 1_{\Sscript{\cU}}=   1_{\Sscript{\cU}}  \times_{\Sscript{A}} a \Big)$;
\end{itemize}
subject to the conditions
\begin{itemize}
\item $\Delta$ is coassociative, co-commutative in a suitable sense and has $\varepsilon$ as a right and left counit;
\item the canonical map
\begin{equation}\label{Eq:beta}
\lfun{\beta}{\cU_{\Sscript{A}} ~\tensor{A}~ {}_{\Sscript{A}}\cU}{\cU_{\Sscript{A}}~\tensor{A}~ \cU_{\Sscript{A}}}{u\tensor{A}v}{uv_{\Sscript{1}}\tensor{A}v_{\Sscript{2}}}
\end{equation}
is bijective, where we denoted $\Delta(v)=v_{\Sscript{1}}\otimes_{A}v_{\Sscript{2}}$ (summations understood). As a matter of terminology, the map $\beta^{-1}(1\tensor{A}-):\cU\to \cU_{\Sscript{A}}\tensor{A}{_{\Sscript{A}}\cU}$ is the so-called \emph{translation map}. The following is a standard notation:
\begin{equation}\label{eq:amore}
\delta(u):=\beta^{-1}(1\tensor{A}u)=u_-\tensor{A}u_+ \quad\textrm{(summation understood)}.
\end{equation}
\end{itemize}

Right $A$-linear maps $\cU_{\Sscript{A}} \to A$ form the module $\cU^{*}$ with a structure of algebra given by the \emph{convolution product}
\begin{equation}\label{Eq:convolution}
f *g: \cU \longrightarrow  A, \quad \Big( u \longmapsto f(u_{\Sscript{1}}) g(u_{\Sscript{2}}) \Big).
\end{equation}
It comes endowed with an algebra map
\begin{equation}\label{Eq:vartheta}
\vartheta=s_*\otimes_{}t_*: A\otimes_{}A \longrightarrow \cU^{*}, \quad \Big( a'\otimes_{}a \longmapsto \left[ u \mapsto \varepsilon(a u)a'  \right]  \Big).
\end{equation}

\begin{example}\label{exam:URSO}
Let $A=\Bbbk[X]$ be the polynomial algebra with one variable. Consider its associated first Weyl algebra $\cU:=A[Y, \partial/\partial X]$, that is, its differential operators algebra. Then the pair $(A,\cU)$ is a (right) co-commutative Hopf algebroid with structure maps: 
$$
\Delta(Y)\,=\, 1\tensor{A}Y+ Y\tensor{A}1,\quad \varepsilon(Y)\,=\, 0, \quad Y_{\Sscript{-}}\tensor{A} Y_{\Sscript{+}}\,=\, 1\tensor{A}Y - Y\tensor{A}1.
$$
\end{example}

\subsection{The universal enveloping Hopf algebroid of a Lie-Rinehart algebra}\label{ssec:ULA}

Let $A$ be a commutative algebra over a field $\K$ of characteristic $0$ and denote by $\mathrm{Der}_{\Sscript{\K}}(A)$ the Lie algebra of all linear derivations of $A$.  Consider a Lie algebra $L$ which is also an $A$-module, and let $\omega: L \to \mathrm{Der}_{\Sscript{\K}}(A)$ be an $A$-linear morphism of Lie algebras.
Following \cite{Rin:DFOGCA}, the pair $(A,L)$ is called a \emph{Lie-Rinehart algebra} with \emph{anchor} map $\omega$ provided that
\begin{equation*}
{[ X, aY ]} = a{[X,Y]}+X(a)Y,
\end{equation*}
for all $X, Y \in L$ and $a, b \in A$, where $X(a)$ stands for $\omega(X)(a)$.

Apart from the natural examples $(A,\mathrm{Der}_{\Sscript{\K}}(A))$ (with anchor the identity map), another  basic source of examples are the smooth global sections of a given Lie algebroid over a smooth manifold.

\begin{example}\label{exam:VayaCon}
A \emph{Lie algebroid}  is a  vector bundle $\mathcal{L} \to \mathcal{M}$ over a smooth manifold, together with a map $\omega: \mathcal{L} \to T\mathcal{M}$ of vector bundles and a Lie structure $[-,-]$ on the  vector space  $\Gamma(\cL)$  of global smooth sections of $\mathcal{L}$, such that the induced map $\Gamma(\omega): \Gamma(\cL) \to \Gamma(T\mathcal{M})$ is a Lie algebra homomorphism, and for all $X, Y \in \Gamma(\cL)$ and any $f \in \mathcal{C}^{\infty}(\mathcal{M})$  one has
\begin{equation}\label{eq:LieAlgd}
[X,fY]\,=\, f[X,Y]+ \Gamma(\omega)(X)(f)Y.
\end{equation}
Then the pair $(\mathcal{C}^{\infty}(\mathcal{M}), \Gamma(\cL))$ is obviously a Lie-Rinehart algebra.
\end{example}

\begin{remark}
As it has been pointed out by the referee, the fact that the map $\Gamma(\omega): \Gamma(\cL) \to \Gamma(T\mathcal{M})$ in Example \ref{exam:VayaCon} is a Lie algebra homomorphism is a consequence of the Jacobi identity and of Relation \eqref{eq:LieAlgd} (see \eg \cite{Grabowski, Herz, Magri}). Therefore, it should be omitted from the definition of a Lie algebroid. Nevertheless, we decided to keep the somewhat redundant definition above to make it easier for the unaccustomed reader to see the parallel with Lie-Rinehart algebras.
\end{remark}

The \emph{universal enveloping Hopf algebroid} of a
Lie-Rinehart algebra is an algebra $ \cV_{\Sscript{A}}\left( L\right) $
endowed with a morphism $\iota _{\Sscript{A}}:A\rightarrow  \cV_{\Sscript{A}}\left( L\right) $ of algebras and an $A$-linear Lie algebra morphism $\iota _{\Sscript{L}}:L\rightarrow  \cV_{\Sscript{A}}\left( L\right) $ such that
\begin{equation}\label{eq:compLRalg}
\iota _{\Sscript{L}}\left(X\right) \iota _{\Sscript{A}}\left( a\right) - \iota _{\Sscript{A}}\left( a\right) \iota _{\Sscript{L}}\left( X\right)=\iota _{\Sscript{A}}\left( X
\left( a\right) \right)
\end{equation}
for all $a\in A$ and $X\in L$, which is universal with respect to this
property. In details, this means that if $\left( W,\phi _{\Sscript{A}},\phi _{\Sscript{L}}\right) $ is another algebra with a morphism $\phi _{\Sscript{A}}:A\rightarrow W$ of algebras and a morphism $\phi _{\Sscript{L}}:L\rightarrow W$ of Lie algebras and $A$-modules such that
\begin{equation*}
\phi _{\Sscript{L}}\left( X\right)\phi _{\Sscript{A}}\left( a\right)  -\phi _{\Sscript{A}}\left( a\right) \phi _{\Sscript{L}}\left( X\right)=\phi _{\Sscript{A}}\left( X \left(a\right) \right) ,
\end{equation*}
then there exists a unique algebra morphism $\Phi : \cV_{\Sscript{A}}
\left( L\right) \rightarrow W$ such that $\Phi \, \iota _{\Sscript{A}}=\phi _{\Sscript{A}}$ and $
\Phi \, \iota _{\Sscript{L}}=\phi _{\Sscript{L}}$.

Apart from the well-known constructions of \cite{Rin:DFOGCA} and \cite{MoerdijkLie}, the universal enveloping Hopf algebroid of a Lie-Rinehart algebra $(A,L)$ admits several other equivalent realizations. In this paper we opted for the following. Set $\eta :L\rightarrow A\otimes L;\, X\longmapsto 1_{\Sscript{A}}\otimes X$ and consider the tensor $A$-ring $T_{\Sscript{A}}\left( A\otimes L\right) $ of the $A$-bimodule $A\otimes L$. It can be shown that
\begin{equation*}
 \cV_{\Sscript{A}}\left( L\right) \cong \frac{T_{\Sscript{A}}\left( A\otimes L\right) }{\cJ}
\end{equation*}
where the two sided ideal $\cJ$ is generated by the set

\begin{equation*}
\cJ:=\left\langle \left.
\begin{array}{c}
\eta \left( X\right) \otimes _{\Sscript{A}}\eta \left( Y\right) -\eta \left( Y\right) \otimes _{\Sscript{A}}\eta \left( X\right) -\eta \left( \left[ X,Y\right] \right) , \\
 \eta \left( X\right) \cdot a- a\cdot \eta \left( X\right) -\omega \left( X\right)\left( a\right)
\end{array}
\right| X,Y\in L, \; a\in A\right\rangle
\end{equation*}

We have the algebra morphism $\iota _{\Sscript{A}}:A\rightarrow
\cV_{\Sscript{A}}\left( L\right);\, a\longmapsto a+\cJ$ and the right $A$-linear Lie algebra map $\iota _{\Sscript{L}}:L\rightarrow \cV_{\Sscript{A}}\left( L\right);\, X\longmapsto \eta \left( X\right) +\cJ$ that
satisfy the compatibility condition \eqref{eq:compLRalg}. It turns out that $\cV_{\Sscript{A}}(L)$ is a co-commutative right $A$-Hopf algebroid with structure maps induced by the assignment
\begin{gather*}
\varepsilon \left( \iota _{\Sscript{A}}\left( a\right) \right)   =a, \qquad \varepsilon \left( \iota _{\Sscript{L}}\left( X\right) \right) =0, \\
\Delta \left( \iota _{\Sscript{A}}\left( a\right) \right)  =\iota _{\Sscript{A}}\left( a\right) \times _{\Sscript{A}}1_{\Sscript{ \cV_{\Sscript{A}}\left( L\right)}
}=1_{ \Sscript{\cV_A( L)} }\times _{\Sscript{A}}\iota _{\Sscript{A}}\left( a\right) , \\
\Delta \left( \iota _{\Sscript{L}}\left( X\right) \right)  =\iota _{\Sscript{L}}\left( X\right) \times _{\Sscript{A}}1_{ \Sscript{\cV_A\left( L\right)}
}+1_{\Sscript{\cV_A\left( L\right)} }\times _{\Sscript{A}}\iota _{\Sscript{L}}\left( X\right) , \\
\beta^{-1}\left( 1_{\Sscript{ \cV_A\left( L\right) }} \tensor{A } \iota_{\Sscript{A} }\left( a \right)\right)   = \iota _{\Sscript{A}}\left( a\right) \tensor{A}1_{\Sscript{ \cV_A\left( L\right)} }=1_{\Sscript{ \cV_A\left( L\right) }}\tensor{A}\iota _{\Sscript{A}}\left( a\right) ,  \\
\beta^{-1}\left( 1_{ \Sscript{\cV_A\left( L\right) }} \tensor{A } \iota_{\Sscript{L }}\left( X \right)\right)   = 1_{ \Sscript{\cV_A\left( L\right) }}\tensor{A}\iota _{\Sscript{L}}\left( X\right) -\iota
_{\Sscript{L}}\left( X\right) \tensor{A}1_{\Sscript{ \cV_A\left( L\right)} }.
\end{gather*}
Another realization for $\cV_{\Sscript{A}}(L)$ can be obtained as a quotient of the smash product (right) $A$-bialgebroid $A\#U_{\K}(L)$, as introduced firstly by Sweedler in \cite{sweedler}, by the ideal $\cI:=\langle a\#X- 1\#aX \mid a\in A, X\in L\rangle$.

\begin{example}\label{exam:laventana}
The first Weyl algebra considered in Example \ref{exam:URSO}, is  in fact the universal enveloping Hopf algebroid of the Lie-Rinehart algebra $(A, \mathrm{Der}_{\Sscript{\K}}(A))$, where $A=\K[X]$ is the polynomial algebra. 
\end{example}

\subsection{Admissible filtrations on general rings.}\label{ssec:AFiltration}
Let $\cU$ be an $A$-ring or, equivalently, let $A \to \cU$ be a ring extension\footnote{Even if we plan to apply the results of this subsection to a co-commutative Hopf algebroid $(A,\cU)$, these are interesting on their own and this justifies the choice of presenting them in the present form.}. 
Throughout this subsection we assume $A$ to be trivially filtered (i.e., $F^{n}A=A$ for all $ n\geq 0$, zero otherwise) and we refer to Appendix \ref{sec:TGFr} for a more detailed treatment in the framework of filtered bimodules.  Following \cite[Definition A.5.1]{Kapranov:2007}, we say that $\cU$ has a \emph{right admissible filtration}, if $A \subset \cU$ (as a ring) and there is  an increasing filtration $\cU=\bigcup_{\Sscript{n \, \in \, \N }} F^{n}\cU$ of $A$-subbimodules such that $F^0\cU=A$, $F^n\cU \,.\, F^m\cU \subseteq F^{n+m}\cU$ and each one of the quotient modules in $\big\{F^n\cU/F^{n-1}\cU\big\}_{\Sscript{n \geq 0}}$ is a finitely generated and projective right $A$-module. We will denote by  $\taun: F^n\cU \to \cU$ and by $\greekn{\tau}{n,\,m}{}:F^n\cU\to F^{m}\cU$ the canonical inclusions for $m\geq n\geq 0$. Notice that $\cU$ can be identified with the direct limit of the system $\{F^n\cU,\tau{\Sscript{n,m}}\}$, that is: $\cU=\injlimit{n}{F^n\cU}$. Thus, the underlying $A$-bimodule ${}_{\Sscript{A}}\cU_{\Sscript{A}}$ is an increasingly filtered $A$-bimodule which is locally finitely generated and projective as a right $A$-module by definition (see \S\ref{ssec:LFGr}).
The following Proposition summarizes the properties of rings with an admissible filtration.

\begin{proposition}\label{prop:FnUfgp}
Let $\cU$ be an $A$-ring endowed with a right admissible filtration $\left\{F^n\cU\mid n\geq 0\right\}$. The following properties hold true
\begin{enumerate}
\item Each of the structural maps $\tau_{\Sscript{n,n+1}}: F^n\cU \to F^{n+1}\cU$ is a split monomorphism of right $A$-modules. In particular, each of the submodules  $F^n\cU$ is a finitely generated and projective right $A$-module and each of the monomorphisms $\taun: F^n\cU \to \cU$ splits in right $A$-modules with retraction $\thetan$.
\item As a filtered right $A$-module, $\cU$ satisfies
$$
\cU \cong gr(\cU)\,=\, A\oplus \frac{F_1\cU}{A}\oplus \frac{F_2\cU}{F_1\cU}\oplus \cdots \oplus \frac{F_{n}\cU}{F_{n-1}\cU}\oplus \cdots
$$
In particular $\cU$ is a faithfully flat right $A$-module.
\end{enumerate}
\end{proposition}
\begin{proof}
Follows from  Lemma \ref{lemma:fgpquotients}, Corollary \ref{coro:FnL} and Remark \ref{rem:vartheta} of the Appendices. The faithfully flatness is a consequence of the fact that $\cU$ is the direct sum of the faithfully flat right $A$-module $A$ and the flat right $A$-module $\bigoplus_{n\in\N}F_{n+1}\cU/F_n\cU$ (see e.g. \cite[Proposition 9, I \S 3]{Bou:AC12}).
\end{proof}

\begin{remark}\label{rem:dualbasis}
Given a right admissible filtration $\left\{F^n\cU\mid n\geq 0\right\}$ on an $A$-ring $\cU$, it follows from Lemma \ref{lemma:fgpquotients} that we have right $A$-linear isomorphisms $\psi_n:F^n\cU\cong\bigoplus_{k=0}^n F^k\cU/F^{k-1}\cU$ for all $n\geq 0$. Let us fix a dual basis $\left\{\left(u_i^n,\gamma_i^n\right)\mid i=1,\ldots,d'_n\right\}$ for every $F^n\cU/F^{n-1}\cU$, $n\geq 0$. These induce a distinguished dual basis $\left\{\left(e_i^n,\lambda_i^n\right)\mid i=1,\ldots,d_n:=\sum_{j=0}^nd'_j\right\}$ on $F^n\cU$ for all $n\geq 0$, which is given as follows. The generating set $\left\{e_i^n\mid i=1,\ldots,d_n\right\}$ is given by $\left\{\psi_n^{-1}\left(u_i^k\right)\mid k=0,\ldots,n, i=1,\ldots,d'_k\right\}$, that is $\psi_n^{-1}\left(u_i^k\right)=e_{i+d_{k-1}}^n$ for all $0\leq k\leq n$ and all $1\leq i\leq d_k'$ ($d_{-1}=0$ by convention). The dual elements are given by extending $\gamma_i^k:F^k\cU/F^{k-1}\cU\to A$ to $\left(\gamma'\right)_i^k:\bigoplus_{k=0}^nF^k\cU/F^{k-1}\cU\to A$, letting $\left(\gamma'\right)_i^k\mid_{F^h\cU/F^{h-1}\cU}=0$ for $h\neq k$, and then composing with $\psi_n$, i.e., $\left(\gamma'\right)_i^k\circ\psi_n=\lambda^n_{i+d_{k-1}}$ for all $k,i$ as above.

This distinguished dual basis has the following useful property which will be helpful in the sequel. Let us denote by $\tau_{m,n}\colon F^{m}\cU\to F^n\cU$ and $j_{m,n}\colon \bigoplus_{h=0}^m F^h\cU/F^{h-1}\cU\to \bigoplus_{k=0}^n F^k\cU/F^{k-1}\cU$ the inclusion morphisms for $m\leq n$. Then $\psi_n\circ \tau_{m,n}=j_{m,n}\circ \psi_m$, whence $\tau_{m,n}\left(e^m_i\right)=e^n_i$ for all $i=1,\ldots,d_m$ and $\lambda^n_i\circ\tau_{m,n}=\lambda^m_i$ if $i=1,\ldots,d_m$, $\lambda^n_i\circ\tau_{m,n}=0$ otherwise.
\end{remark}

\subsection{A filtration on the convolution algebra $\cU^*$ of a co-commutative Hopf algebroid}\label{ssec:FUstra}

The idea for this example came to us from \cite[A.5.8]{Kapranov:2007}. Let $(A,\cU)$ be a co-commutative (right) Hopf algebroid.  We say that $(A,\cU)$ is  endowed with a \emph{(right) admissible filtration} if the $A$-ring $\cU$ has a right admissible filtration $\cU=\bigcup_{\Sscript{n \, \in \, \N }} F^{n}\cU$ as in \S\ref{ssec:AFiltration},  which is also compatible with the comultiplication in the sense that it satisfies
$$
\Delta(F^n\cU) \subseteq \sum_{p+q=n}\Img{F^p\cU_{\Sscript{A}} \tensor{A} F^q\cU_{\Sscript{A}}}=\sum_{p+q=n}F^p\cU_{\Sscript{A}} \tensor{A} F^q\cU_{\Sscript{A}}=F^n\left(\cU_{\Sscript{A}} \tensor{A} \cU_{\Sscript{A}}\right)
$$
(the counit is automatically filtered). The inclusion $\varfun{\tau_0}{A}{\cU}$ plays the role of the morphisms $s=t$.

\begin{example}\label{exm:ULAdmiss}
As in \S\ref{ssec:ULA}, consider the universal enveloping Hopf algebroid $\cU:=\cV_{\Sscript{A}}(L)$ of a given Lie-Rinehart algebra $(A,L)$. Since the tensor $A$-ring  $T_{\Sscript{A}}\left( A\otimes _{\K }L\right) $ is endowed with an increasing filtration:
\begin{equation*}
F^{n}\left( T_{\Sscript{A}}\left( A\tensor{\K }L\right) \right)
:=\bigoplus_{k=0}^{n}T_{\Sscript{A}}\left( A\tensor{\K }L\right) ^{k},
\end{equation*}
where $T_{\Sscript{A}}\left( A\tensor{\K }L\right) ^{k}:=\left( A\tensor{\K
}L\right) \tensor{A}\left( A\tensor{\K }L\right) \tensor{A}\cdots
\tensor{A}\left( A\tensor{\K }L\right)$ for $k$ times, this induces a filtration on $\cU$ via the canonical projection. Thus,  the $n$-th term of the filtration  $F^n\cU$ is the right $A$-module generated by the products of the images of elements of $L$ in  $\cU$ of length at most $n$.
If we assume as usual that $A$ is trivially filtered then both $\iota _{A}:A\rightarrow \mathcal{U}$ and $\varepsilon :\mathcal{U} \rightarrow A$
are filtered. Moreover, from
\begin{equation*}
\Delta \left( \iota _{L}\left( X\right) \right)  =\iota _{L}\left( X\right)
\tensor{A}1_{\mathcal{U} }+1_{\mathcal{U}
}\tensor{A}\iota _{L}\left( X\right)
\, \in \, \Img{F^{\Sscript{1}}\mathcal{U}
\tensor{A}F^{\Sscript{0}}\mathcal{U}} +\Img{F^{\Sscript{0}}\mathcal{U}
\tensor{A}F^{\Sscript{1}}\mathcal{U}} =F^{\Sscript{1}}\left( \mathcal{U} \tensor{A}\mathcal{U} \right)
\end{equation*}
it follows that
\begin{equation*}
\Delta \left( F^{n}\mathcal{U} \right) \subseteq
\sum_{k=0}^{n}\Delta \left( \iota _{L}\left( L\right) ^{k}\right) \subseteq
\sum_{k=0}^{n}\Delta \left( \iota _{L}\left( L\right) \right) ^{k}\subseteq
\sum_{k=0}^{n}F^{k}\left( \mathcal{U} \tensor{A}\mathcal{U} \right) \subseteq F^{n}\left( \mathcal{U}
\tensor{A}\mathcal{U} \right)
\end{equation*}
(notice that this makes sense since $\mathrm{Im}\left( \Delta \right)
\subseteq \left. \mathcal{U} \right. \times
_{A}\left. \mathcal{U} \right. $, which is a filtered algebra with filtration induced by the one of $\mathcal{U} \tensor{A}\mathcal{U}$). Summing up, $\mathcal{U} $ is a filtered co-commutative Hopf algebroid. Furthermore, it
is well-known that if $L$ is a projective right $A$-module, then we have a graded isomorphism of $A$-algebras $\mathrm{gr}\mathcal{U} \cong S_{\Sscript{A}}\left( L\right)$, the symmetric algebra of $L_{\Sscript{A}}$ (see e.g. \cite[Theorem 3.1]{Rin:DFOGCA}). From this, one deduces that if $L$ is also finitely generated as right $A$-module, then the quotient modules $F^{n}\mathcal{U} /F^{n-1}
\mathcal{U} $ are finitely generated and projective as right $
A$-modules as well. Therefore, under these additional hypotheses, $\mathcal{U}$ turns out to be a co-commutative right Hopf algebroid endowed with an admissible filtration.
\end{example}

We are going to show in the forthcoming subsections that the convolution algebra $\cU^*$ of a co-commutative right Hopf algebroid $(A,\cU)$ with an admissible filtration is a projective limit of algebras and a complete Hopf algebroid, where the comultiplication $\Delta_*: \cU^{*} \to \cU^{*}_{\Sscript{A}}~ \cmptens{A}~{^{}_{\Sscript{A}}\cU^*}$ is induced  by the multiplication of $\cU$, $A$ is trivially filtered and $\cU^{*}$ is considered as an $A$-bimodule via the source and the target induced by the algebra map $\vartheta: A\tensor{}A \to \cU^*$ of equation \eqref{Eq:vartheta}. The counit will be the map $\varfun{\varepsilon_*}{\cU^*}{A}$ such that $f \mapsto f(1)$. The construction of the antipode for $\cU^*$  will require an additional hypothesis and it will be treated separately in \S\ref{ssec:AUstra}.

Notice that $\cU^*\cong  \prlimit{n}{F^n\cU^*}$ as $A$-bimodules via the isomorphism
\begin{equation}\label{eq:iso*}
\sfun{\Phi}{\cU^*}{\prlimit{n}{(F^n\cU)^*}}{f}{\left(\tau_n^*(f)\right)_{n\geq 0}};\;\; \prlimit{n}{(F^n\cU)^*} \rightarrow \cU^*; \; \left( (g_n)_{\Sscript{n \geq 0}} \mapsto g:=\injlimit{n}{g_n} \right),
\end{equation}
and it can be endowed with a natural decreasing filtration
\begin{equation}\label{Eq:FUstar}
F_0\cU^*:=\cU^* \quad \text{ and } \quad  F_{n+1}\cU^*:=\ker{\taun^*}=\mathsf{Ann}(F^n\cU)
\end{equation}
(see \S\ref{ssec:DLF} for the general case). Besides, $\cU^*$ is also a projective limit of $(A\otimes_{}A)$-algebras, where the projective system is endowed with the algebra maps $\vartheta_n=(\tau_n)^* \circ \vartheta: A\otimes_{}A \to (F^n\cU)^*$ and where $(F^n\cU)^*$ is the convolution algebra of the $A$-coalgebra  $(F^n\cU)_{\Sscript{A}}$.

\begin{lemma}\label{lemma:U*complete}
The pair $\left(\cU^*,F_n\cU^*\right)$ gives a complete $A$-bimodule as well as a complete algebra.
\end{lemma}
\begin{proof}
In \S\ref{ssec:DLF} it is proved that it is a complete $A$-bimodule and so a complete module over $\K$ as well. In light of Remark \ref{rem:calg}, it is enough for us to prove that the filtration $F_n\cU^*$ is compatible with the convolution product to conclude the proof. Notice that the $\ann{F^n\cU}$'s are ideals, whence we have that $F_n\cU^**F_m\cU^*\subseteq F_{n+m}\cU^*$ whenever $n$ or $m$ is $0$. If $mn>0$ then, given $f\in F_n\cU^*$ and $g\in F_m\cU^*$, we have that $(f*g)\left(F^{m+n-1}\cU\right)\subseteq \sum_{p+q=n+m-1}f\left(F^p\cU\right)g\left(F^q\cU\right)=0$, because when $p\geq n$ it happens that $q=m+n-1-p\leq m-1$. Therefore, $f*g\in \ann{F^{n+m-1}\cU}=F_{n+m}\cU^*$ and hence $F_n\cU^**F_m\cU^*\subseteq F_{n+m}\cU^*$ for all $m,n\geq 0$. Recalling that we consider $A$ trivially filtered, we have that $F_n\cU^*$ induces on $\cU^*$ a filtration as a algebra and as an $A$-bimodule at the same time.
\end{proof}

\begin{remark}\label{rem:Kstar}
We already know that the convolution algebra $\cU^*$ is an augmented one and the augmentation is given by the algebra map (which is going to be the counit) $\varepsilon_*: \cU^* \to A$, $f \mapsto f(1)$. Therefore, one can consider the $\cI$-adic topology on $\cU^*$ with respect to the two-sided ideal $\cI:=\ker{\varepsilon_*}$.  However, in the filtration of Lemma \ref{lemma:U*complete} we have $\cI=F_1\cU^*$, and so $\cI^n \subseteq F_n\cU^*$, for every $n \geq 0$. Thus the $\cI$-adic topology is finer than the linear topology obtained from the filtration $\{F_n\cU^*\mid n \geq 0\}$ of Equation \eqref{Eq:FUstar}.
\end{remark}

\subsection{The topological comultiplication and counit of $\cU^*$}\label{ssec:CUUstar}
Next we want to show that the multiplication $\varfun{\mu}{\cU\otimes\cU}{\cU}$ induces a comultiplication $\varfun{\Delta_*}{\cU^*}{\cU^*\cmptens{A}\cU^*}$ which endows $\cU^*$ with a structure of comonoid   in the monoidal category of  complete bimodules $\left(\CBim{A}{A},\what{\otimes}_{\Sscript{A}},A,\alpha\right)$ where $A$ is trivially filtered. We will often make use of the notation and conventions of \S\ref{Not:infty} and Remark \ref{rem:limits}. The unaccustomed reader is invited to go through them before proceeding.
 
Let us perform the tensor product $\cU_{\Sscript{A}}\tensor{A}{_{\Sscript{A}}\cU}$. The multiplication $\mu:\cU\otimes \cU\to\cU$ which gives the algebra structure to $\cU$ factors through the tensor product over $A$:
\begin{equation*}
\mu\left(x\tau_0(a)\otimes y\right)=x\tau_0(a)y=\mu(x\otimes \tau_0(a)y),
\end{equation*}
and it is $A$-linear with respect to both $A$-actions on $ \cU_{\Sscript{A}}\tensor{A}{_{\Sscript{A}}\cU}$, namely
\begin{equation*}
a(x\tensor{A} y)=(\tau_0(a)x)\tensor{A} y \quad \text{and} \quad (x\tensor{A} y)a=x\tensor{A} (y\tau_0(a)).
\end{equation*}
Therefore, it induces a filtered $A$-bilinear morphism $\mu^*:\cU^*\to (\cU\tensor{A}\cU)^*$ and maps $\mu_{\Sscript{n,m}}:F^n\cU_{\Sscript{A}}\tensor{A}\, {{ }_{\Sscript{A}}F^m\cU}\to F^{n+m}\cU$, which dually give rise to a family of morphisms of $A$-bimodules
\begin{equation*}
\xymatrix{ \Delta_{n,m}:  \cU^*  \ar@{->}^-{\tau_{\Sscript{m+n}}^*}[r] & \left(F^{m+n}\cU\right)^* \ar@{->}^-{\mnm^*}[r] &  \left(F^n\cU_{\Sscript{A}}\tensor{A}{_{\Sscript{A}}F^m\cU}\right)^*, }
\end{equation*}
for every $n,m \geq 0$ such that $\Delta_{n,m}(f)\left(x\tensor{A} y\right)=f(xy)$ for all $f\in\cU^*$, $x\in F^n\cU$ and $y\in F^m\cU$. Given $f \in \cU^*$, for each element $u \in \cU$ we define $f\leftharpoonup u: \cU_{\Sscript{A}} \to A_{\Sscript{A}}$ to be the linear map which acts as $v \mapsto f(uv)$.

\begin{lemma}\label{lema:Deltamn}
For any $f \in \cU^*$ and for all $n,m\in\N$, we have 
\begin{equation}\label{eq:delta}
\big(\phi^{-1}_{m,n} \circ \Delta_{n,m}\big)(f)=\sum_{k=1}^{d_n}\tau_m^*\big(f\leftharpoonup \tau_n\big(e^n_k\big)\big)\tensor{A}\lambda^n_k  \, \in  \left(F^m\cU\right)^*{}_{\Sscript{A}}\tensor{A}{}_{\Sscript{A}}\left(F^n\cU\right)^*
\end{equation}
where  $\left\{e_k^n,\lambda_k^n\mid k=1,\cdots,d_n\right\}$ is the dual basis of $(F^n\cU)_{\Sscript{A}}$ given in Remark \ref{rem:dualbasis} and $\phi_{m,n}:\left(F^mM^*\right)^{}_{\Sscript{R}}\tensor{R} {^{}_{\Sscript{R}}\left(F^nM^*\right)} \cong \left(F^nM_{\Sscript{R}}\tensor{R} {{}_{\Sscript{R}}F^mM}\right)^*$ are the canonical isomorphisms (see Corollary \ref{coro:FnL}).
\end{lemma}

As a consequence of Lemma \ref{lema:Deltamn} and of the fact that $(\cU\tensor{A}\cU)^*\cong \cU^*\cmptens{A}\cU^*$ as filtered bimodules via the completion of canonical map $\phi_{\Sscript{\cU,\,\cU}}$ (see Proposition \ref{prop:MRN} and diagram \eqref{diag:Delta*} below), we have an $A$-bilinear comultiplication
\begin{equation*}
\varfun{\Delta_{*}:={\psiup_{\Sscript{\cU,\cU}} } \circ \mu^*}{\cU^*}{(\cU^*)_{\Sscript{A}} ~ \cmptens{A} ~ {_{\Sscript{A}}(\cU^*)}}
\end{equation*}
which makes the following diagram to commute
\begin{equation}\label{diag:Delta*}
\xymatrix @C=40pt @R=20pt{
\cU^* \ar@{-->}[rr]^-{\Delta_*} \ar[dr]_{\mu^*} \ar[dd]_{\tau_{m+n}^*} & & \cU^*\cmptens{A}\cU^* \ar[dd]^-{\Pi_{m,n}} \ar@<-2pt>[dl]_-{\what{\phi_{\cU,\cU}}}\\
 & (\cU\tensor{A}\cU)^* \ar@<-2pt>[ur]_-{\psiup_{\cU,\cU}} \ar[d]^-{(\tau_n\tensor{A}\tau_m)^*} & \\
F^{n+m}\cU^* \ar[r]^-{\mu_{n,m}^*} & (F^n\cU\tensor{A}F^m\cU)^*  & F^m\cU^*\tensor{A}F^n\cU^* \ar[l]_-{\phi_{m,n}}}
\end{equation}
for all $m,n\geq0$ where the projections $\Pi_{\Sscript{m,n}}$ make of $\cU^*\cmptens{A}\cU^* $ the projective limit of the projective system $\left(F^m\cU^*\tensor{A}F^n\cU^*, \tau_{\Sscript{p,\,m}}^{*}\tensor{A}\tau_{{\Sscript{q,\,n}}}^{*}\right)$, see \S\ref{ssec:TTAC}.
Thanks to relations \eqref{eq:zlim} of the Appendices and \eqref{eq:delta} and by resorting to the notations introduced in Remark \ref{rem:dualbasis}, one may write explicitly
\begin{equation}\label{eq:deltastar}
\Delta_{*}(f) \,=\,\limn\left(  \sum_{i=1}^{d_n} \big(f\leftharpoonup \tau_n\big(e_i^n\big)\big) \tensor{A}E_{\scriptstyle{\lambda_i^n}} \right),
\end{equation}
where we set $E_{\scriptstyle{\lambda_i^n}}:=\thetan^*(\lambda^n_i)$ (recall that for $g\in \cU^*$ we have that $g-\thetan^*\taun^*(g)\in\ker{\taun^*}=F_{n+1}\cU^*$).
Furthermore, the comultiplication $\Delta_{*}$ is uniquely determined by the following rule.  For every $f \in \cU^*$,
\begin{equation}\label{Eq:fuv}
\Delta_*(f) =\limn\left( \sum_{(f)} f_{\Sscript{(1), \,n}} \tensor{A} f_{\Sscript{(2),\, n}}\right) \, \Longleftrightarrow \left[ f(uv) \,=\, \limn \left(\sum_{(f)}  f_{\Sscript{(1),\, n}}\big( f_{\Sscript{(2),\, n}}(u) v \big)\right), \; \text{for every } u, v \in \cU \right].
\end{equation}

Now we can state the subsequent lemma.

\begin{lemma}\label{lemma:deltaepsialg}
The comultiplication $\varfun{\Delta_*}{\cU^*}{\cU^*\cmptens{A}\cU^*}$ is a morphism of filtered $A$-bimodules when $\cU^*\cmptens{A}\cU^*$ is endowed with its canonical filtration. Both the comultiplication $\varfun{\Delta_*}{\cU^*}{\cU^*\cmptens{A}\cU^*}$ and the counit $\varfun{\varepsilon_*}{\cU^*}{A}$ are morphisms of complete algebras.
\end{lemma}
\begin{proof}
By definition of $\Delta_*$, the first claim follows from the filtered isomorphism $(\cU\tensor{A}\cU)^*\cong \cU^*\cmptens{A}\cU^*$, once observed that the transpose of a filtered morphism of increasingly filtered modules is filtered with respect to the induced decreasing filtrations on the duals.
For the second claim, we only give the proof for the comultiplication, since the counit is clearly a morphism of complete algebras.
To show that $\Delta_*$ is unital, recall first that the unit of $\cU^*$ is the counit $\varepsilon=1_{\Sscript{\cU^*}}$ of $\cU$ and the unit of $\cU^*\tensor{A}\cU^*$ is $\varepsilon \tensor{A}\varepsilon=1_{\Sscript{\cU^*\tensor{A}\cU^*}}$, so $1_{\Sscript{\what{\cU^*\tensor{A}\cU^*}}}=\what{\varepsilon\tensor{A}\varepsilon}$ (the notation is that of Remark \ref{rem:limits}). Since
$$\what{\phi_{\Sscript{\cU,\cU}}}\left(\what{\varepsilon\tensor{A}\varepsilon}\right)(u\tensor{A}v) = \what{\phi_{\Sscript{\cU,\cU}}} \left(\gamma_{\Sscript{\cU^*\tensor{A}\cU^*}}(\varepsilon\tensor{A}\varepsilon)\right)(u\tensor{A}v) = \phi_{\Sscript{\cU,\cU}}(\varepsilon\tensor{A}\varepsilon)(u\tensor{A}v) = \varepsilon(\varepsilon(u)v)=\varepsilon(uv)$$
it follows from \eqref{Eq:fuv} that $\Delta_*(\varepsilon)=\what{\varepsilon\tensor{A}\varepsilon}$.

In view of Remark \ref{rem:counit}, to prove that $\Delta_{*}$ is multiplicative it is enough to show that $\Delta_*(f*g)=\boldsymbol{\mu}\left(\Delta_*(f)\cmptens{}\Delta_*(g)\right)$, for every $f,g \in \cU^*$, where $\boldsymbol{\mu}$ is the multiplication of the complete algebra $\cU^* \cmptens{A}\cU^*$ as in Remark \ref{rem:cmptensalgebra}. By employing the notation of \eqref{eq:deltastar}, we know that
\begin{gather}
\Delta_*(f*g)={\limn}\left(  \sum_{k=1}^{d_n} ((f*g)\leftharpoonup \tau_n\big(e_k^n\big)) \tensor{A} E_{\scriptstyle{\lambda_k^n}} \right), \label{Eq:fstarg}\\
\boldsymbol{\mu}\left(\Delta_*(f) \cmptens{} \Delta_*(g)\right) = {\limn}\left(  \sum_{i,\, j, \, =1}^{d_n, \,d_n} (f\leftharpoonup \tau_n\big(e_i^n\big))*(g\leftharpoonup \tau_n\big(e_j^n\big)) \tensor{A} (E_{\scriptstyle{\lambda_i^n}} *E_{\scriptstyle{\lambda_j^n}} ) \right).
\end{gather}
To show that the last two equations are equal, we need to show that the involved Cauchy sequences converge to the same limit.  In view of \eqref{diag:Delta*}, this amounts to show that $\Pi_{p,q}\left(\Delta_*(f*g)\right) = \Pi_{p,q}\left(\boldsymbol{\mu}\left(\Delta_*(f) \cmptens{} \Delta_*(g)\right) \right)$ for all $p,q\in\N$. Let  $x\in F^q\cU$, $y\in F^p\cU$ for some $p+q=n$
 and set $k=n+1=p+q+1$. We compute
\begin{align*}
& \phi_{p,q} \left( \Pi_{p,q} \Big( \boldsymbol{\mu}\left(\Delta_*(f)\cmptens{A}\Delta_*(g)\right)\Big)\right)(x\tensor{A}y) \overset{ \eqref{eq:phiPi}}{=} \sum_{i,\, j}^{d_{k},\, d_{k}}\left(\left(f\leftharpoonup \tau_{k}\big(e_j^{k}\big)\right)*\left(g\leftharpoonup \tau_{k}\big(e_j^{k}\big)\right)\right)\left( \left(E_{\Sscript{\lambda_i^k}}*E_{\Sscript{\lambda_j^k}}\right)(\tau_q(x)) \tau_p(y)\right) \\
 & \stackrel{(**)}{=}\sum_{ (x)}\sum_{i,\, j}^{d_{k},\, d_{k}}  \left(\left(f\leftharpoonup e_i^k\right)*\left(g\leftharpoonup e_j^k\right)\right)\left(E_{\Sscript{\lambda_i^k}}(x_1)E_{\Sscript{\lambda_j^k}}(x_2) y\right)\stackrel{(*)}{=} \sum_{(x),\, (y)}\sum_{i,\, j}^{d_k,\, d_k}  f\left(\big(e_i^k E_{\Sscript{\lambda_i^k}}(x_1)\big)E_{\Sscript{\lambda_j^k}}(x_2) y_1\right)g(e_j^ky_2) \\
&\stackrel{(\triangle)}{=}  \sum_{ (x),\, (y)}\sum_{j}^{d_k}  f\left(x_1E_{\Sscript{\lambda_j^k}}(x_2) y_1\right)g(e_j^ky_2)  \stackrel{(\blacktriangle)}{=} \sum_{(x),\, (y)} f(x_1y_1) g \left(  \sum_{j}^{d_k}   (e_j^kE_{\Sscript{\lambda_j^k}}(x_2))y_2\right)
 \\
 & \stackrel{(\triangle)}{=} \sum_{(x),\, (y)} f(x_1y_1)g(x_2y_2) \,=\, \sum_{(xy)} f((xy)_1)g((xy)_2) =   (f*g)(xy) \stackrel{\eqref{diag:Delta*}}{=} \phi_{p,q} \left( \Pi_{p,q}\left( \Delta_{*}(f*g) \right) \right)(x\tensor{A}y)
\end{align*}
where in $(*)$ we used the left $A$-linearity of $\Delta_{\cU}$ and in $(\blacktriangle)$ the fact that $\Img{\Delta_{\cU}}\subseteq \cU\times_{\Sscript{A}}\cU$. The equalities $(\triangle)$ follow from the fact that $\Delta_{\cU}$ is compatible with the filtration and from the observation after equation \eqref{eq:deltastar}. From $(**)$ up to the end of the computation, we omitted the inclusions $\tau_h$'s.  In conclusion, we have
$$
\Pi_{p,q} \Big( \boldsymbol{\mu}\left(\Delta_*(f)\cmptens{A}\Delta_*(g)\right)\Big)\,=\, \Pi_{p,q} \left(\Delta_{*}(f*g)\right), \quad \text{ for every } p,q \geq 0,
$$
whence $\Delta_*$ is multiplicative as well.
\end{proof}

\begin{proposition}\label{prop:coalgebra}
Let $\left(A,\cU\right)$ be a co-commutative Hopf algebroid with $\cU$ endowed with an admissible filtration $\left\{F^n\cU\mid n\geq 0\right\}$. Then $\left(\cU^*,\Delta_*,\varepsilon_*\right)$ is a coalgebra in the monoidal category $\left(\CBim{A}{A},\what{\otimes}_A,A,\alpha\right)$ of complete $A$-bimodules, where $\cU^*\cmptens{A}\cU^*$ is filtered with its canonical filtration.
\end{proposition}
\begin{proof}
We already know from Lemma \ref{lemma:U*complete} that $\cU^*\cong \prlimit{n}{\cU^*/\mathsf{Ann}\left(F^n\cU\right)}$  is a complete $A$-bimodule. The map $\Delta_*$ is $A$-bilinear by construction and we know from Lemma \ref{lemma:deltaepsialg}  that it is also filtered. The $A$-bilinear map $\varfun{\varepsilon_*}{\cU^*}{A}$ is filtered as well, in view of Lemma \ref{lemma:deltaepsialg}.

Let us prove then that $\Delta_*$ is coassociative and counital, with counit $\varepsilon_*$. Let us begin with counitality. Since $\varepsilon_*$ is filtered, $\varepsilon_*\otimes_A\cU^*$ is filtered and hence we have $\varfun{\varepsilon_*\cmptens{A}\cU^*}{\cU^*\cmptens{A}\cU^*}{\cU^*}$ which acts as
$$
\left(\varepsilon_*\cmptens{A}\cU^*\right) \left( \limn\left( \sum_{i=1}^{r_n}f_i^{(n)}\tensor{A} g_i^{(n)} \right) \right) \,=\, \limn\left(\sum_{i=1}^{r_n} f_i^{(n)}\left(1_{\Sscript{\cU}}\right) \cdot g_i^{(n)} \right).
$$
Applying this formula to $\Delta_{*}(f)$ for any $f \in \cU^*$, we get
\begin{equation*}
\left(\varepsilon_*\cmptens{A}\cU^*\right) \left(\Delta_{*}(f)\right) = \limn\left(\sum_{i=1}^{d_n} f(\tau_n\big(e_i^n\big))\cdot E_{\Sscript{\lambda_i^{n}}} \right) = \limn\left(\sum_{i=1}^{d_n} f(\tau_n\big(e_i^n\big))\cdot \lambda_i^{n}\theta_n \right) = \limn\left(f\tau_n\theta_n\right) = f,
\end{equation*}
since $\{e_i^n,\lambda_i^n\mid i=1\cdots d_n\}$ is the dual basis of $F^n\cU$ given in Remark \ref{rem:dualbasis} and  $f\tau_n\theta_n-f\in \mathsf{Ann}(F^{n-1}\cU)=F_n\cU^*$. This shows that $\left(\varepsilon_*\cmptens{A}\cU^*\right)\circ \Delta_*=\id_{\cU^*}$. Analogously, we obtain $\left(\cU^*\cmptens{A}\varepsilon_*\right)\circ \Delta_*=\id_{\cU^*}$.

Finally, we have to check the coassociativity of the comultiplication.
To this end we will use the characterization of $\Delta_{*}$ given in \eqref{Eq:fuv}.  For a given $f \in \cU^*$, we have
\begin{equation*}
\begin{split}
(\Delta_{*}\cmptens{A}\cU^*) \left(\Delta_{*}(f)\right)\,=\,  \limn\sum_{(f)}\left( \left(\limk \sum_{\left(f_{(1)}\right)} (f_{\Sscript{(11),\,n,\,k}} \tensor{A}f_{\Sscript{(12),\,n,\,k}}) \right) \tensor{A} f_{\Sscript{(2),n}} \right), \\
(\cU^*\cmptens{A}\Delta_{*}) \left( \Delta_{*}(f)\right)\,=\,  \limn\sum_{(f)}\left(  f_{\Sscript{(1),\,n}} \tensor{A}  \left(\limk\sum_{\left(f_{(2)}\right)}(f_{\Sscript{(21),\,n,\,k}} \tensor{A}f_{\Sscript{(22),\,n,\,k}}) \right)  \right) .
\end{split}
\end{equation*}
For simplicity we will drop the sum $\sum_{(f)}$ accompanying the algebraic tensor product as all the involved topologies are linear. In light of \eqref{eq:defalpha}, the coassociativity of $\Delta_{*}$ will follow once it will be shown that
\begin{equation}\label{eq:coass}
\limn\left( f_{\Sscript{(11),\,n,\,n}} \tensor{A}f_{\Sscript{(12),\,n,\,n}} \tensor{A} f_{\Sscript{(2),\,n}}  \right) = \limn\left( f_{\Sscript{(1),n}} \tensor{A}  f_{\Sscript{(21),\,n,\,n}} \tensor{A}f_{\Sscript{(22),\,n,\,n}} \right)
\end{equation}
as a limit point in the complete space $\what{\cU^*\tensor{A}\cU^*\tensor{A}\cU^*}$ (the completion of  $\cU^*\tensor{A}\cU^*\tensor{A}\cU^*$) .   By omitting both the associativity constraint $a$ and its transpose $a^*$ observe that, for all given $u,v, w \in \cU$, we have
\begin{equation*}
\begin{split}
\what{\phi_{\Sscript{\cU\tensor{A}\cU, \cU}}}  \left( \what{\phi_{\Sscript{\cU,\cU}} \tensor{A} \cU^*} \right) \left( \limn\left( f_{\Sscript{(11),n,n}} \tensor{A}f_{\Sscript{(12),n,n}} \tensor{A} f_{\Sscript{(2),n}}  \right)  \right)  (u\tensor{A}v\tensor{A}w)  =   \limn f_{\Sscript{(11),n,n}}\Big(f_{\Sscript{(12),n,n}}(f_{\Sscript{(2),n}}(u)v)  w\Big) \overset{\eqref{Eq:fuv}}{=} f(u(vw)), \\
\what{\phi_{\Sscript{\cU, \cU\tensor{A}\cU}}} \left( \what{\cU^* \tensor{A} \phi_{\Sscript{\cU,\cU}}} \right) \left( \limn\left( f_{\Sscript{(1),n}} \tensor{A}  f_{\Sscript{(21),n,n}} \tensor{A}f_{\Sscript{(22),n,n}} \right) \right) (u\tensor{A}v\tensor{A}w) = \limn f_{\Sscript{(1),n}}\Big( f_{\Sscript{(21),n,n}}(f_{\Sscript{(22),n,n}}(u)v)w \Big)  \overset{\eqref{Eq:fuv}}{=} f((uv)w) .
\end{split}
\end{equation*}
Comparing this last equations leads to equality \eqref{eq:coass} and then to the coassociativity of $\Delta_{*}$.
\end{proof}

\subsection{The topological antipode of $\cU^*$}\label{ssec:AUstra}
Now we proceed to construct the topological antipode for $\cU^*$, under the further hypothesis that the translation map of $\cU$ is a  filtered morphism of algebras. Such an assumption is always  fulfilled in the case of the universal enveloping Hopf algebroid of a Lie-Rinehart algebra with finitely generated and projective module $L_{\Sscript{A}}$, as we will see in Example \ref{exam:SUL}.

At the level of the algebra structure, the antipode is provided by the following  map (compare with \cite[Theorem 3.1]{KowalzigExtBVA} and \cite[Theorem 5.1.1]{CGK:2016} for the case when $\cU$ is finitely generated and projective right $A$-module):
\begin{equation}\label{Eq:tantip}
\cS_{*}: \cU^* \longrightarrow \cU^*,\quad \Big(  f \longmapsto \left[  u \mapsto \varepsilon\big( f(u_{\Sscript{-}}) u_{\Sscript{+}}\big) \right] \Big),
\end{equation}
where  $u \mapsto \beta^{-1}(1\tensor{A}u) = u_{\Sscript{-}} \tensor{A} u_{\Sscript{+}} $ is the translation map obtained from the inverse of the map $\beta: \cU_{\Sscript{A}} \tensor{A} {\,}_{\Sscript{A}}\cU \to \cU_{\Sscript{A}}\tensor{A} \cU_{\Sscript{A}}$, which sends  $u\tensor{A}v \mapsto uv_{\Sscript{1}}\tensor{A}v_{\Sscript{2}}$. 
Consider $\delta: \cU \to \cU \tensor{A} \cU:= \cU_{\Sscript{A}} \tensor{A} {}_{\Sscript{A}}\cU$, $ u \mapsto  u_{\Sscript{-}} \tensor{A} u_{\Sscript{+}}$ the map of Equation \eqref{eq:amore}. As it was shown in \cite[Proposition 3.7]{Schau:DADOQGHA}, the map $\delta$ enjoys a series of properties. Here we recall few of them  which will be needed in the sequel.  First notice that $\beta^{-1} (v\tensor{A} u)=v\umin\tensor{A}\umas$, so for all $u,v\in\cU$ and $a\in A$ we have
\begin{eqnarray}
1\tensor{A} u & =& \umin u_{\Sscript{+,\,1}} \tensor{A} u_{\Sscript{+,\,2}}  \,\, \in \, \cU_{\Sscript{A}} \, \tensor{A} \, \cU_{\Sscript{A}} \label{Eq:B4}   \\ u_{\Sscript{1,\, -}} \tensor{A} u_{\Sscript{1,\,+}}\tensor{A} u_{\Sscript{2}} &=& \umin\tensor{A} u_{\Sscript{+,\, 1}}\tensor{A}u_{\Sscript{+,\, 2}} \,\, \in \, (\cU_{\Sscript{A}} \,\tensor{A} \, {}_{\Sscript{A}}\cU)\, \tensor{A} \, \cU_{\Sscript{A}} \label{Eq:B5}   \\
 u_{\Sscript{+,\,-}} \tensor{A} u_{\Sscript{-}} \tensor{A} u_{\Sscript{+,\,+}}  &=& u_{\Sscript{-,\,1}} \tensor{A} u_{\Sscript{-,\,2}} \tensor{A} u_{\Sscript{+}} \,\, \in \, \cU_{\Sscript{A}} \,\tensor{A} \, \cU_{\Sscript{A}}\, \tensor{A} \, \cU_{\Sscript{A}}   \label{Eq:B55}  \\
\umin \, \umas &=& \tau_{\Sscript{0}}(\varepsilon(u) ) \,\,  \in F^0\cU = A  \label{Eq:B555}
\\
(uv)_{\Sscript{-}} \tensor{A} (uv)_{\Sscript{+}} & =& v_{\Sscript{-}} u_{\Sscript{-}} \tensor{A} u_{\Sscript{+}}v_{\Sscript{+}}   \,\, \in \cU_{\Sscript{A}} \, \tensor{A} \, {}_{\Sscript{A}}\cU \label{Eq:B66}
\\
a\tensor{A} 1\,\,=\,\, 1\tensor{A} a &=& a_{\Sscript{-}} \tensor{A} a_{\Sscript{+}} \,\, \in \,  \cU_{\Sscript{A}} \,\tensor{A} \, {}_{\Sscript{A}}\cU. \label{Eq:B6}
\end{eqnarray}
In particular, by equation \eqref{Eq:B66} we have that $\delta$ is an algebra map, viewed as a map from $\cU$ to $\cU^{\Sscript{\mathrm{op}}} \times_{\Sscript{A}} \cU$.
The subsequent lemma is crucial for showing that $\cS_{*}$ is multiplicative.

\begin{lemma}\label{lema:Sstar}
Let $f,g, h \in \cU^*$ and $u \in \cU$. Then we have
\begin{equation}
\begin{split}
\cS_{*}(f*g) (u) \,=\, \cS_{*}(f)\Big( g(u_{\Sscript{-}})u_{\Sscript{+}} \Big) \hspace{2cm} \label{Eq:S1} \\
\big(\cS_{*}(f) * h\big) (u) \,=\,  \big(h\leftharpoonup f(u_{\Sscript{-}})\big) (u_{\Sscript{+}}) \,=\, \Big(\big( \varepsilon\leftharpoonup f(u_{\Sscript{-}}) \big) * h\Big) (u_{\Sscript{+}}) 
\end{split}
\end{equation}
\end{lemma}
\begin{proof}
We will implicitly use the co-commutativity of the comultiplication of $\cU$ as well as the $A$-linearity of $\delta$. Computing the left hand side of the first equality gives
\begin{equation*}
\cS_{*}(f*g) (u) = \varepsilon\Big( (f*g)(u_{\Sscript{-}})\, u_{\Sscript{+}} \Big) =  \varepsilon\Big( f(u_{\Sscript{-,1}})\,g(u_{\Sscript{-,2}})\, u_{\Sscript{+}} \Big) \overset{\eqref{Eq:B55}}{=}  \varepsilon\Big( f(u_{\Sscript{+,-}})\,(g(u_{\Sscript{-}})\, u_{\Sscript{+,+}}) \Big) = \cS_{*}(f)\Big( g(u_{\Sscript{-}})u_{\Sscript{+}} \Big),
\end{equation*}
where in the last equality we used \eqref{Eq:tantip} and \eqref{Eq:B6}. This leads to the stated first equality.

As for the second one, we have
\begin{equation*}
(\cS_{*}(f) *h)(u) = \cS_{*}(f)(u_{\Sscript{1}})  h(u_{\Sscript{2}}) =   \varepsilon\Big( f(u_{\Sscript{1,-}})\,u_{\Sscript{1,+}}\Big)\, h(u_{\Sscript{2}})  \overset{\eqref{Eq:B5}}{=}
 \varepsilon\Big( f(u_{\Sscript{-}})\,u_{\Sscript{+,1}}\Big)\, h(u_{\Sscript{+,2}})
\end{equation*}
from which one deduces on the one hand that $(\cS_{*}(f) *h)(u) =  \varepsilon\Big( f(u_{\Sscript{-}})\,u_{\Sscript{+,1}}\Big)\, h(u_{\Sscript{+,2}}) = \Big( (\varepsilon \leftharpoonup f(u_{\Sscript{-}})) * h\Big) (u_{\Sscript{+}})$ and on the other hand that $(\cS_{*}(f) *h)(u) =  \varepsilon\Big( f(u_{\Sscript{-}})\,u_{\Sscript{+,1}}\Big)\, h(u_{\Sscript{+,2}}) = \varepsilon( u_{\Sscript{+,1}})\, h(f(u_{\Sscript{-}})\,u_{\Sscript{+,2}})=h(f(u_{\Sscript{-}})\,u_{\Sscript{+}})$.
\end{proof}

The map $\beta$ is compatible with the increasing filtration on both  $\cU_{\Sscript{A}} \tensor{A} {\,}_{\Sscript{A}}\cU$ and $\cU_{\Sscript{A}}\tensor{A} \cU_{\Sscript{A}}$.
Namely,
$$
\cF^{\Sscript{n}}\big( \tuau \big)\,=\, \sum_{p+q=n} \tfuafu{p}{q} \subseteq \tuau \quad \text{and} \quad \cF^{\Sscript{n}}\left(\cU_{\Sscript{A}} \tensor{A} {\,}_{\Sscript{A}}\cU\right) = \sum_{p+q=n}\Img{F^{\Sscript{p}}\cU_{\Sscript{A}}\tensor{A}{_{\Sscript{A}}F^{\Sscript{q}}\cU}}.
$$
The left-most inclusion is clear since the $\taun : F^n\cU_{\Sscript{A}} \to \cU_{\Sscript{A}}$'s are split monomorphisms of right $A$-modules. The structure of $A$-bimodule on $\tuau$ is given by $a(u\tensor{A}v)a'=au\tensor{A}a'v$. Applying $\beta$ to each term of the canonical filtration of $\ltuau$, we have
\begin{equation*}
\begin{split}
\beta\Big( \cF^n\big( \ltuau\big) \Big) & = \sum_{p+q=n}F^p\cU\cdot \Delta(F^q\cU) \subseteq \sum_{p+l+k=n}F^p\cU \cdot F^k\cU_{\Sscript{A}} \tensor{A} F^l\cU_{\Sscript{A}} \subseteq  \sum_{p+l+k=n}\tfuafu{p+k}{l} \\ & \subseteq \sum_{i+j=n} \tfuafu{i}{j} = \cF^n(\tuau).
\end{split}
\end{equation*}
This means that $\beta$ is a filtered morphism of $A$-bimodules. On the other hand, $\beta^{-1}$ is a filtered bilinear map if and only if $\delta$ is a filtered algebra map. We point out that, in general, none of the equivalent conditions may be true, although this is the case for universal enveloping Hopf algebroids, as the next example shows.

\begin{example}\label{exam:SUL}
Take $(A,L)$ and $\cU=\cV_{\Sscript{A}}(L)$ as in Example \ref{exm:ULAdmiss}. Then, the following computation
\begin{equation*}
\delta \left( F^{n}\mathcal{U} \right) \,\subseteq\,
\sum_{k=0}^{n}\delta \left( \iota _{L}\left( L\right) ^{k}\right) \,\subseteq\,
\sum_{k=0}^{n}\delta \left( \iota _{L}\left( L\right) \right) ^{k}  \,\subseteq\,
\sum_{k=0}^{n}F^{k}\left( \mathcal{U} \tensor{A}\mathcal{U}
 \right) \subseteq F^{n}\left( \mathcal{U}
\tensor{A}\mathcal{U} \right)
\end{equation*}
shows that $\delta$ is a filtered algebra map, which
implies that $\beta ^{-1}$ is also filtered, as we have
\begin{eqnarray*}
\beta ^{-1}\left( F^{n}\left( \mathcal{U} \tensor{A}
\mathcal{U} \right) \right)  & \subseteq &\sum_{p+q=n}\beta
^{-1}\left( F^{p}\mathcal{U} \tensor{A}F^{q}\mathcal{U}
 \right) \subseteq \sum_{p+q=n}F^{p}\mathcal{U} \centerdot \delta\left( F^{q}\mathcal{U}\right)  \\
&\subseteq &\sum_{p+k+h=n}F^{p}\mathcal{U} \cdot  \Img{F^{k}
\mathcal{U} \tensor{A}F^{h}\mathcal{U} }
\subseteq F^{n}\left( \mathcal{U} \tensor{A}\mathcal{U}
 \right).
\end{eqnarray*}
\end{example}

\begin{remark}\label{rem:delta}
In the case of the universal enveloping Hopf algebroids of Lie-Rinehart algebras, as we have seen in Example \ref{exam:SUL}, the translation map is always continuous and then the associated map $\cS_{*}$ of equation \eqref{Eq:tantip} is  a filtered map. Indeed, if we assume  that $\delta$ is filtered, then
$$
\taun(u)_{\Sscript{-}}\tensor{A} \taun(u)_{\Sscript{+}} \in \sum_{p+q=n} \mathsf{Im}\big( \taup\tensor{A}\tauq\big),
$$ 
for every $u \in F^n\cU$ and $n \geq 0$. Therefore, $\cS_{*}(F_n\cU^*) \subseteq F_n\cU^*$, for every $n \geq 0$.
\end{remark}

\begin{proposition}\label{prop:tantip}
Let $(A,\cU)$ be a co-commutative (right) Hopf algebroid endowed with an admissible filtration and assume $\delta$ is a filtered algebra map. Then the map $\cS_{*}$ of equation \eqref{Eq:tantip} is a morphism of complete algebras such that $\cS_*\circ s_* = t_*$ and $\cS_*\circ t_* = s_*$.
\end{proposition}
\begin{proof}
We need to check that $\cS_{*}$ is multiplicative and that it exchanges the source with the target, as we already know that it preserves the filtration, in view of Remark \ref{rem:delta}. Recall that the unit of $\cU^*$ is given by $\vartheta: A\tensor{}A \to \cU^*$ sending $a\tensor{}a' \mapsto \left[ u \mapsto a\varepsilon(a'u)\right]$. Given $a, a' \in A$ we have
\begin{equation*}
\begin{split}
\cS_{*}(\vartheta(a\tensor{}a'))(u) & = \varepsilon\Big( \vartheta(a\tensor{}a')(u_{\Sscript{-}}) \umas \Big) =  \varepsilon\Big( \varepsilon(a'u_{\Sscript{-}})a \umas \Big) =  \varepsilon\Big( a'u_{\Sscript{-}}a \umas \Big) =   \varepsilon\Big( u_{\Sscript{-}}a \umas a' \Big) \\
 & = \varepsilon\Big( u_{\Sscript{-}}a \umas \Big)a'  = \varepsilon\Big( (au)_{\Sscript{-}} (au)_{\Sscript{+}} \Big)a'  \overset{\eqref{Eq:B555}}{=}   \varepsilon(au)a',
\end{split}
\end{equation*}
whence $\cS_{*} (\vartheta (a\tensor{}a')) = \vartheta(a'\tensor{}a)$. Therefore, $\cS_{*} \circ s_* = t_*$ and $\cS_{*} \circ t_* = s_*$, where $s_*, t_*$ are as in equation \eqref{Eq:vartheta}. Let us check that $\cS_{*}$ is multiplicative. If we consider $f,g  \in \cU^*$ and  $u \in\cU$, we have
\begin{equation*}
\Big(\cS_{*}(f) *\cS_{*}(g)\Big)(u)  \overset{\eqref{Eq:S1}}{=} \Big(\cS_{*}(g) \leftharpoonup f(u_{\Sscript{-}})\Big) (u_{\Sscript{+}}) =   \cS_{*}(g)  \Big(f(u_{\Sscript{-}}) \, u_{\Sscript{+}} \Big)
\overset{\eqref{Eq:S1}}{=}  \cS_{*}(g*f) (u)= \cS_{*}(f*g)(u).
\end{equation*}
This shows that $\cS_{*}(f*g)=\cS_{*}(f) * \cS_{*}(g)$, which finishes the proof.
\end{proof}

The following is the main result of this subsection.

\begin{proposition}\label{prop:Bosco}
Let $(A,\cU)$ and $\delta$ be as in Proposition \ref{prop:tantip}. Then $(A,\cU^*)$ is a complete Hopf algebroid with structure maps $s_*$, $t_*$, $\Delta_*$, $\varepsilon_*$ and $\cS_*$. In particular, this is the case for the universal enveloping Hopf algebroid $\cU=\cV_{\Sscript{A}}(L)$ of a Lie-Rinehart algebra $(A,L)$, where $L_{\Sscript{A}}$ is finitely generated and projective.
\end{proposition}
\begin{proof}
We only need to check that the algebra map $\cS_{*}$ enjoys the properties for being an algebraic antipode since it is already continuous. Let $f \in \cU^*$ and take an arbitrary element $u \in \cU$. Then
\begin{equation*}
\begin{split}
{\limn}\sum_{(f)} \Big( f_{\Sscript{1,\,n}}*\cS_*(f_{\Sscript{2,\,n}})\Big) (u) &  \,=\,  {\limn}\sum_{(f)}  \Big( \cS_*(f_{\Sscript{2,n}}) *f_{\Sscript{1,\,n}}\Big) (u)  \,\overset{\eqref{Eq:S1}}{=}\,  {\limn} \sum_{(f)} f_{\Sscript{1,\,n}} \Big(f_{\Sscript{2,\,n}}(\umin)\,\umas \Big) \\
& \overset{\eqref{Eq:fuv}}{=}\,  f(\umin\umas) \,=\, f(\varepsilon(u))\,=\, f(1)\varepsilon(u)\,=\, s_{*}(\varepsilon_{*}(f))(u).
\end{split}
\end{equation*}
Therefore,  for every $f \in \cU^*$, we have
$$
{\limn}\sum_{(f)} \Big(  f_{\Sscript{1,\,n}} * \cS_{*}(f_{\Sscript{2,\,n}} )\Big) \,=\, s_{*}(\varepsilon_{*}(f)).
$$
Now let us check that $\cS_*^2=id_{\Sscript{\cU^*}}$ which will be sufficient to claim that $\cS_*$ is an antipode for the complete bialgebroid $\cU^*$. Recall that $\delta: \cU_{\Sscript{A}} \to \cU_{\Sscript{A}} \tensor{A} { }_{\Sscript{A}}\cU$ is right $A$-linear, so that we can consider the map
$$
(\delta\tensor{A}{}_{\Sscript{A}}\cU) \circ \delta: \cU_{\Sscript{A}} \longrightarrow \cU_{\Sscript{A}} \tensor{A} {}_{\Sscript{A}}\cU_{\Sscript{A}} \tensor{A} \cU_{\Sscript{A}},\; \; \Big( u \longmapsto u_{\Sscript{-,\,-}}\tensor{A} u_{\Sscript{-,\,+}}\tensor{A} u_{\Sscript{+}}\Big)
$$
Let us compute the image of the element $u_{\Sscript{-,\,-}}\tensor{A} u_{\Sscript{-,\,+}} u_{\Sscript{+}} \in \cU_{\Sscript{A}} \tensor{A} {}_{\Sscript{A}}\cU$ by the map $\beta$:
\begin{eqnarray*}
\beta\big( u_{\Sscript{-,\,-}}\tensor{A} u_{\Sscript{-,\,+}} u_{\Sscript{+}}  \big) & =&  \big( u_{\Sscript{-,-}}\tensor{A} 1\big) . (u_{\Sscript{-,\,+}} u_{\Sscript{+}})   \quad \in\;  \cU_{\Sscript{A}} \tensor{} \cU_{\Sscript{A}}\\
&=& \Big( u_{\Sscript{-,\,-}} u_{\Sscript{-,\,+,\,1}}\tensor{A} u_{\Sscript{-,\,+,\,2}}  \Big). \umas  \\
&\overset{\eqref{Eq:B4} }{=}& \Big( 1\tensor{A} \umin\Big) . \umas  \,\,=\,\,  u_{\Sscript{+,\,1}}\tensor{A}\umin u_{\Sscript{+,2}} \\ &=&  u_{\Sscript{+,\,2}}\tensor{A}\umin u_{\Sscript{+,\,1}} \,\, \overset{\eqref{Eq:B4} }{=}\,\, u\tensor{A} 1\,\,=\,\, \beta(u\tensor{A} 1)
\end{eqnarray*}
Therefore, for every $u \in \cU$,  we have
\begin{equation}\label{Eq:uminus}
u_{\Sscript{-,\,-}}\tensor{A} u_{\Sscript{-,\,+}} \, u_{\Sscript{+}} \,\,=\,\,  u\tensor{A} 1 \, \in \, \cU_{\Sscript{A}} \tensor{A} \, {}_{\Sscript{A}} \cU.
\end{equation}
In this way, taking a function $f \in \cU^*$ and an element $u \in \cU$, we then get
\begin{equation*}
\cS_*^2(f)(u) =\cS_*\big( \cS_*(f)  \big) (u) = \varepsilon\Big( \cS_*(f)(\umin) \, \umas  \Big) = \varepsilon\Big( f(u_{\Sscript{-,\,-}}) \, u_{\Sscript{-,\,+}} \, \umas  \Big) \overset{\eqref{Eq:uminus}}{=} \varepsilon\big( f(u) \, 1\big) = f(u),
\end{equation*}
whence $\cS_*^2=id_{\Sscript{\cU^*}}$ and this finishes the proof of the fact that $(\cU^*, \Delta_{*}, \varepsilon_{*}, \cS_{*})$ is a complete Hopf algebroid. The particular case follows immediately from Remark \ref{rem:delta}.
\end{proof}

\begin{remark}\label{rem:deltaii}
As we have seen in this subsection, the fact that $\delta$ is a continuous map seems to be essential in carrying out the construction of the topological antipode for the convolution algebra $\cU^*$ of an admissible co-commutative (right) Hopf algebroid $\cU$. However, notice that no restrictive condition was imposed on the $A$-module $L$ in order to obtain the continuity of the translation map $\delta$ for the Hopf algebroid $\cU=\cV_{\Sscript{A}}(L)$. In particular, the fact that the filtration on $\cU$ is admissible is not used.

Besides, both conditions (i.e., admissibility of the filtration of $\cU$ and continuity of $\delta$) are satisfied for any universal enveloping algebra of a finitely generated and projective Lie-Rinehart algebra, that is to say, our results apply to the classical geometric context of Lie algebroids that was of our interest. 

Summing up, we have been able to show that the convolution algebra $\cU^*$ of a filtered co-commutative (right) Hopf algebroid $\cU$ is a complete Hopf algebroid if $\cU$ is admissibly filtered and $\delta$ is continuous. A question which we consider worthy to be addressed is if the converse is true as well. If this is not the case, we would be very glad to see a counterexample. In this sense, providing examples of admissibly filtered co-commutative Hopf algebroid whose translation map is not filtered could be of great interest. Concluding, we think that both questions deserve further attention but also that they are out of the purposes of the present paper. Therefore, we will not go into details here and we will leave them for future investigation. 
\end{remark}

\section{The main morphism of complete commutative Hopf algebroids}\label{sec:MCH}
In this section we give our main result. It is concerned with the universal (right) Hopf algebroid of a  Lie-Rinehart algebra with an admissible filtration, as in \S\ref{ssec:FUstra}. This in particular encompasses the situation of a Lie algebroid over a smooth (connected) manifold, by using the global sections functor as in \S\ref{ssec:muchocaldo}.

\subsection{The finite dual of a co-commutative Hopf algebroid.}\label{ssec:Fdual}
We recall from \cite{LaiachiGomez} the construction of what is known as \emph{the finite dual Hopf algebroid}. This construction is one of the main tools used in building up our application in the forthcoming subsections, so it is convenient to recall it in some detail.

Following \cite{LaiachiGomez}, given a (right) co-commutative Hopf algebroid $(A,\cU)$, we consider the category $\cat{A}_{\Sscript{\cU}}$ of those right $\cU$-modules whose underlying right $A$-module structure is finitely generated and projective. This category is a symmetric rigid monoidal linear category with identity object $A$, whose right $\cU$-action is given by $a \centerdot  u =\varepsilon(au)$. Furthermore, the forgetful functor  $\boldsymbol{\omegaup}: \cat{A}_{\Sscript{\cU}} \to \proj{A}$ to the category of finitely generated and projective $A$-modules,  plays the role of \emph{the fibre functor} which is a non trivial symmetric strict monoidal faithful functor (as we are assuming that ${\rm Spec}(A) \neq \emptyset$).

The tensor product of two right $\cU$-modules $M$ and $N$ is the $A$-module $M\tensor{A}N$ endowed with the following right $\cU$-action:
$$
(m\tensor{A}n) \centerdot u = (m \centerdot u_{\Sscript{1}}) \tensor{A} (n \centerdot u_{\Sscript{2}}).
$$
The dual object of a right $\cU$-module $M$ belonging to $\cat{A}_{\Sscript{\cU}}$ is the $A$-module $M^{*}=\hom{-A}{M}{A}$ with the right $\cU$-action
\begin{equation}\label{Eq:star}
\varphi \centerdot u : M \longrightarrow A, \quad \Big( m \longmapsto \varphi(m\centerdot u_{\Sscript{-}}) \centerdot u_{\Sscript{+}}  \Big),
\end{equation}
where $u_{\Sscript{-}}\tensor{A}u_{\Sscript{+}} = \beta^{-1}(1\tensor{A}u)$. We will often omit the symbol $\centerdot$ in what follows, when the action will be clear from the context.

The commutative Hopf algebroid constructed from the data $\left(\cat{A}_{\Sscript{\cU}}, \boldsymbol{\omegaup}\right)$, will be denoted by $(A,\cU^{\bcirc})$ and refereed to as \emph{the finite dual Hopf algebroid} of $(A,\cU)$.
From its own definition, $\cU^{\bcirc}$ is the quotient algebra
\begin{equation}\label{Eq:Uo}
\cU^{\bcirc}=\frac{\underset{M\,\in\, \mathrm{Ob}
( \cat{A}_{\Sscript{\cU}}) }{\bigoplus} M^{\ast }\tensor{T_{M}}M
}{\mathfrak{J}_{\Sscript{\cat{A}_{\cU}}}}
\end{equation}
by the two sided ideal $\mathfrak{J}_{\Sscript{\cat{A}_{\cU}}}$ generated by  the set
\begin{equation*}
\Big\{ \big(\varphi \tensor{T_{N}}f\left( m\right)\big) -\big(\varphi \circ f\tensor{T_{M}}m\big) \mid \varphi \in \left. N^{\ast }\right. ,m\in M,f\in T_{MN},M,N\in \mathrm{Ob}\left( {\cat{A}_{\cU}} \right) \Big\},
\end{equation*}
where we denoted by $M$ and $\left. M^{\ast }\right. $ the objects $\boldsymbol{\omegaup}
\left( M\right) $ and $\left. \boldsymbol{\omegaup} \left( M\right)^{\ast } \right. $ and
where we used the following short forms: $T_{MN}:=\hom{\Sscript{\cat{A}_{\cU}}}{M}{N}$, $T_{M}:=\hom{\Sscript{\cat{A}_{\cU}}}{M}{M}$. Therein, we also identify each element of the form $\varphi \tensor{T_{M}}m \in M^{\ast }\tensor{T_{M}}M$ with its image in the direct sum $\underset{\Sscript{M\,\in\, \mathrm{Ob}(\cat{A}_{\Sscript{\cU}})}}{\bigoplus}  M^{\ast }\tensor{T_{M}}M$.

The structure maps of the finite dual Hopf algebroid $(A,\cU^{\bcirc})$ are given as follows. Write $\overline{\varphi \otimes _{T_{M}}m}$ for the equivalence class of  the image (in the above direct sum) of a generic element for the form $\varphi \tensor{T_M}m \in M^*\tensor{T_M}M$, for some object $M \in \cat{A}_{\Sscript{\cU}}$. Since all involved maps are linear, we will be dealing most of all just with generic elements of the form $\overline{\varphi \tensor{T_{M}}m}$, bypassing the more general summation notation. Thus the structure maps on $\cU^{\bcirc }$ are given by
\begin{eqnarray*}
&&\sfun{\texttt{u}}{\K }{\cU^{\bcirc }}{1_{\K}}{\overline{\id_{A}\tensor{\K}1_{A}}}, \\
&&\sfun{\texttt{m}}{\cU^{\bcirc }\tensor{}\cU^{\bcirc }}{\cU^{\bcirc }}{\overline{\psi \tensor{ T_{N}}n}\tensor{}\overline{\varphi \tensor{T_{M}}m}}{\overline{\left( \psi \star \varphi\right) \tensor{T_{M\tensor{A}N}}\left( m\tensor{A}n\right) }}, \\
&&\sfun{\etaup}{A\tensor{}A}{\cU^{\bcirc }}{a\tensor{}b}{\overline{l_{a}\tensor{\K}b}},\;\; \text{where } l_a: A \to A, \big( 1 \mapsto a \big)\text{ is the left multiplication by }\, a,\\
&&\sfun{\varepsilon_{\circ}}{\cU^{\bcirc }}{A}{\overline{\varphi \tensor{T_{M}}m}}{\varphi \left( m\right)} ,\\
&&\sfun{\Delta_{\circ}}{\cU^{\bcirc }}{\cU ^{\bcirc }\tensor{A}\cU^{\bcirc }}{\overline{\varphi \tensor{T_{M}}m}}{\sum_{i=1}^{r}\overline{\varphi \tensor{T_{M}}e_{i}}\tensor{A}\overline{e_{i}^{\ast }\tensor{T_{M}}m}}, \;\; \text{where } \{e_i,e_i^*\}_i \text{ is a dual basis for } M_{\Sscript{A}}\\
&&\sfun{\cS_{\circ}}{\cU^{\bcirc }}{\cU^{\bcirc }}{\overline{\varphi \tensor{T_{M}}m}}{\overline{\mathrm{ev}_{m}\tensor{T_{\left. M^{\ast }\right. }}\varphi }},  \;\; \text{where } {\rm ev}_m: M^* \to A\; \text{ is the evaluation at } m \text{ map}.
\end{eqnarray*}
For every $\psi \in \left. N^{\ast }\right. $ and $\varphi \in \left. M^{\ast }\right.$,  the map $\psi \star \varphi :M\tensor{A}N\rightarrow A$ acts as $m\tensor{A}n\mapsto \varphi \left( m\right) \psi \left( n\right) $.

Notice that there is a linear map
\begin{equation}\label{Eq:zeta}
\zeta: \cU^{\bcirc} \longrightarrow \cU^{*},\quad \Big(  \bara{\varphi\tensor{T_M}m} \longmapsto \left[ u \mapsto \varphi(mu)\right]  \Big).
\end{equation}

The following lemma is a straightforward computation, see \cite{LaiachiGomez}.
\begin{lemma}\label{lema:zeta}
The linear map $\zeta$ is an homomorphism of $(A\tensor{}A)$-algebras.
\end{lemma}
It is noteworthy  to mention that the algebra map $\zeta$, in contrast with the case of algebras over a field, is not known to be injective. However, if the base algebra $A$ is  a Dedekind domain for example, then it is guaranteed that $\zeta$  is injective for every $\cU$, see   \cite{LaiachiGomez} for more details.

\subsection{The completion of the finite dual and the convolution algebra}\label{ssec:zeta}

Let $(A,\cU)$ be a co-commutative (right) Hopf algebroid and consider its finite dual $(A,\cU^{\bcirc})$ as a commutative Hopf algebroid with structure maps given as  in \S\ref{ssec:Fdual}.  Here  we assume that $\cU$  is  endowed with an admissible (increasing) filtration $\{F^n\cU\}_{n\,\in \, \mathbb{N}}$ as in \S\ref{ssec:FUstra}. The admissible filtration on the Hopf algebroid $\cU$ induces a filtration on the convolution algebra $\cU^*$ given as in \eqref{Eq:FUstar} of \S\ref{ssec:FUstra}. It turns out that $(A,\cU^*)$ with this filtration is a complete Hopf algebroid with structure maps explicitly given in \S\ref{sec:Ustar}.

\begin{proposition}\label{prop:App}
Let $(A,\cU)$ be a co-commutative (right) Hopf algebroid with an admissible filtration and consider its finite dual $(A,\cU^{\bcirc})$. Then the  canonical map $\zeta:\cU^{\bcirc}\to\cU^*$ of equation \eqref{Eq:zeta} is filtered with respect to the filtrations $F_n\cU^{\bcirc}=\cK^n$ and $F_{n+1}\cU^*=\ann{F^{n}\cU}$ for all $n\geq 0$ as in \eqref{Eq:FUstar}, where $\cK=\ker{\varepsilon_{\circ}:\cU^{\bcirc}\to A}$ is the kernel of the counit of $\cU^{\bcirc}$.
\end{proposition}
\begin{proof}
It can be easily checked that $\varepsilon_*\circ \zeta=\varepsilon_\circ$, where $\varepsilon_*:\cU^*\to A$ and $\varepsilon_\circ:\cU^{\bcirc} \to A$ are the counits. In particular this implies that $\zeta\left(\cK\right)\subseteq \ker{\varepsilon_*}$. Hence the claim will be proved if we will be able to show that $\ker{\varepsilon_*}\subseteq F_1\cU^*=\ann{F^0\cU}=\ann{A}$, because in this case multiplicativity of $\zeta$ will imply that
\begin{equation*}
\zeta\left(F_n\cU^\bcirc\right)=\zeta(\cK^n)\subseteq \zeta(\cK)^n\subseteq \left(F_1\cU^*\right)^n\subseteq  F_n\cU^*.
\end{equation*}
However, if $f\in\ker{\varepsilon_*}$ then $f(1_\cU)=0$, whence $f(\tau_0(a))=f(1_\cU\blacktriangleleft a)=f(1_\cU)a=0$. Consequently, $F^0\cU=A\subseteq \ker{f}$, from which it follows that $\ker{\varepsilon_*}\subseteq \ann{F^0\cU}$ as desired.
\end{proof}

In light of Proposition \ref{coro:CHc}, $(A,\what{\cU^{\bcirc}})$ is a complete Hopf algebroid. On the other hand, we know from Proposition \ref{prop:Bosco} that $(A,\cU^{*})$ admits a structure of complete Hopf algebroid whenever the translation map of $\cU$ is a filtered algebra map.  Combining all this allows us to improve the content of Lemma \ref{lema:zeta} and claim our main result as follows.

\begin{theorem}\label{thm:triangle}
Let $(A,\cU)$ be a co-commutative (right) Hopf algebroid with an admissible filtration and  assume that the translation map $\delta$ of $\cU$ is a filtered algebra map. Then the $(A\otimes_{}A)$-algebra map $\zeta: \cU^{\bcirc} \to \cU^*$ of equation \eqref{Eq:zeta} factors through a continuous morphism $\what{\zeta}: \what{\cU^{\bcirc}} \to \cU^*$ of complete Hopf algebroids. Thus we have a commutative diagram:
$$
\xymatrix@R=15pt@C=30pt{ \cU^{\bcirc} \ar@{->}^-{\zeta}[rr] \ar@{->}_-{\gamma}[rd]  &  & \cU^* \\ & \what{\cU^{\bcirc}} \ar@{->}_-{\what{\zeta}}[ru] & }
$$
In particular, this applies to $\cU=\cV_{\Sscript{A}}(L)$, the universal enveloping Hopf algebroid of any Lie-Rinehart algebra $(A,L)$ such that $L_{\Sscript{A}}$ is a finitely generated and projective module.
\end{theorem}
\begin{proof}
In Proposition \ref{prop:App} we showed that $\zeta$ is a filtered algebra map. Thus, by applying the completion 2-functor of Theorem \ref{thm:Athm}  to $\zeta$ ($A$ is trivially filtered), we obtain that  $\what{\zeta}$ is a continuous morphism of complete algebras. 
Now, since we already know that $\varepsilon_*\circ\zeta =\varepsilon_\circ$ and in view of Lemma \ref{lema:zeta}, we are left to show that $\what{\zeta}$ is compatible with the comultiplications and the antipodes. That is, the following relations hold
$$\left(~\what{\zeta}\cmptens{A}\what{\zeta}~\right) \circ \what{\Delta_{\circ}}=\Delta_* \circ \what{\zeta}\quad \text{and} \quad \what{\zeta} \circ \what{\cS_\circ} = \cS_*\circ \what{\zeta}.$$
However, notice that to this aim it will be enough to show the following ones
$$\gamma_{\cU^*\tensor{A}\cU^*} \circ (\zeta\tensor{A}\zeta) \circ \Delta_\circ=\Delta_* \circ \zeta\quad \text{and} \quad \zeta \circ \cS_\circ = \cS_*\circ \zeta.$$
Hence, let us consider an element of the form $\overline{\varphi \tensor{T_M} m}\in \cU^\bcirc$. So we obtain an element in $\what{\cU^*\tensor{A}\cU^*}=\cU^*\cmptens{A}\cU^*$ given by
\begin{equation*}
(\gamma_{\cU^*\tensor{A}\cU^*}(\zeta\tensor{A}\zeta)\Delta_\circ)\left(\overline{\varphi \tensor{T_M} m}\right)=\what{\left(\sum_{i} \zeta \left( \overline{\varphi \tensor{T_M} e_i} \right) \tensor{A} \zeta \left( \overline{e_i^* \tensor{T_M} m} \right) \right)}=\limn\left(\sum_{i} \zeta \left( \overline{\varphi \tensor{T_M} e_i} \right) \tensor{A} \zeta \left( \overline{e_i^* \tensor{T_M} m} \right) \right).
\end{equation*}
For every $u,v\in\cU$, it satisfies
\begin{equation*}
\limn\left(\sum_{i} \zeta \left( \overline{\varphi \tensor{T_M} e_i} \right) \left( \zeta \left( \overline{e_i^* \tensor{T_M} m} \right)(u)v \right)\right) = \limn\left(\sum_{i} \varphi\left( e_ie_i^*(mu)v \right)\right)=\varphi(m(uv)) = \zeta\left(\overline{\varphi \tensor{T_M} m}\right)(uv)
\end{equation*}
whence, by the criterion of equation \eqref{Eq:fuv}, we have that $\gamma_{\cU^*\tensor{A}\cU^*} \circ (\zeta\tensor{A}\zeta) \circ \Delta_\circ=\Delta_* \circ \zeta$. Moreover,
\begin{equation*}
\left(\zeta\cS_\circ\left(\overline{\varphi \tensor{T_{M}} m}\right)\right)(u)=\zeta\left(\overline{\mathrm{ev}_m \tensor{T_{M^*}} \varphi}\right)(u)=\left(\varphi \centerdot u \right)(m)\stackrel{\eqref{Eq:star}}{=}\varepsilon_*\left(\varphi \left(m u_-\right)u_+\right)\stackrel{\eqref{Eq:tantip}}{=}\cS_*\left(\zeta\left(\overline{\varphi \tensor{T_{M}} m}\right)\right)(u),
\end{equation*}
for every $u \in \cU$, and the proof is complete.
\end{proof}

As in Example \ref{exam:Sacarrelli}, we are going to consider the $A$-bimodule $A\tensor{} A$ to be endowed with the $K$-adic filtration given by $K:=\ker{\mu_{A}:A\tensor{} A\to A}$, even if $A$ itself is trivially filtered.

\begin{proposition}\label{prop:zeroneveropen}
Let $(A,\cU)$ and $(A,\cU^{\bcirc})$ be as in Proposition \ref{prop:App} and assume that $\zeta:\cU^{\bcirc}\to\cU^{*}$ is injective. Then the following assertions are equivalent
\begin{enumerate}[label=(\alph*)]
\item the morphism $\what{\zeta}: \what{\cU^{\bcirc}} \to \cU^*$ is a filtered isomorphism,\label{list:1}
\item the morphism $\gr\left(\,\what{\zeta}\,\right):\gr{\left(\what{\cU^{\bcirc}}\right)}\to \gr{\left(\cU^{*}\right)}$ is a graded isomorphism,\label{list:2}
\item the morphism $\what{\zeta}$ is surjective and the $\cK$-adic filtration on $\cU^{\bcirc}$ coincides with the one induced from $\cU^{*}$ via $\zeta$,\label{list:3}
\item the graded morphism $\gr\left(\,\what{\zeta}\,\right):\gr{\left(\what{\cU^{\bcirc}}\right)} \to \gr{\left(\cU^{*}\right)}$ is surjective and the $\cK$-adic filtration on $\cU^{\bcirc}$ coincides with the one induced from $\cU^{*}$ via $\zeta$,\label{list:4}
\item the graded morphism $\gr\left({\zeta}\right):\gr{\left({\cU^{\bcirc}}\right)} \to \gr{\left(\cU^{*}\right)}$ is surjective and the $\cK$-adic filtration on $\cU^{\bcirc}$ coincides with the one induced from $\cU^{*}$ via $\zeta$,\label{list:5}
\end{enumerate}
Moreover, the following assertions are equivalent as well
\begin{enumerate}[resume*]
\item the morphism $\what{\zeta}: \what{\cU^{\bcirc}} \to \cU^*$ is an homeomorphism, \label{list:6}
\item the morphism $\what{\zeta}: \what{\cU^{\bcirc}} \to \cU^*$ is open and injective and $\cU^\bcirc$ is dense in $\cU^*$, \label{list:7}
\item the $\cK$-adic topology on $\cU^{\bcirc}$ is equivalent to the one induced from $\cU^{*}$ via $\zeta$ and $\cU^\bcirc$ is dense in $\cU^*$.\label{list:8}
\end{enumerate}
If in addition the morphism $\what{\vartheta}$ induced by the algebra map $\vartheta:A\otimes A\to\cU^*$ of equation \eqref{Eq:vartheta} is a filtered isomorphism (as in the example mentioned in the introduction), then all the assertions from \ref{list:1} to \ref{list:8} are equivalent.
\end{proposition}
\begin{proof}
Before proceeding with the proof, there are some facts that have to be highlighted or recalled. First of all, notice that injectivity of $\zeta$ implies that the filtration on $\cU^\bcirc$ is separated. Secondly, recall that a morphism of filtered bimodules $f:M\to N$ is said to be strict if $f(F_kM)=f(M)\cap F_kN$ for all $k\geq 0$. In particular, $\zeta$ is strict if and only if the $\cK$-adic filtration on $\cU^\bcirc$ coincides with the one induced from $\cU^*$ via $\zeta$ itself. Thirdly, a filtered morphism (as $\what{\zeta}$ for example) is a filtered isomorphism if and only if it is bijective and strict. Finally, we have that $\gr\left(\gamma_{\Sscript{\cU^\bcirc}}\right):\gr\left(\cU^\bcirc\right)\to\gr\left(\what{\cU^\bcirc}\right)$ is always an isomorphism (see e.g.~\cite[Proposition D.3.1]{NasOys}), so that $\gr\left(\,\what{\zeta}\,\right)$ is injective (resp.~surjective, bijective) if and only if $\gr(\zeta)$ is. 
Now, by applying \cite[Cor. D.III.5, D.III.6 and D.III.7]{NasOys} one proves that \ref{list:3} $\Leftrightarrow$ \ref{list:1} $\Leftrightarrow$ \ref{list:2} $\Leftrightarrow$ \ref{list:4} $\Leftrightarrow$ \ref{list:5}.

For the remaining equivalent facts, notice that $\what{\zeta}$ is surjective if and only if for every $x\in\cU^*$ and for all $k\geq 0$, there exists $m_k\in \cU^\bcirc$ such that $x-m_k\in F_k\cU^*$ or, equivalently, if and only if $\cU^\bcirc$ is dense in $\cU^*$. This proves the equivalence between \ref{list:6} and \ref{list:7}, so that we may focus on \ref{list:7} $\Leftrightarrow$ \ref{list:8}. Assume initially that $\what{\zeta}$ is an open and injective map. From this it follows that for all $h\geq 0$, $F_h\what{\cU^\bcirc}$ is open in $\cU^*$. In particular, there exists $k\geq 0$ such that $F_k\cU^*\subseteq F_h\what{\cU^\bcirc}$. Thus, $M\cap F_k\cU^*\subseteq M\cap F_h\what{\cU^\bcirc}=F_h\cU^\bcirc$, which expresses the fact that the $\cK$-adic topology is equivalent to the induced one. Conversely, assume that these two topologies are equivalent and that $\cU^\bcirc$ is dense in $\cU^*$ (that is, that $\what{\zeta}$ is surjective). We plan to prove first that every $F_t\what{\cU^\bcirc}$ is open in $\cU^*$ (which implies that $\what{\zeta}$ is open) and then that $\what{\zeta}$ is injective. To this aim, pick $t\geq 0$ and consider $k$ (which we may assume greater or equal than $t$) such that $\cU^\bcirc\cap F_k\cU^*\subseteq F_t{\cU^\bcirc}$. Then every $y\in F_k\cU^*$ is of the form $y=\what{\zeta}\left(\left(m_i+F_i\cU^\bcirc\right)_{\Sscript{i\geq 0}}\right)=\left(m_i+F_i\cU^*\right)_{\Sscript{i\geq 0}}$ for some $\left(m_i+F_i\cU^\bcirc\right)_{\Sscript{i\geq 0}}\in\what{\cU^\bcirc}$ such that $m_k\in F_k\cU^*\cap \cU^\bcirc\subseteq F_t\cU^\bcirc$, whence 
$$
m_t+F_t\cU^\bcirc = m_k+F_t\cU^\bcirc = 0
$$
in the quotient $\cU^\bcirc/F_t\cU^\bcirc$ and so $\left(m_i+F_i\cU^\bcirc\right)_{\Sscript{i\geq 0}}\in F_t\what{\cU^\bcirc}$. Summing up, we showed that for every $t\geq 0$, there exists a $k\geq t$ such that $F_k\cU^*\subseteq F_t\what{\cU^\bcirc}$ and hence that $\what{\zeta}$ is an open map. Let us show now that it is injective as well. To this aim, let $\left(m_i+F_i\cU^\bcirc\right)_{\Sscript{i\geq 0}}$ be an element in $\ker{\,\what{\zeta}\,}$. This implies that $m_k\in F_k\cU^*\cap {\cU^\bcirc}$ for all $k\geq 0$ and that, since the two topologies are equivalent, for every $i\geq 0$ there exists $j_i\geq i$ such that $F_j\cU^*\cap \cU^\bcirc\subseteq F_i\cU^\bcirc$, whence
$$
m_k+F_k\cU^\bcirc = m_{j_k}+F_k\cU^\bcirc \in \left(F_{j_k}\cU^*\cap \cU^\bcirc\right)+F_k\cU^\bcirc=F_k\cU^\bcirc,
$$
so that $\left(m_i+F_i\cU^\bcirc\right)_{\Sscript{i\geq 0}}=0$. With this we conclude the proof that \ref{list:6} $\Leftrightarrow$ \ref{list:7} $\Leftrightarrow$ \ref{list:8}.

Finally, \ref{list:1} clearly implies \ref{list:6} and since $\zeta\circ \eta=\vartheta$, if $\what{\vartheta}$ is a filtered isomorphism then $\what{\zeta}$ admits the filtered section $\what{\etaup}\circ \what{\vartheta}^{-1}$. Therefore, if in such a case $\what{\zeta}$ is also injective, then it is a filtered isomorphism.
\end{proof}

The subsequent Corollary gives another condition for the injectivity of the map $\zeta$, and so another application of the result  \cite[Theorem 4.2.2]{LaiachiGomez}.
\begin{corollary}\label{coro:Equivalence}
Let $(A,L)$ be a Lie-Rinehart algebra and consider $\cU=\cV_{\Sscript{A}}(L)$ its universal enveloping Hopf algebroid. Assume  that $\cU^{\circ}$ is an Hausdorff topological space with respect to the $\cK$-adic topology and that  $\what{\zeta}$ is an homeomorphism. Then $\zeta$ is injective, and therefore, there is an equivalence of symmetric rigid monoidal categories between the category of right  $L$-modules  and the category of right $\cU^{\circ}$-comodules, with finitely generated and projective underlying $A$-module structure.
\end{corollary}

\begin{remark}\label{rem:laventanadelfrente}
As a final remark, we point out that the completion of $\zeta$ might fail to be an homeomorphism, even if $\zeta$ is injective and $A$ is the base field, as it is shown in \cite{LaiachiPaolo} for an apparently trivial example: namely the enveloping Hopf algebra $\cU=U(L)$ of the one dimensional complex Lie algebra  $L$.
Nevertheless, we believe that in the Hopf algebroid framework some unexpected result may show up. For instance, we just mention that the classical Sweedler dual coalgebra $\cU^{o}$ of the first Weyl algebra $\cU$ as in Example \ref{exam:URSO} is zero, while the finite dual Hopf algebroid $\cU^{\bcirc}$ is not.
This, in our opinion, suggests that the problem of $\what{\zeta}$ being an homeomorphism  or not for universal enveloping Hopf algebroids is still worthy to be studied. 

In fact, we believe that the presence of a non-trivial algebra of infinite jets $\cJ(A)=\what{A\otimes A}$ (see Example \ref{exam:Sacarrelli} for the definition) could make the difference. Let us consider again the diagram \eqref{eq:maindiagram} for $(A,\cU)$ a co-commutative Hopf algebroid with an admissible filtration such that the translation map $\delta$ is filtered,
\begin{equation*}
\xymatrix@R=15pt@C=30pt{ \what{\cU^{\bcirc}} \ar@{->}^-{\what{\zeta}}[rr]   &  & {\cU}^{*}  \\ & \cJ(A) \ar@{->}_-{\what{\vartheta}}[ru]  \ar@{->}^-{\what{\etaup}}[lu] & }
\end{equation*}
In case $\what{\vartheta}$ turns out to be an isomorphism of complete Hopf algebroids (as for example when $A=\C[X]$ and $\cU=\diff(A)$), then one may reasonably conjecture that $\what{\zeta}$ could become an isomorphism as well.

If we look instead at the aforemenioned case, that is to say, $A=\C$ and $L=\C X$, then $\cJ(\C)=\what{\C\otimes\C}\cong \C\otimes\C\cong \C$ and hence the completion $\what{\vartheta}$ of the algebra map $\vartheta\colon\C\otimes\C\to U(L)$ corresponds to the unit $\C\to \C[[X]]$, which is far away from being an isomorphism.
\end{remark}


\appendix

\section{Complete bimodules and the completion 2-functor}\label{sec:CBCF}

In this section we revise some notions on linear topology of rings and modules which are well-known or folklore, apart perhaps from the adjunction between the topological tensor product and the continuous hom functor. For a more exhaustive treatment of the material of this section, we refer to \cite{MR0163908,MR0358652,MR1420862,NasOys}. The reason that pushed us to put this material in a comprehensive way was the apparent lack of a single reference in the literature which could clarify in an exhaustive way  the constructions performed for complete Hopf algebroids in Subsection \ref{sec:CHA}.
We decided then to include a detailed exposition, especially for readers who are not familiar with this context.

\subsection{Filtered bimodules over filtered algebras}\label{ssec:CBFA}
Here we retrieve some basic notions and results in order to make explicit our assumptions and fix some notations. For further details on this subsection, we refer to \cite[Chapter I, \S\S1-3]{MR1420862} and \cite[Chapter I-III]{MR0358652}.

As far as we will be concerned with this, a \emph{linear topology} on an algebraic structure is a topology on the underlying set with respect to which all structure maps are continuous. An algebra $R$ is said to be \emph{filtered} if there exists a decreasing chain of two-sided ideals
\begin{equation*}
R=F_{0}R\supseteq F_{1}R\supseteq \cdots
\end{equation*}
that satisfies $F_{n}R\cdot F_{m}R\subseteq F_{n+m}R$ for every $m,n\in \N$. We will denote it as a pair $\left(
R,F_{n}R\right) $ or we will just say that $R$ is filtered. Given a filtration on $R$, this induces a linear topology on it such that $\left\{ F_{n}R\mid n\in \N\right\} $ is a fundamental system of neighborhoods of $0$
and $\left\{x+F_{n}R\mid n\in \N\right\} $ is a fundamental system of neighborhoods of $x\in R$. A subset $U$ is open in $R$ if and only if for every $x\in U$ there exists an $n\in \N$ such that $x+F_{n}R\subseteq U$, while a subset $V$ is closed if and only if $V=\cap_{n \geq 0}\big(  V+ F_{n}R\big)$. This, in particular, implies that $\left\{ x+F_{n}R\mid x\in R,n\in \N \right\} $ is a basis for this topology. Furthermore, by \cite[III.49, \S\ 6.3]{MR0358652}, this topology is compatible with the ring structure on $R$ (cf. also Example 3 in the same page).

\begin{remark}\label{rem:K}
We will always endow the base ring $\K$ with the trivial filtration $F_0(\K)=\K$ and $F_n(\K)=0$ for every $n\geq 0$. This filtration induces on $\K$ the discrete topology because $\left\{0\right\}$ is open by definition, whence every point is open. This topology is always compatible with all algebraic structures on $\K$ (even
\begin{equation*}
\sfun{(-)^{-1}}{\K^*}{\K^*}{k}{k^{-1}}
\end{equation*}
in case $\K$ is a field, cf. e.g. \cite[III.55, \S\ 6.7, Example 1]{MR0358652}).
\end{remark}

In view of Remark \ref{rem:K} and of \cite[III.53, \S\ 6.6, Remarque]{MR0358652}, the linear topology on $R$ is compatible with the module structure, too.  In other words, $R$ is a topological algebra.

Let $R$, $S$ be filtered algebras. An $\left( S,R\right) $-bimodule $M$ is said to be \emph{filtered} (as a bimodule) if there exists a decreasing chain of submodules
\begin{equation*}
M=F_{0}^{S,R}M\supseteq F_{1}^{S,R}M\supseteq \cdots
\end{equation*}
that satisfies
\begin{equation*}
F_{n}S\cdot F_{m}^{S,R}M\subseteq F_{n+m}^{S,R}M\text{\qquad and\qquad }F_{n}^{S,R}M\cdot F_{m}R\subseteq F_{n+m}^{S,R}M
\end{equation*}
for every $m,n\in \N$.  We will denote it as a pair $\left(M,F_{n}^{S,R}M\right) $ or we will just say that $M$ is filtered. Note that each $F_{n}^{S,R}M$ is in particular an $\left( S,R\right) $-subbimodule. If $M$ is a filtered $\left( S,R\right) $-bimodule then it can be endowed with a linear topology such that the given filtration forms a fundamental system of neighborhoods of 0. As above, a basis for this topology is given by the open sets $\left\{m+F_{n}^{S,R}M\mid m\in M,n\in \N\right\} $. For the sake of simplicity, the filtration on a $(S,R)$-bimodule $M$ will be denoted by $\left\{F_nM\mid n\in \N\right\}$.

A filtration $\left\{F_nM\mid n\in\N\right\}$ on an $(S,R)$-bimodule $M$ is said to be \emph{finer} than another filtration $\left\{G_nM\right\}$ on $M$  if and only if for every $n\in\N$ there exists an $m\in\N$ such that $F_mM\subseteq G_nM$ (cf. \cite[I.38, \S\ 6.3, Proposition 4]{MR0358652}). As a consequence, the linear topology induced by the filtration $\left\{F_nM\right\}$ is finer than the one induced by the filtration $\left\{G_nM\right\}$.
Two filtrations $\left\{F_nM\mid n\in\N\right\}$ and $\left\{G_nM\mid n\in\N\right\}$ on an $(S,R)$-bimodule $M$ are said to be \emph{equivalent} if and only if each one is finer than the other one. In particular, the linear topologies induced on $M$ are equivalent.

The category of filtered $\left( S,R\right) $-bimodules, denoted by $\FBim{S}{R}$, is defined as follows. The objects are filtered $\left(S,R\right) $-bimodules $M$. The arrows are $\left(S,R\right) $-bimodule maps $f:M\rightarrow N$ that satisfies $f\left(F_{n}M\right)\subseteq F_{n}N$,  for every $n\in \N$.\footnote{Such homomorphisms are called \emph{of degree $0$} in the literature. Cf. e.g. \cite[Definition D.I.5]{NasOys}.} It is a (co)complete additive category with kernels and cokernels.
If $(M_\lambda,f_{\lambda,\,\mu})$ is a projective system of filtered $(S,R)$-bimodules then its projective limit $\prlimit{\lambda}{M_\lambda}$ is filtered with filtration
\begin{equation}\label{eq:projFiltr}
F_k\left(\prlimit{\lambda}{M_\lambda}\right)=\prlimit{\lambda}{F_k\left(M_\lambda\right)}.
\end{equation}
We have a functor $\varfun{D}{\Bim{S}{R}}{\FBim{S}{R}}$ which associates to $M$ in $\Bim{S}{R}$ the $(S,R)$-bimodule $M$ itself with filtration
\begin{equation*}
F_nM:=\sum_{h+k=n}F_kS\cdot M\cdot F_hR
\end{equation*}
for all $n\in\N$. This filtration is called the \emph{induced} filtration. If $(M,F_nM)$ is a filtered $(S,R)$-bimodule, for $k\in\N$ the \emph{$k$-shifted module $M[k]$} is the same $(S,R)$-bimodule as $M$, but filtered with a different filtration given by $F_nM[k]=F_{n+k}M$.\footnote{In \cite[I.2.7]{MR1420862} the shifted module is denoted by $T(n)M$. Notice that here we consider only positively filtered modules with decreasing filtration.}

\begin{remark}\label{rem:continuity}
As every function from a discrete topological space to any topological space is continuous, every morphism from a trivially filtered bimodule to any bimodule is automatically filtered.

Moreover, independently from being filtered or not, an $(S,R)$-bimodule homomorphism $\varfun{f}{M}{N}$ is continuous with respect to the linear topologies induced by the filtrations if and only if for every $n\in \N$ there exists $m\left(n\right) \in \N$ such that $f\left( F_{m\left( n\right) }M\right) \subseteq F_{n}N$. In particular, any morphism of filtered $(S,R)$-bimodules is continuous (cf. also \cite{MR1320989}). On the other hand, given $M,N$ two filtered $(S,R)$-bimodules, one can prove that an $(S,R)$-bimodule homomorphism $\varfun{f}{M}{N}$ is continuous with respect to the linear topologies induced by the given filtrations if and only if there exists a sub-filtration on $M$ equivalent to the former one and with respect to which $f$ is filtered.
\end{remark}

In light of Remark \ref{rem:continuity}, we will distinguish \emph{homeomorphism} as topological spaces from \emph{filtered isomorphism} as filtered bimodules: the second terminology will be used for isomorphism of filtered bimodules whose inverse is also filtered. Note that every filtered isomorphism is in fact an homeomorphism.

Now, if $M$ and $N$ are filtered $(S,R)$ and $(R,T)$-bimodules respectively, then there is a natural filtration on their tensor product $M\tensor{R}N$ given by
\begin{equation}\label{eq:filtrations}
 \cF_n\left(M\tensor{R}N\right):=\sum_{p+q=n}\Img{F_{p}M\tensor{R}F_{q}N}
\end{equation}
for all $n\in\N$, where the notation in the right hand side is the obvious one. 
We will consider this one as the standard filtration on the tensor product of filtered $(S,R)$ and $(R,T)$-bimodules, for all algebras $S$, $R$, $T$.

\begin{proposition}
If $\varfun{f}{M}{M^{\prime}}$ and $\varfun{g}{N}{N^{\prime}}$ are morphisms of filtered $(S,R)$ and $(R,T)$-bimodules respectively, then $f\tensor{R}g : M\tensor{R}N \rightarrow M'\tensor{R}N'$ is a morphism of filtered  $(S,T)$-bimodules. In particular, we have a bicategory $\cB im_\K^{\mathsf{flt}}$ which has filtered algebras as $0$-cells and whose categories of $\{1,2\}$-cells are the categories of filtered bimodules over filtered algebras with vertical and horizontal compositions given by the composition of bilinear morphisms and the usual tensor product, filtered as in \eqref{eq:filtrations}, respectively.
\end{proposition}

Notice that the category $\rmod{\K}^{\mathsf{flt}}$ of filtered modules is monoidal with tensor product $\otimes$ and unit $\K$. Filtered algebras are monoids in $\rmod{\K}^{\mathsf{flt}}$ and filtered $(S,R)$-bimodules are objects in ${}_{\Sscript{S}}^{}(\rmod{\K}^{\mathsf{flt}})^{}_{\Sscript{R}}$. This in particular, means that the categories of $\{1,2\}$-cells ${}_{\Sscript{S}}(\cB im_\K^{\mathsf{flt}}){}_{\Sscript{R}}$ of $\cB im_\K^{\mathsf{flt}}$  are exactly  $ {}_{\Sscript{S}}^{}(\rmod{\K}^{\mathsf{flt}})^{}_{\Sscript{R}}$.

\begin{remark}
Similar to the discrete case (i.e., usual bimodules), we have an isomorphism between the category $\FBim{S}{R}$ of filtered $(S,R)$-bimodules and the category ${_{S\otimes R^{\Sscript{\Sf{op}}}}\Sf{Mod}^{\Sf{flt}}}$ of filtered $S\otimes R^{\Sscript{\text{op}}}$-modules, where $R^{\Sscript{\Sf{op}}}$ denotes the opposite algebra of $R$ and $S\otimes R^{\Sscript{\text{op}}}$ is a filtered algebra with filtration as in \eqref{eq:filtrations}.
\end{remark}

\begin{example}\label{ex:HomFiltr}
Let $P\in \FBim{S}{T}$ and $N\in \FBim{R}{T}$ and let us denote by $\FHom{-,T}{N}{P}$ the abelian group of filtered morphisms $\varfun{f}{N}{P}$ which are $T$-linear. As one can expect it is an object in $\Bim{S}{R}$. It is also filtered with filtration given by
\begin{equation}\label{eq:FiltrHom}
F_n\FHom{-,T}{N}{P}=\FHom{-,T}{N}{P[n]}=\Big\{f\in \hom{-,T}{N}{P}\mid f\left(F_kN\right)\subseteq F_{n+k}P\text{ for all }k\geq0\Big\}.
\end{equation}
Observe that this is the filtration induced by the filtered bimodule of all homomorphisms of finite degree $\mathsf{HOM}_{-,T}\left(N,P\right)$ onto its subgroup $F_0\mathsf{HOM}_{-,T}\left(N,P\right)=\FHom{-,T}{N}{P}$ (see e.g. \cite[I.2.5]{MR1420862}).
\end{example}

It is worthy to mention that the linear topology makes of $M$ a linearly topologized $\left( S,R\right) $-bimodule in the sense of \cite[Definition 1.1]{MR1608699}, whence  it endows  $\CEnd{-R}{M}$ with the topology of uniform convergence on $M$.

\subsection{The completion 2-functor}\label{ssec:CF}
In this subsection we will recall the construction of the completion functor from the category of filtered bimodules to the one of complete bimodules. As a main reference for the material presented here, we suggest \cite[Chap. D, \S\S\ I-II]{NasOys} and \cite[Chap. I, \S3]{MR1420862}.

Let $S$, $R$ be filtered algebras and let $\left( M, F_{n}M\right) $ be a filtered $(S,R)$-bimodule. We recall that $M$ is \emph{Hausdorff} (or \emph{separable}) if and only if for every pair of elements $x,y\in M$ there exist two open sets $U,V\subseteq M$ such that $x\in U$, $y\in V$ and $U\cap V=\emptyset$. However, by definition of the linear topology on $M$, this is equivalent to say that $\bigcap_{n\in\N}F_nM=0$. Moreover, a sequence $\left\{m_k \mid k\geq0\right\}$ in a Hausdorff filtered $(S,R)$-bimodule $M$ is a \emph{Cauchy sequence} if and only if for every $p\in \N$, there exists $q\in \N$ such that for all $k,h\geq q$ we have that $m_k-m_h\in F_p(M)$. It is \emph{convergent} to an element $m\in M$ if and only if for every $p\in \N$, there exists $q\in \N$ such that for all $k\geq q$ we have that $m-m_k\in F_p(M)$. The bimodule $M$ is said to be \emph{complete} with respect to the linear topology induced by the filtration if and only if every Cauchy sequence is convergent.

Now, the filtration on $M$ gives rise to a projective system of $(S,R)$-bimodules given by
\begin{equation}\label{eq:filtprojsyst}
\lfun{\pi _{m,n}^{M}}{\frac{M}{F_{m}M}}{\frac{M}{F_{n}M}}{x+F_{m}M}{ x+F_{n}M}
\end{equation}
for all $m\geq n$ and this allows us to give an effective characterization of when a filtered bimodule is Hausdorff and complete, as well as a universal construction of its Hausdorff completion. To this aim, set
\begin{equation*}
\what{M}:=\prlimit{n}{ \frac{M}{F_{n}M}}
\end{equation*}
and consider the canonical morphism $\varfun{\gamma_M}{M}{\what{M}}$ rendering commutative the diagram
\begin{equation}\label{Eq:gamma}
\xymatrix @R=15pt @C=15pt{  M \ar@{-->}^-{\gamma_M}[rr] \ar@{->}_-{\pi_n}[dr]  & & \what{M} \ar@{->}^-{p_n}[dl]  \\&  \frac{M}{F_nM} &}
\end{equation}
for all $n\in\N$, where $\varfun{p_{n}}{\what{M}}{{M}/{F_{n}M}}$ are the natural projections. The subsequent result can be proven directly (see also \cite[Proposition D.II.3]{NasOys}).

\begin{proposition}\label{prop:completion}
An object $\left( M, F_{n}M\right) $ in $\FBim{S}{R}$ is complete and Hausdorff as a topological space if and only if the map $\gamma_{\Sscript{M}}$  of diagram \eqref{Eq:gamma}  is an isomorphism.
\end{proposition}

This justifies the following definition.

\begin{definition}\label{def:completion}
\label{Def:T2compl}For a filtered $\left( S,R\right) $-bimodule $M$, we define its \emph{Hausdorff completion} to be the inverse limit $\what{M}$ over the natural projective system as in \eqref{eq:filtprojsyst}.
\end{definition}

As a matter of terminology, henceforth we will understood that a complete bimodule is Hausdorff as well, whence we will just refer to \emph{complete} bimodules and \emph{completions} of filtered bimodules.

The fact that Definition \ref{def:completion} is consistent (i.e., that the completion $\what{M}$ of a filtered $(S,R)$-bimodule $M$ is a complete $(S,R)$-bimodule) follows  from  the subsequent Lemma \ref{lemma:Mhatcomplete} (see also \cite[Proposition D.II.3]{NasOys}).

\begin{lemma}\label{lemma:Mhatcomplete}
Let $\left(M,F_nM\right)$ be a filtered $(S,R)$-bimodule. We have an isomorphism $\what{M}/F_{n}\what{M}\cong {M}/{F_{n}M}$ in $\FBim{S}{R}$ for all $n\geq 0$ which is compatible with the morphisms of the projective system $\eqref{eq:filtprojsyst}$. In particular, $\what{M}=\prlimit{n}{M/F_nM}$ is a complete $(S,R)$-bimodule.
\end{lemma}

The proof is omitted, but we point out that every quotient module $M/F_nM$ is filtered with the discrete filtration $F_k\left(M/F_nM\right):=F_kM/F_nM$ for all $k\geq 0$ and that $\what{M}$ is filtered with the filtration given in \eqref{eq:projFiltr}, which satisfies
\begin{equation}\label{eq:Mhatfiltration}
F_{m}\what{M}=\ker{ p_{m}:\what{M}\rightarrow \frac{M}{F_{m}M}}.
\end{equation}
In particular, the canonical $(S,R)$-bilinear morphism $\varfun{\gamma_{\Sscript{M}}}{M}{\what{M}}$ is always filtered and every $M/F_nM$ is a complete $(S,R)$-bimodule. Moreover, the canonical projections $\pi_n^{\Sscript{M}}:M\to M/F_nM$ are filtered for all $n\geq 0$, whence continuous, and every $F_nM$ is closed in $M$, as preimage of the closed set $0$ in $M/F_nM$.

\begin{remark}\label{rem:gamma}
Assume that $M$ is a complete $(S,R)$-bimodule. Then the inverse morphism of  $\varfun{\gamma_{\Sscript{M}}}{M}{\what{M}}$ is given by the assignment
\begin{equation}\label{eq:invCan}
\sfun{\sigma_{\Sscript{M}}}{\what{M}}{M}{\left(x_k+F_kM\right)_{k\geq 0}}{\lim_{k\to\infty}(x_k)},
\end{equation}
which is well-defined because the limit $\underset{k\to\infty}{\lim}(x_k)$ is independent of the representatives chosen for the equivalence classes $x_k+F_kM\in {M}/{F_kM}$, $k\geq0$.
\end{remark}

The following conventions turn out to be very useful in dealing with completions, whence we opted for introduce them in this general context.

\begin{notations}\label{Not:infty}
Given a filtered $(S,R)$-bimodule $M$, by a slight abuse of notation we are going to denote the elements of its completion $\what{M}$ by $\what{x}_\infty$, meaning by that an $\N $-tuple $\big( x_n + F_nM \big)_{n \geq 0} \in \prod_{n \geq 0} {M}/{F_nM}$ satisfying $x_{n+1}-x_{n} \in F_{n}M$ for all $n \geq 0$.
Observe that this condition is in fact equivalent to claim that $\big( x_n + F_nM \big)_{n \geq 0} \in \prlimit{n}{ {M}/{F_{n}M}}$. If $x \in M$, then its image via $\gamma_{\Sscript{M}}$ in $\what{M}$ will be denoted by $\what{x}$, which corresponds to the $\N $-tuple $\big( x + F_nM \big)_{n \geq 0}$. When $M$ is complete, and so $\gamma_{\Sscript{M}}$ and $\sigma_{\Sscript{M}}$ are mutually inverse functions, the element $x_\infty:=\sigma_{\Sscript{M}}\left(\what{x}_\infty\right)$ belongs to $M$ and the condition \begin{equation*}
\big( x_n + F_nM \big)_{n \geq 0}=\what{x}_{\infty}=\gamma_{\Sscript{M}}\sigma_{\Sscript{M}}\left(\what{x}_\infty\right) = \gamma_{\Sscript{M}}\left(x_\infty\right)=\big( x_\infty + F_nM \big)_{n \geq 0}
\end{equation*}
implies the following: for all $n \geq 0$, there exists $k(n) \geq 0$ such that for every $p \geq k(n)$ we have $x_{\infty} - x_p \in F_nM$, i.e.,  $x_\infty=\limxn$ in $M$.
\end{notations}

\begin{remark}\label{rem:limits}
Observe that a sequence $\left\{x_n\mid n\in\N\right\}$ in $M$ is Cauchy if and only if the sequence $\left\{\what{x_n}\mid n\in\N \right\}$ in $\what{M}$ is Cauchy. Moreover, every element $\what{x}_\infty\in\what{M}$ can be seen as a Cauchy sequence in $M$, in the sense that if $\what{x}_\infty=\left(x_n+F_nM\right)_{n\geq0}\in  \what{M} $ then it follows that $\left\{x_n\mid n\in\N\right\}$ is Cauchy in $M$. It turns out, with the conventions introduced, that $\what{x}_\infty=\limn\left(\what{x_n}\right)$ in $\what{M}$ where $\left\{x_n\mid n\in\N\right\}$ is the Cauchy sequence defining $\what{x}_\infty$.

Again, by a slightly but consistent abuse of notation, we are going to write $\what{x}_\infty=\limn\left(x_n\right)$ whenever $\what{x}_\infty=\left(x_n+F_nM\right)_{n\geq0}$. This proves to be very useful when one will have to compute, for example, $\what{f}\,\left(\what{x}_\infty\right)$ for a given $f:M\to N$ (the meaning of $\what{f}$ is the expected one). Indeed
\begin{equation}\label{eq:fhat}
\what{f}\,\left(\limn(x_n)\right)=\what{f}\,\left(\what{x}_\infty\right)=\left(f(x_n)+F_nM\right)_{n\in\N}=\limn(f(x_n))=:\what{f(x)}_{\infty}.
\end{equation}

Notice further that for two given Cauchy sequences $\{x_n\mid n\in\N\}$ and $\{y_n\mid n\in\N\}$ in $M$, we have that $\limn(x_n)=\limn(y_n)$ in $\what{M}$ if and only if $\what{x}_\infty=\what{y}_\infty$, if and only if $x_n-y_n\in F_nM$ for all $n\in\N$.
\end{remark}

\begin{remark}\label{rem:gammafiltiso}
With this new notations, it is easy to show that if $M$ is complete, then the morphism $\sigma_{\Sscript{M}}$ is filtered as well, so that $\gamma_{\Sscript{M}}$ is a filtered isomorphism. Indeed, if we have $\what{x}_\infty \in \ker{ p_{n}:\what{M}\rightarrow {M}/{F_{n}M}}$ and if $x_\infty=\sigma_{\Sscript{M}}\left(\what{x}_\infty\right)$ is the limit of the sequence $\left\{x_k\mid k\geq 0\right\}$ as in Remark \ref{rem:gamma}, then $0=p_k\left(\what{x_\infty}\right)=x_\infty + F_kM$ for all $0\leq k\leq n$. In particular $x_\infty\in F_nM$.

For the sake of completeness, recall that the inverse limit topology on $\what{M}$ is the coarsest topology for which all the canonical projections $p_n$'s are continuous. It can be proven that the inverse limit topology is equivalent to the linear topology induced by the filtration \eqref{eq:projFiltr}.
\end{remark}

Consistently with our definition of a complete bimodule over filtered algebras, we say that a filtered algebra $R$ is a \emph{complete} algebra if it is also complete as a module.
Given a filtered algebra $R$, its completion $\what{R}$ as a filtered module inherits a structure of filtered algebra itself, which is complete as a module and such that the natural map $\gamma_{\Sscript{R}} : R \to \what{R}$ is a map of filtered algebras.
Explicitly, the multiplication $\varfun{\widetilde{\mu}}{\what{R}\times \what{R}}{\what{R}}$ is given by $\widetilde{\mu}\left(\what{x}_\infty, \what{y}_\infty\right)=\what{xy}_\infty$ and the unit is $1_{\Sscript{\what{R}}}=\what{1_{\Sscript{R}}}$. Therefore, the completion of a filtered algebra is a complete algebra, as expected. We point out in advance that $\what{R}$ (in fact, any complete algebra) with a slightly different multiplication will turn out to be a monoid inside the monoidal category of complete modules with a suitable tensor product (see the Lemma \ref{rem:calg} below): in general, indeed, the ordinary tensor product $\otimes$ does not endow $\rmod{\K}^{\mathsf{c}}$ with a monoidal structure. In this way, we will be able to recover the definition of a complete algebra as a monoid in a monoidal category.

\begin{remark}\label{rem:reflimit}
It is well-known that the forgetful functor from a category of modules to the category of abelian groups creates, preserves and reflects limits (for the terminology see e.g. \cite[\S V.1]{MR1712872} and \cite[\S13]{adamek}). As a consequence, for a given filtered $(S,R)$-bimodule $M$, if we consider $\gamma_{\Sscript{S}}\otimes \gamma_{\Sscript{R^{\text{op}}}}:S\otimes R^{\text{op}}\to \what{S}\otimes \what{R}^{\text{op}}$ and if $\cR:\Bim{\what{S}}{\what{R}}\to \Bim{S}{R}$ is the restriction of scalars functor, then a projective cone $\tau:N\to\cD_{\Sscript{M}}$ in $\Bim{\what{S}}{\what{R}}$ on the functor $\cD_{\Sscript{M}}:\N\to \Bim{\what{S}}{\what{R}}$ mapping $n$ to $M/F_nM$ is a limiting cone of $\cD_{\Sscript{M}}$ if and only if $\cR(\tau):\cR(M)\to \cR\cD_{\Sscript{M}}$ is a limiting cone of $\cR\cD_{\Sscript{M}}$.
\end{remark}

\begin{lemma}\label{rem:invCan}
Given a filtered $(S,R)$-bimodule $M$, its completion $\what{M}$ is a complete $\left(\,\what{S},\what{R}\,\right)$-bimodule.
\end{lemma}

\begin{proof}
The completion $\what{M}$ of a filtered $(S,R)$-bimodule $M$ can be endowed with an $\left(\,\what{S},\what{R}\,\right)$-bimodule structure as follows. If $\left(M,F_nM\right)$ is a filtered $(S,R)$-bimodule, then for every $n\geq 0$ we have that $F_nS\cdot M\subseteq F_nM$ and $M\cdot F_nR\subseteq F_nM$, whence $M/F_nM$ is an $\left(S/F_nS,R/F_nR\right)$-bimodule and a filtered $\left(\,\what{S},\what{R}\,\right)$-bimodule via restriction of scalars through the canonical projections $p_n^S:\what{S}\to S/F_nS$ and $p_n^R:\what{R}\to R/F_nR$ respectively. In this way, the morphisms $\varfun{\pi^M_{n,m}}{M/F_nM}{M/F_mM}$ turn out to be $\left(\,\what{S},\what{R}\,\right)$-bilinear as well. It follows that an $\left(\,\what{S},\what{R}\,\right)$-bimodule structure is induced on the $(S,R)$-bimodule $\what{M}$  and it is explicitly given as follows: if $\what{r}_\infty\in\what{R}$, $\what{s}_\infty\in\what{S}$ and $\what{x}_\infty\in\what{M}$ then $\what{(s x r)}_{\infty}=\what{s}_\infty \cdot \what{x}_\infty \cdot \what{r}_\infty\in\what{M}$ corresponds to the $\N$-tuple $\left(s_n\cdot x_n\cdot r_n+ F_nM\right)_{n\geq 0}$.

In this way, the canonical projections $p_{m}:\what{M}\rightarrow {M}/{F_{m}M}$ are $\left(\,\what{S},\what{R}\,\right)$-bilinear, too. In particular, $\left(\what{M},F_n\what{M}\right)$ is a filtered $\left(\,\what{S},\what{R}\,\right)$-bimodule and the family of canonical projections $\what{M}\rightarrow {\what{M}}/{F_{\ast}\what{M}}$ is a projective cone in $\Bim{\what{S}}{\what{R}}$. Hence $\what{M} \cong \prlimit{n}{{\what{M}}/{F_n\what{M}}}$ as $\left(\,\what{S},\what{R}\,\right)$-bimodules as well in view of Remark \ref{rem:reflimit} and Lemma \ref{lemma:Mhatcomplete}.
This concludes the proof of the statement.
\end{proof}

In principle, we may consider on the one hand the full subcategory $\CBim{S}{R}$ of $\FBim{S}{R}$ given by complete $(S,R)$-bimodules. Objects are filtered $(S,R)$-bimodules $\left(M,F_nM\right)$ such that $M\cong \prlimit{n}{{M}/{F_nM}}$ as bimodules and arrows are filtered morphisms of complete bimodules $$\CHom{S,\,R}{M}{N}=\FHom{S,\,R}{M}{N}.$$
On the other hand, analogously, we may consider the full subcategory $\CBim{\what{S}}{\what{R}}$ of $\FBim{\what{S}}{\what{R}}$ given by complete $\left(\, \what{S},\what{R}\,\right)$-bimodules and filtered morphisms of complete bimodules.

\begin{proposition}\label{prop:catequiv}
We have an equivalence of categories between $\CBim{S}{R}$ and $\CBim{\what{S}}{\what{R}}$.
\end{proposition}

\begin{proof}
The proof is a consequence of Proposition \ref{prop:completion} and Remark \ref{rem:reflimit} together with Lemma \ref{rem:invCan}.
\end{proof}

The key role played by Proposition \ref{prop:catequiv} will be that of allowing us to work in both categories $\CBim{S}{R}$ and $\CBim{\what{S}}{\what{R}}$ indifferently, depending on our needs or on what we would like to stress, even if the algebras $S$ and $R$ are not themselves complete.

Denote by $\mathscr{U}:\CBim{S}{R} \rightarrow \FBim{S}{R}$ the functor that forgets the completeness, i.e., that associates to every complete $\left( S,R\right) $-bimodule its underlying filtered $\left( S,R\right)$-bimodule structure. What we showed in Remark \ref{rem:gammafiltiso} can be restated now by saying that if $M$ is complete, then $M\cong \what{\mathscr{U}(M)}$ is a filtered isomorphism.

The other way around, we have a functor
\begin{equation*}
\varfun{\what{\left( -\right) }}{\FBim{S}{R}}{\CBim{S}{R}}
\end{equation*}
that associates every filtered bimodule with its completion and every morphism of filtered bimodules $\varfun{f}{M}{N}$ with the morphism $\what{f}:=\prlimit{n}{\widetilde{f_n}}$, where $\varfun{\widetilde{f_n}}{{M}/{F_nM}}{{N}/{F_nN}}$ is the map induced on the quotients (see e.g. \cite[Chapter I, \S3]{MR1420862}).

\begin{remark}\label{rem:counit}
Using Proposition \ref{prop:completion} and Remark \ref{rem:gammafiltiso} one can check that $\what{(-)}$ is left adjoint to the forgetful functor $\mathscr{U}$, i.e., that we have a natural isomorphism 
\begin{equation}\label{eq:rightadj}
\FHom{S,\,R}{N}{\mathscr{U}(M)} \cong \CHom{S,\,R}{ \what{N}}{M}.
\end{equation}
The unit of this adjunction is the canonical map $\varfun{\gamma_{\Sscript{N}}}{N}{\mathscr{U}\left(\what{N}\right)}$ for all $N\in \FBim{S}{R}$; the counit is "its inverse" $\varfun{\sigma_{\Sscript{M}}}{\what{\mathscr{U}(M)}}{M}$ for all $M\in \CBim{S}{R}$. \footnote{One should notice that here some confusion may arise, as we denoted by $\varfun{\gamma_{\Sscript{M}}}{M}{\what{M}=\what{\mathscr{U}(M)}}$ also the canonical isomorphism in the category of complete bimodules whose actual inverse is $\sigma_M$. Cf. Remark \ref{rem:gamma}. Nevertheless, as we may embed $\CBim{S}{R}\left(M,\what{\mathscr{U}(M)}\right)\subseteq \FBim{S}{R}\left(\mathscr{U}(M),\mathscr{U}\left(\what{\mathscr{U}(M)}\right)\right)$ and $\mathscr{U}(\gamma_{\Sscript{M}})=\gamma_{\mathscr{U}(M)}$, we can identify the two morphism and it will be clear from the context which one we are referring to.}

Furthermore, we point out that the bijection in equation \eqref{eq:rightadj} encodes the universal property of the completion: every filtered morphism $\varfun{g}{N}{M}$ from a filtered $(S,R)$-bimodule $N$ to a complete $(S,R)$-bimodule $M$ factors through the completion of $N$, i.e., we have a commutative diagram of filtered morphisms
\begin{equation}\label{eq:univprophat}
\xymatrix@R=15pt@C=30pt{N \ar[rr]^{g} \ar[dr]_{\gamma_N} & & M \\
 & \what{N} \ar[ur]_{\what{g}} & }
\end{equation}
\end{remark}

Summing up, we obtained a commutative diagram
\begin{equation}\label{Eq:dash}
\xymatrix @R=5pt @C=30pt{  &  \CBim{S}{R}  \ar@{<->}[dd] \\
\FBim{S}{R} \ar@{->}^-{\what{(-)}}[ru]  \ar[rd]_-{\what{(-)}} &  \\
 & \CBim{\what{S}}{\what{R}} }
\end{equation}
where the vertical arrow is the equivalence of Proposition \ref{prop:catequiv}.
We end this subsection by recalling the following fact, which will be used in the forthcoming constructions.

\begin{lemma}
Given filtered algebras $S$, $R$ and $T$, for every complete $(R,T)$-bimodule $N$ the assignment
\begin{equation}\label{eq:homfunct}
\varfun{\CHom{-,T}{N}{-}}{\CBim{S}{T}}{\CBim{S}{R}}
\end{equation}
gives a well-defined functor.
\end{lemma}

\subsection{The topological tensor product of filtered bimodules}\label{ssc:TTPFB}

The main objective of this subsection is to show (or rather to recall) briefly that the functor
\begin{equation*}
\varfun{\CHom{S,\,R}{M}{\CHom{-,\,T}{N}{-}}}{\CBim{S}{T}}{\CBim{S}{R}}
\end{equation*}
is representable for every pair of complete bimodules ${}_{\Sscript{S}}M_{\Sscript{R}}$ and ${}_{\Sscript{R}}N_{\Sscript{T}}$. The representing object will be called \emph{the complete (or topological) tensor product} of $M$ and $N$. For a more exhaustive treatment of the construction of this tensor product over a single commutative ring we refer to \cite{MR0163908}.

\begin{theorem}\label{th:adjunction}
Let $R$, $S$ and $T$ be filtered algebras. For every complete bimodules ${}_SM_R$, ${}_RN_T$, ${}_SP_T$ we have  a filtered isomorphism, whence an homeomorphism as linear topological spaces,
\begin{equation*}
\CHom{S,\,R}{M}{\CHom{-,\,T}{N}{P}}\cong \CHom{S,\,T}{\what{M\tensor{R}N}}{P}
\end{equation*}
which is natural in $M$ and $P$ and where $\what{M\tensor{R}N}$ is the completion of the filtered tensor product $M\tensor{R}N$.
\end{theorem}
\begin{proof}
The usual Hom-tensor adjunction for bimodules tells us that we have a pair of natural isomorphisms of abelian groups:
\begin{gather*}
\varfun{\psi}{\hom{S,\,R}{M}{\hom{-,\,T}{N}{P}}}{\hom{S,\,T}{M\tensor{R}N}{P}}, \quad \Big(  f \longmapsto \big[ x\tensor{R}y \mapsto f(x)(y) \big] \Big) \\
\varfun{\phi}{\hom{S,\,T}{M\tensor{R}N}{P}}{\hom{S,R}{M}{\hom{-,\,T}{N}{P}}}, \quad \Big(  g \longmapsto \big[ x \mapsto [ y \mapsto g(x\tensor{R}y)] \big] \Big).
\end{gather*}
It is easy to see that if $f\in \CHom{S,\,R}{M}{\CHom{-,\,T}{N}{P}}$, then $\psi(f)$ is filtered.
Therefore we can further associate to $\psi(f)$ the (unique) morphism $\sigma_P\circ\what{\psi(f)}$ and the assignment
\begin{equation*}
f \longmapsto \Big[\varfun{\big(\sigma_P\circ\what{\psi(f)}\big)}{\what{M\tensor{R}N}}{ P}\Big]
\end{equation*}
leads to a well-defined map
\begin{equation*}
\varfun{\psiup}{\CHom{S,\,R}{M}{\CHom{-,\,T}{N}{P}}}{\CHom{S,\,T}{\what{M\tensor{R}N}}{P}}.
\end{equation*}
Explicitly, for all $f\in \CHom{S,\,R}{M}{\CHom{-,\,T}{N}{P}}$ and all $\what{(x\tensor{R}y)}_{\infty} \in  \what{M\tensor{R} N}$
$$
\psiup(f)\Big(  \what{(x\tensor{R}y)}_{\infty} \Big) \,=\, \what{f(x)(y)}_{\infty}\,=\,
\underset{k\to \infty}{\lim}\Big(f(x_{k})(y_{k})\Big),
$$
(finite summations in the tensor product are understood and notation \eqref{eq:fhat} is used).

To check that it is filtered, let $f\in F_n\CHom{S,\,R}{M}{\CHom{-,\,T}{N}{P}}=\CHom{S,\,R}{M}{\CHom{-,\,T}{N}{P}[n]}$.
For all $l\geq 0$, if $\what{(x\tensor{R}y)}_{\infty} \in  F_l\left(\what{M\tensor{R} N}\right)$, then we may assume that $x_k\tensor{R}y_k=0$ for every $k\leq l$ and that $x_k\tensor{R}y_k\in \cF_l\left(M\tensor{R}N\right)$ for every $k>l$. Thus, the Cauchy sequence $\{ f(x_k)(y_k)\}_{ k \geq 0}$ lies in $F_{l+n}P$ as well as its limit, since this is a closed subset  in $P$. Therefore, for all $n\geq 0$ we have
$$
\psiup\left(F_n\CHom{S,\,R}{M}{\CHom{-,\,T}{N}{P}}\right)\subseteq F_n\CHom{S,\,T}{\what{M\tensor{R}N}}{P}.
$$
The other way around, let us check that
\begin{equation*}
\varfun{\phiup}{\CHom{S,\,T}{\what{M\tensor{R}N}}{P}}{\CHom{S,\,R}{M}{\CHom{-,\,T}{N}{P}}}
\end{equation*}
which assigns to every $g\in \CHom{S,\,T}{\what{M\tensor{R}N}}{P}$ the morphism
\begin{equation*}
\sfun{\phiup(g)}{M}{\CHom{-,\,T}{N}{P}}{x}{\left[y\mapsto g\left(\gamma_{M\tensor{R} N}\left(x\tensor{R} y\right)\right)\right]},
\end{equation*}
is a well-defined filtered map. To this end, set $\gamma=\gamma_{M\tensor{R} N}$ and
observe that for every  $h,k,n\in\N$ and for all $g\in F_n\CHom{S,\,T}{\what{M\tensor{R}N}}{P}=\CHom{S,\,T}{\what{M\tensor{R}N}}{P[n]}$ we have that
\begin{equation*}
\left(\phiup(g)(F_hM)\right)(F_kN)\subseteq g\left(\gamma\left(\Img{F_hM\tensor{R} F_kN}\right)\right)\subseteq g\left(\gamma\left(\cF_{h+k}\left(M\tensor{R}N\right)\right)\right)\subseteq g\left(F_{h+k}\left(\what{M\tensor{R} N}\right)\right) \subseteq F_{h+k+n}P.
\end{equation*}
For all $g\in \CHom{S,\,T}{\what{M\tensor{R}N}}{P}$, this proves at once that $\left(\phiup(g)(M)\right)(F_kN)\subseteq F_kP$ (take $n=0=h$), whence $\phiup(g)$ lands into $\CHom{-,\,T}{N}{P}$, that $\phiup(g)(F_hM)\subseteq \CHom{-,\,T}{N}{P[h]}$ (take $n=0$), whence $\phiup(g)$ is filtered, and also that $\phiup$ itself is filtered as well.
Since  $\phiup$ and $\psiup$ are mutually inverse functions, they establish a  filtered isomorphism, as claimed. 
\end{proof}

\begin{definition}\label{def:CTP}
In view of Theorem \ref{th:adjunction}, we say that $\what{M\tensor{R}N}$ is the \emph{complete} (or \emph{topological}) tensor product over the filtered algebra ${R}$ of the complete bimodules ${_SM_R}$ and ${_RN_T}$ as indicated. It is coherent with the notion of completeness introduced in \S\ref{ssec:CF}. As a matter of notation, we will write $M~\cmptens{R}N:=\what{M\tensor{R}N}$.
\end{definition}

Furthermore, Theorem \ref{th:adjunction} can be restated by saying that, for $_RN_T$ complete, the functor
\begin{equation*}
\sfun{-\cmptens{{R}}N}{\CBim{S}{R}}{\CBim{S}{T}}{M}{M\cmptens{{R}}N}
\end{equation*}
is left adjoint to the functor
\begin{equation*}
\sfun{\CHom{-,T}{N}{-}}{\CBim{S}{T}}{\CBim{S}{R}}{P}{\CHom{-,T}{N}{P}}.
\end{equation*}

\begin{remark}  Following Theorem \ref{th:adjunction}, it is reasonable to call this complete tensor product a \emph{topological} tensor product as it is the left adjoint to the \emph{continuous} $\mathsf{Hom}$ functor between complete bimodules. We point out however that our definition of a topological tensor product satisfies a different universal property with respect to, e.g., \cite[Definition 2.1]{Seal} or \cite[Theorem 20.1.2]{Semadeni}. Namely, assume that ${_SM_R},{_RN_T}, {_SP_T}$ are complete bimodules over filtered algebras as indicated. Endow $M\times N$ with the filtration $F_k(M\times N)=F_kM\times F_kN$. The induced linear topology coincides with the product linear topology, i.e., the coarsest linear topology for which the canonical projections are continuous, and $M\times N$ is a complete $(S,T)$-bimodule with respect to this filtration.  The canonical morphism $M\times N\to M\otimes_{\Sscript{R}} N$ maps $F_k(M\times N)$ into $\Img{F_kM \tensor{R} F_kN}\subseteq \cF_{2k}(M\tensor{R} N)\subseteq \cF_k(M\tensor{R}N)$, whence it is filtered (and continuous) and the same hold for the composition $\tau:=\left(M\times N\to M\tensor{R} N\to M\cmptens{R}N\right)$. Endow $M\times N$ with the bi-filtration $F_{h,k}(M\times N)=F_hM\times F_kN$  (for the definition of a bi-filtration see, e.g., \cite[\S X.2]{Borel}). We observe that $\tau$ is bi-filtered\footnote{By a bi-filtered morphism we mean a morphism  $f:M\times N\to P$ such that $f(F_hM\times F_kN)\subseteq F_{h+k}P$. In particular, if $f$ is bi-filtered then $f(-,n):M\to P$ and $f(m,-):N\to P$ are filtered for all $m\in M$ and $n\in N$. Bi-filtered morphisms can be seen as a counterpart of \emph{separately continuous} functions (for an account on the subject we refer the reader to \cite{Piotrowski}).} as well. The bijective correspondence between $R$-balanced $(S,T)$-bilinear morphisms ${_SM_R}\times{_RN_T}\to {_SP_T}$ and morphisms in $\Hom{S,R}{M}{\Hom{-,T}{N}{P}}$ restricts to a bijective correspondence between $R$-balanced $(S,T)$-bilinear bi-filtered morphisms $M\times N\to P$ and elements in $\CHom{S,R}{M}{\CHom{-,T}{N}{P}}$. From this it follows that the complete tensor product could be considered as a topological tensor product in the sense that it satisfies the following universal property: there exists a complete $(S,T)$-bimodule and a bi-filtered $R$-balanced $(S,T)$-bilinear morphism $\tau:M\times N\to M\cmptens{R}N$ such that for every other complete $(S,T)$-bimodule $P$ and every bi-filtered $R$-balanced $(S,T)$-bilinear morphism $f:M\times N\to P$ there exists a unique filtered $(S,T)$-bilinear morphism $\widetilde{f}:M\cmptens{R}N\to P$ such that $f=\widetilde{f}\circ\tau$.
\end{remark}

Given $M,N$ two filtered $R$-bimodules over a filtered algebra $R$, we have three (in principle, different) ways to obtain a complete $\what{R}$-bimodule from $M\tensor{R}N$. The first and more natural one is $\what{M\tensor{R} N}$: since $M\tensor{R} N$ is a filtered $R$-bimodule, Lemma \ref{rem:invCan} ensures that $\what{M\tensor{R} N}$ is a complete $\what{R}$-bimodule.\footnote{Note that the writing $M\cmptens{R}N$ doesn't make sense in this context, unless both $M$ and $N$ are complete, whereas $\what{M\tensor{R} N}$ does.} The other two come from the construction we performed in this subsection. Namely, they are $\what{M}\cmptens{R}\what{N}$ and $\what{M}\cmptens{\what{R}}\what{N}$, i.e., the complete tensor product of the complete $\what{R}$-bimodules $\what{M}$, $\what{N}$ over the filtered algebras $R$ and $\what{R}$.
It turns out, however, that the three constructions give rise to the same complete bimodule up to the isomorphism of the following proposition (cf. also Remark \ref{rem:superiso}).

\begin{proposition}\label{prop:coherentcompl}
Assume that ${_{\Sscript{S}}M_{\Sscript{R}}}$ and ${_{\Sscript{R}}N_{\Sscript{T}}}$ are two filtered bimodules over filtered algebras as denoted. Then we have a filtered isomorphism of $(S,T)$-bimodules $\what{M\tensor{R}N}\cong \what{M}\cmptens{\what{R}}\what{N}$, natural in both variables,
explicitly given by
\begin{gather}
\varphi_{M,N}: \what{M}\cmptens{\what{R}}\what{N} \longrightarrow \what{M\tensor{R}N}, \quad \left[\, \limn\Big( \limk(x_{k,n} )\tensor{\what{R}}\liml(y_{l,n})\Big) \longmapsto \limn\big( x_{n,n}\tensor{R}y_{n,n} \big) \,\right], \label{Eq:Phi} \\
\psi_{M,N}:  \what{M\tensor{R}N} \longrightarrow \what{M}\cmptens{\what{R}}\what{N} , \quad \left[\, \limn\Big( x_{n} \tensor{R}y_{n}\Big) \longmapsto \limn\big( \what{x_{n}}\tensor{\what{R}}\what{y_{n}} \big) \,\right]. \label{Eq:Psi}
\end{gather}
\end{proposition}

\begin{remark}\label{rem:superiso}
Keeping assumptions and notations from Proposition \ref{prop:coherentcompl}, the algebra morphism $\gamma_R:R\to\what{R}$ induces a filtered $(S,T)$-bilinear morphism $\what{M}\tensor{R}\what{N}\to \what{M}\tensor{\what{R}}\what{N}$. Moreover, we can consider the filtered $(S,T)$-bilinear morphism $\gamma_M\tensor{R}\gamma_N:M\tensor{R}N\to \what{M}\tensor{R}\what{N}$. These induce the following composition
\begin{equation*}
\xymatrix @C=18pt{
\frac{M\tensor{R}N}{\cF_n\left(M\tensor{R}N\right)} \ar[rr]^-{\widetilde{\gamma_M\tensor{R}\gamma_N}} & & \frac{\what{M}\tensor{R}\what{N}}{\cF_n\left(\what{M}\tensor{R}\what{N}\right)} \ar[r] & \frac{\what{M}\tensor{\what{R}}\what{N}}{\cF_n\left(\what{M}\tensor{\what{R}}\what{N}\right)}
}
\end{equation*}
for all $n\geq 0$, which in turn induces exactly the morphism $\psi_{M,N}$ of the statement of Proposition \ref{prop:coherentcompl}.

Nevertheless, in what follows we will be concerned mainly with the complete tensor product over complete algebras. Thus, we decided to focus on the bare minimum to introduce the isomorphisms \eqref{Eq:Phi} and \eqref{Eq:Psi}. Furthermore, we will often omit these in the computations and we will identify $\what{M\tensor{R}N}$ with $\what{M}\cmptens{\what{R}}\what{N}$ as well, in order to simplify the exposition.
\end{remark}

As it happened for filtered algebras and bimodules, complete algebras and bimodules form a bicategory.

\begin{proposition}
We have a bicategory $\cB im_\K^{\mathsf{c}}$ which has complete algebras as $0$-cells and whose categories of $\{1,2\}$-cells are the categories of complete bimodules over complete algebras. The vertical compositions are given by the ordinary compositions of morphisms. The horizontal compositions are given by the composition functors $-\cmptens{B}-:=\what{(-)}\circ (-\tensor{B}-)$
$$-\cmptens{B}-:\CBim{A}{B}\times \CBim{B}{C}\to \CBim{A}{C}$$
for all complete algebras $A,B,C$. The constraints are induced by those of the bicategory $\cB im_\K^{\mathsf{flt}}$.
\end{proposition}

\begin{proof}
In view of Proposition \ref{prop:catequiv}, the natural isomorphisms $\varphi_{-,-}$ and $\psi_{-,-}$ described in equations \eqref{Eq:Phi} and \eqref{Eq:Psi} can be regarded as isomorphisms of complete $\Big(\,\what{S}, \what{T}\,\Big)$-bimodules and for this reason we are going to denote them in the same way.
The left and right unit constraints (or identities, as they are called in \cite{Benabou}) are  deduced by using the natural isomorphisms $\varphi_{A,-}$ and $\varphi_{-,B}$ in conjunction with the natural isomorphism $\varfun{\sigma_{\Sscript{X}}}{\what{\mathscr{U}(X)}}{X}$ of  Remark \ref{rem:gamma} for $X$ an object in $\CBim{{A}}{{B}}$ and $A,B$ complete algebras (cf.~also Remark \ref{rem:counit}). The associativity constraint $\varfun{\alpha_{\Sscript{M,\,N,\,P}}}{\left(M\cmptens{\what{R}}N\right)\cmptens{\what{R}}P}{M\cmptens{\what{R}}\left(N\cmptens{\what{R}}P\right)}$ is obtained by observing that  for every ${}_{\Sscript{S}}X_{\Sscript{R}}$, ${}_{\Sscript{R}}Y_{\Sscript{T}}$ and ${}_{\Sscript{T}}Z_{\Sscript{K}}$ filtered bimodules, there is a unique $\Big(\,\what{S}, \what{K}\,\Big)$-bilinear map $\alpha_{\Sscript{\what{X},\,\what{Y},\,\what{Z}}}$ making commutative the following diagram

\begin{equation}\label{eq:defalpha}
\xymatrix @C=60pt @R=15pt{
\left(\what{X}\cmptens{\what{R}}\what{Y}\right)\cmptens{\what{T}}\what{Z} \ar@{.>}[d]_-{\alpha_{\what{X},\what{Y},\what{Z}}}  & \what{X\tensor{R} Y}\cmptens{\what{T}}\what{Z} \ar[l]_-{\psi_{X,\,Y}\cmptens{\what{T}}\what{Z}} & \what{\left(X\tensor{R}Y\right)\tensor{T}Z} \ar[d]^-{\what{a_{X,\,Y,\,Z}}} \ar[l]_-{\psi_{X\tensor{R}Y,\,Z}} \\
\what{X}\cmptens{\what{R}}\left(\what{Y}\cmptens{\what{T}}\what{Z}\right) & \what{X}\cmptens{\what{R}}\what{Y\tensor{T}Z} \ar[l]_-{\what{X}\cmptens{\what{R}}\psi_{Y,\,Z}} & \what{X\tensor{R}\left(Y\tensor{T}Z\right)} \ar[l]_-{\psi_{X,\,Y\tensor{T}Z}}
}
\end{equation}
where $a_{X,\,Y,\,Z}$ is the usual associativity constraint.
\end{proof}

It turns out then that the completion functor fits properly in the wider framework of bicategories.

\begin{theorem}\label{thm:Athm}
Let $\K$ be a commutative ground ring which we consider trivially filtered. Then the completion construction developed in this section induces a  2-functor
$$
\xymatrix@R=0pt{ & \cB im_{\Sscript{\K}}^\mathsf{flt}  \ar@{->}[rr] & &  \cB im_{\Sscript{\K}}^\mathsf{c} \\ \rm{0\text{-}cells} & R \ar@{|->}[rr] & & \what{R} \\  \rm{1\text{-}cells} & {}_{\Sscript{R}}M_{\Sscript{S}} \ar@{|->}[rr] & & {}_{\Sscript{\what{R}}}\what{M}_{\Sscript{\what{S}}} \\ \rm{2\text{-}cells} & \Big[f: M \to N \Big] \ar@{|->}[rr] & & \Big[\what{f}: \what{M} \to \what{N} \Big] }
$$
from the bicategory $\cB im_{\Sscript{\K}}^\mathsf{flt}$ of filtered algebras and filtered bimodules to the bicategory $\cB im_{\Sscript{\K}}^\mathsf{c}$ of complete algebras and complete bimodules.
\end{theorem}
\begin{proof}
The construction of the stated 2-functor at the level of 0-cells is clear. At the level of  $\{\text{1,2}\}$-cells, the needed family of functors is given by the functors exhibited in diagram \eqref{Eq:dash}, precisely by the lower diagonal one. The required natural transformations for $\what{(-)}$ are given in Proposition \ref{prop:coherentcompl}.
Finally, the coherence axioms (i.e., the hexagons and the squares in \cite[Definition 4.1]{Benabou}) are fulfilled by construction.
\end{proof}

\begin{corollary}\label{prop:moncatcomplbimod}
Let $R$ be a filtered algebra. Then the category of complete $\what{R}$-bimodules $\CBim{\what{R}}{\what{R}}$ is monoidal with tensor product the topological tensor product $-\cmptens{\what{R}}-$ and with unit the completion algebra $\what{R}$ of $R$. Moreover, the completion functor $\varfun{\what{\left(-\right)}}{\FBim{R}{R}}{\CBim{\what{R}}{\what{R}}}$ is a monoidal functor.
\end{corollary}

\begin{remark}\label{rem:calg}
It follows from Corollary \ref{prop:moncatcomplbimod} that $\left(\rmod{\K}^{\mathsf{c}},\what{\otimes},\K\right)$ is a monoidal category.  It can be checked that for a filtered algebra $(R,\mu,\eta)$, $R$ is a complete module (i.e., a complete algebra) if and only if $\left(R,\sigma_{\Sscript{R}}\,\what{\mu},\eta\right)$ is a monoid in the monoidal category $\rmod{\K}^{\mathsf{c}}$. A similar thing happens for complete bimodules over complete algebras.
In fact, up to an equivalence of categories, we may regard complete algebras as monoids in the monoidal category $\rmod{\K}^{\mathsf{c}}$ and complete bimodules over complete algebras $A$ and $B$ as objects in ${{}_{\Sscript{A}}\left(\rmod{\K}^{\mathsf{c}}\right){}_{\Sscript{B}}}$ and conversely (in accordance with \cite[\S A.1]{MR1320989}, for example). Furthermore, we point out that it could possible to deduce Theorem \ref{thm:Athm} from a more general framework as claimed in \cite[Examples 2.2 and 6.2]{shulman}, once proven that $\left(\rmod{\K}^{\mathsf{flt}},\otimes,\K\right)$ and $\left(\rmod{\K}^{\mathsf{c}},\what{\otimes},\K\right)$ are monoidal categories and that $\what{(-)}:\rmod{\K}^{\mathsf{flt}}\to \rmod{\K}^{\mathsf{c}}$ is a monoidal functor.
\end{remark}

\section{Topological tensor product of linear duals of locally finitely generated and projective filtered bimodules}\label{sec:TGFr}

In this appendix we plan to study the linear dual of the tensor product of two locally finitely generated and projective filtered modules (for instance, rings with an admissible filtration as in \S\ref{ssec:FUstra}). In particular, we will show that this bimodule is homeomorphic to the topological tensor product of the duals.

\subsection{Locally finitely generated and projective filtered modules}\label{ssec:LFGr}
Let $R$ be a ring and $M$ a right $R$-module endowed with an \emph{ascending filtration} $\left\{F^nM\mid n\in\N\right\}$. This is said to be \emph{exhaustive} if $\bigcup_{n\geq 0}F^nM=M$. In view of our aims, we assume $R$ trivially filtered. We denote by $\gr^n\left(M\right)$ the quotient module $F^nM/F^{n-1}M$ for all $n\geq 0$ ($F^{-1}M=0$ by convention), and by $\gr\left(M\right)$ the associated graded module $\gr\left(M\right)=\bigoplus_{n\geq 0}\gr^n\left(M\right)$. Henceforth and in line with Appendix \ref{sec:CBCF}, we denote increasing filtrations with upper indices and decreasing ones with lower indices. Moreover, $\tau_{\Sscript{m,\,n}}:F^nM\to F^mM$ and $\tau_{\Sscript{n}}:F^nM\to M$ for all $m\geq n\geq 0$ will denote the canonical inclusions.

\begin{lemma}\label{lemma:fgpquotients}
Let $R$ be any ring, $M$ a right $R$-module endowed with an ascending filtration $\left\{F^kM\mid k\in\N\right\}$ and let $n\in\N$. If the quotient modules $F^kM/F^{k-1}M$ are projective right $R$-modules for all $0\leq k\leq n$, then $F^nM\cong \gr\left(F^nM\right)$ as filtered modules. In particular, $F^nM$ is projective. If moreover the quotient modules $F^kM/F^{k-1}M$ are finitely generated for $0\leq k\leq n$, then $F^nM$ is finitely generated as well. Finally, if the filtration is exhaustive and the quotient modules $F^nM/F^{n-1}M$ are projective for all $n\in\N$, then there exists an isomorphism of filtered modules $M\cong \gr(M)$ and $M_{\Sscript{R}}$ itself is projective.
\end{lemma}
\begin{proof}
Since every quotient module $F^kM/F^{k-1}M$ is projective as right $R$-module, for all $0\leq k\leq n$, we have a split exact sequence of right $R$-modules
\begin{equation*}
\xymatrix{
0 \ar[r] & F^{n-1}M \ar@<+0.5ex>[r]^-{\tau_{\Sscript{n-1,\,n}}} & F^{n}M \ar@<+0.5ex>[r]^-{} \ar@{.>}@<+0.5ex>[l]^-{} & \left(F^nM/F^{n-1}M\right) \ar[r]\ar@{.>}@<+0.5ex>[l]^-{} & 0
}
\end{equation*}
from which it follows that, as right $R$-modules,
\begin{equation*}
F^nM\cong F^{n-1}M\oplus \left(F^nM/F^{n-1}M\right).
\end{equation*}
Proceeding inductively, we have that
\begin{equation}\label{Eq:isoFn}
F^nM\cong \bigoplus_{k=0}^n \frac{F^kM}{F^{k-1}M}=\gr\left(F^nM\right).
\end{equation}
Observing that for all $m\leq n$, $F^m\gr\left(F^nM\right)=\bigoplus_{k=0}^m F^kM/F^{k-1}M=\gr\left(F^mM\right)$ and $F^mF^nM=F^mM$, it is clear that the isomorphism preserves the filtrations as claimed. Moreover, as direct sum of projective right $R$-modules, $F^nM$ is projective as well.

The second claim is clear, as the direct sum is finite. About the last claim in the statement, saying that the filtration is exhaustive means that $M\cong \injlimit{n}{F^nM}$ as filtered modules. Since $F^nM\cong\gr\left(F^nM\right)\cong F^n\left(\gr(M)\right)$ as filtered modules, we have that $M\cong\injlimit{n}{F^nM}\cong \injlimit{n}{F^n\left(\gr(M)\right)}\cong\gr(M)$ as claimed.
As direct sum of projective right $R$-modules, $M$ is itself projective.
\end{proof}

Henceforth, all ascending filtrations will be exhaustive. In analogy with \cite[\S4]{Cartan:1958}, we will say that an increasingly filtered right $R$-module $M$ such that the quotient modules $F^nM/F^{n-1}M$ are finitely generated and projective is a \emph{locally finitely generated and projective} (filtered) module. 

\subsection{The  topology on the linear dual of a locally finitely generated and projective filtered bimodule}\label{ssec:DLF}

Assume that we are given an increasingly filtered $R$-bimodule $M$ which is locally finitely generated and projective as a filtered right $R$-module (the definition of an increasingly filtered bimodule can be easily obtained by dualizing that for decreasingly filtered bimodules in Appendix \ref{sec:CBCF}). In particular, this means that each member of the increasing filtration $\{F^nM\mid n\in\N\}$ is actually an $R$-subbimodule with a monomorphism $\taun:F^nM \to M$ and that the factors $F^nM/F^{n-1}M$ are finitely generated and projective right $R$-modules.

Since the filtration $\left\{F^nM\mid n\in\N\right\}$ is exhaustive, we may identify the right $R$-module ${M_{\Sscript{R}}}$ with the inductive limit $M = \injlimit{n}{F^nM}$ of the system $\left\{F^nM,\tau_{\Sscript{n,\,n+1}}\right\}_{n\in\N}$. Therefore, $M^*=\Hom{-,A}{M}{A}\cong \prlimit{n}{F^nM^*}$ as a left $R$-module via the left $R$-linear isomorphism
\begin{equation}\label{Eq:D}
M^* \rightarrow \prlimit{n}{F^nM^*}, \quad \Big( f \mapsto (\taun^*(f))_{\Sscript{n \geq 0}}\Big); \qquad \prlimit{n}{F^nM^*} \rightarrow M^*, \quad \Big( (g_{\Sscript{n}})_{\Sscript{n \geq 0}} \mapsto g:=\injlimit{n}{g_{\Sscript{n}}} \Big)
\end{equation}
where $(r\cdot f)(x)=rf(x)$ for all $f\in M^*$, $r\in R$ and $x\in M$. However, $M^*$ is also a right $R$-module with $\left(f\leftharpoonup x\right)(m)=f(x\cdot m)$ for all $f\in M^*$, $m\in M$ and $x\in R$, and it turns out that the isomorphism \eqref{Eq:D} is right $R$-linear as well. Therefore, $M^*\cong \prlimit{n}{F^nM^*}$ as $R$-bimodules.
Notice that $g:M\to R$ is the unique right $R$-linear map that extends all the $g_{\Sscript{n}}$'s at the same time, that is $g \, \tau_{\Sscript{n}}= g_{\Sscript{n}}$ for all $n\geq 0$.

\begin{corollary}\label{coro:FnL}
Let $M$ be an increasingly filtered $R$-bimodule which is locally finitely generated and projective as right $R$-module. The following properties hold true.
\begin{enumerate}[label=(\roman*),leftmargin=1cm]
\item Each of the subbimodules  $F^nM$ is a finitely generated and projective right $R$-module and each of the structural maps $\tau_{\Sscript{n,\,n+1}}: F^nM \to F^{n+1}M$ is a split monomorphism of right $R$-modules. Moreover, the transposes $\tau_{\Sscript{n}}^*:M^*\to F^nM^*$ are surjective, so that for all $n\in\N$, $\tau_{\Sscript{n}}$ is a split monomorphism too.
\item For every $m,n \geq 0$, we have an isomorphism of $R$-bimodules
\begin{equation*}
\phi_{\Sscript{m,\, n}}: \,\left(F^mM^*\right)^{}_{\Sscript{R}}\tensor{R} {^{}_{\Sscript{R}}\left(F^nM^*\right)} \cong \left(F^nM_{\Sscript{R}}\tensor{R} {{}_{\Sscript{R}}F^mM}\right)^*
\end{equation*}
such that $\phi_{\Sscript{m,\, n}}\left(f\tensor{R}g\right)\left(x\tensor{R}y\right)=f\left(g(x)y\right)$ for all $x \in F^nM$, $y\in F^mM$, $f\in F^mM^*$ and $g \in F^nM^*$.
\end{enumerate}
\end{corollary}

\begin{proof}
The first claim of $(i)$ follows directly from Lemma \ref{lemma:fgpquotients}. To prove the second one, we proceed as follows. If $\left\{S_i,\varphi_{j,i}\mid i,j\in\mathbb{N},j\geq i\right\}$ is an inverse system in the category of left $R$-modules with surjective transition maps $\varphi_{j,i}:S_j\to S_i$, $j\geq i$, then every projection $\varphi_i:\prlimit{n}{S_n}\to S_i$ is surjective as well (cf. e.g. \cite[Remark 2.14]{Gh}). Since the transition maps $\tau_{n,n+1}$ are split monomorphisms, their transposes $\tau_{n,n+1}^*$ are surjective, whence the canonical maps $\varphi_i:\prlimit{n}{(F^nM)^*}\to (F^iM)^*$ are surjective as well. If we denote by $\Phi$ the isomorphism of \eqref{Eq:D}, then it satisfies $\varphi_n\circ\Phi=\tau_n^*$, whence $\tau_n^*$ is surjective.

Finally, in view of \cite[Lemma 11.3]{BSZ} and the hom-tensor adjunction respectively, we have the chain of isomorphisms of $R$-bimodules
\begin{equation*}
\left(F^mM\right)^*{}_{\Sscript{R}}\tensor{R} {_{\Sscript{R}}\left(F^nM\right)}^* \cong \rhom{R}{F^nM_{\Sscript{R}}}{\left(F^mM\right)_{\Sscript{R}}^*}\cong \rhom{R}{F^nM_{\Sscript{R}}\tensor{R}{}_{\Sscript{R}}F^mM}{R}
\end{equation*}
which proves $(ii)$.
\end{proof}

\begin{remark}\label{rem:vartheta}
Since $\taun^*$ is surjective and $F^nM^*$ is a finitely generated and projective left $R$-module, there is a left $R$-linear section $F^nM^*\to M^*$ of $\taun^*$ which induces a right $R$-linear retraction $\thetan: M \to F^nM$ of $\taun$. In particular, each of the maps $\taun^*: M^* \to F^nM^*$ is a split epimorphism of left $R$-modules with section $\thetan^*: F^nM^* \to M^*$ as well. Denote temporarily by $\pi_{\Sscript{n}}:M^*\to M^*/\ker{\taun^*}$ the canonical projection. Even if $\thetan^*$ is just left $R$-linear, the composition $\pi_{\Sscript{n}}\circ \thetan^* : F^nM^*\to M^*/\ker{\taun^*}$ is $R$-bilinear as it is the inverse of the $R$-bilinear isomorphism $\widetilde{\taun^*}:M^*/\ker{\taun^*}\to F^nM^*$.
\end{remark}

Now, the right linear dual $M^*$ inherits naturally a decreasing filtration which converts it into a complete $R$-bimodule. Namely, mimicking \cite[Appendix A.2]{MSS}, let us consider the filtration
\begin{equation}\label{eq:filtdual}
F_0M^*=M^* \quad \text{and} \quad F_{n+1}M^*=\ker{\tau_{\Sscript{n}}^*}, \quad \text{for} \,\, n \geq 0. \,\footnote{Observe that $\ker{\tau_{\Sscript{n}}^*}=\left\{f\in M^*\mid F^nM\subseteq \ker{f}\right\}$, whence we will often use the notation $\ann{F^nM}$ to refer to it.} 
\end{equation}
Notice that no confusion may arise in the notation, as the upper or lower indices help in distinguishing between $F_nM^*$, the $n$-th term of the decreasing filtration on $M^*$, and $F^nM^*:=(F^{n}M)^{*}$, the dual of the $n$-th term of the increasing filtration on $M$. In view of (i) of Corollary \ref{coro:FnL}, we have an isomorphism of $R$-bimodules $F^nM^*\cong M^*/F_nM^*$. From this together with the isomorphism \eqref{Eq:D} and Proposition \ref{prop:completion} we deduce that the filtration $\{F_nM^*\mid n \in\N\}$ induces a linear topology over $M^*$ for which it is a complete $R$-bimodule.

\begin{remark}
In order to be able to evaluate limits of Cauchy sequences in $M^*$ on an element of $M$ it is useful to notice the following. Let $\{f_n\}_{\Sscript{n \geq 0}}$ be a Cauchy sequence of  right $R$-linear maps  in $M^*$ and let $f = \limn(f_n)$ denote its limit in $M^*$. Therefore we have that $f-f_n\in F_nM^*=\ker{\tau_{n-1}^*}$ for every $n\geq 1$. For all $x\in M$, there exists an $l\geq0$ such that $x\in F^lM$ and hence for every $k\geq l+1$ we have that
$$f_{k}(x)=f_{k}(\tau_l(x))=f(\tau_l(x))=f(x).$$
This means that the sequence of elements $\left\{f_n(x)\right\}_{n\geq0}$ eventually becomes constant in $A$ and equal to the value of $f$ on $x$. Thus, it is meaningful to set $f(x)=\left(\limn(f_n)\right)(x):=\limn(f_n(x))$.

On the other hand, notice that we may consider the inductive limit function of the inductive cone $\{\tau_n^*(f_{n+1})\}_{n \, \in \, \mathbb{N}}$. However,  $\injlimit{n}{\tau_n^*(f_{n+1})} = \injlimit{n}{\tau_n^*(f)}=f=\limn(f_n)$.
\end{remark}

\subsection{The topological tensor product and the associativity constraint}\label{ssec:TTAC}
It is useful  to recall that the full subcategory of $R$-bimodules which are locally finitely generated and projective on the right is closed under taking tensor products (compare with  \cite[Theorem C.24, page 93]{Majewski}). Indeed, let $M,N$ be filtered $R$-bimodules which are locally finitely generated and projective on the right. Then we have an $R$-bilinear isomorphism
\begin{multline*}
\bigoplus_{p+q=n}\frac{F^pM}{F^{p-1}M}\tensor{R}\frac{F^qN}{F^{q-1}N} \, \longrightarrow \,\frac{\cF^n(M\tensor{R}N)}{\cF^{n-1}(M\tensor{R}N)}, \quad  \\ \bigg((x_p+F^{p-1}M)\tensor{R}(y_q+F^{q-1}N)\mapsto (x_p\tensor{R}y_q)+\cF^{n-1}(M\tensor{R}N)\bigg)
\end{multline*}
where $\cF^n(M\tensor{R}N)=\sum_{p+q=n}F^pM\tensor{R}F^qN$. Thus the factors $\cF^n(M\tensor{R}N)/\cF^{n-1}(M\tensor{R}N)$ are finitely generated and projective as right $R$-modules. Therefore, $M\tensor{R}N$ is locally finitely generated and projective as claimed.

Next, we want to compare the topology that the linear dual $(N\tensor{R}M)^*$ inherits from the structure of locally finitely generated and projective module,  with that of $\what{M^*\tensor{R}N^*}=M^*\cmptens{R}N^*$, the topological tensor product of the complete $R$-bimodules $M^*$ and $N^*$.  At the algebraic level,  we have a canonical $R$-bilinear map
\begin{equation}\label{Eq:eva}
\xymatrix@R=0pt{ (M^*)_{\Sscript{R}} \tensor{R}  {_{\Sscript{R}}(N^*)} \ar@{->}^-{\phi_{\Sscript{M,N}}}[rr] & &   (N_{\Sscript{R}}\tensor{R}{_{\Sscript{R}}M})^* \\ f\tensor{R}g \ar@{|->}^-{}[rr] & & \left[ y\tensor{R}x \longmapsto f(g(y)x) \right] }
\end{equation}
which makes the following diagram to commute
\begin{equation}\label{eq:phifilt}
\xymatrix@R=25pt @C=30pt{ M^*\tensor{R}  N^* \ar@{->}^-{\phi_{\Sscript{M,N}}}[r] \ar@{->}_-{ \left(\ntau{m}{M} \right)^*\tensor{R} \left(\ntau{n}{N} \right)^*}[d]  &   \left(N\tensor{R}M\right)^*  \ar@{->}^-{ \left(\ntau{n}{N}\tensor{R}\ntau{m}{M} \right)^*}[d]  \\  F^mM^*\tensor{R}F^nN^* \ar@{->}^-{\phi_{\Sscript{m,n}}}[r]  & \left(  F^nN\tensor{R}F^mM \right)^*.    }
\end{equation}

In view of the technical subsequent Lemma \ref{lemma:filtrintersection}, it turns out that the natural transformation $\phi_{\Sscript{M,\, N}}$ is a continuous map (in fact a morphism of filtered $R$-bimodules) where $\cF_n\left(M^*\tensor{R}  N^*\right)=\sum_{p+q=n}\Img{F_pM^*\tensor{R}F_qN^*}$ and the decreasing filtration on $(N\tensor{R}M)^*$ is given as in \eqref{eq:filtdual}, that is,  $F_0(N\tensor{R}M)^*=(N\tensor{R}M)^*$  and $F_n(N\tensor{R}M)^*=\ker{\boldsymbol{\tau}_{\Sscript{n-1}}^*}$, $n \geq 1$, where  $\boldsymbol{\tau}_{\Sscript{n}}^*: (N\tensor{R}M)^* \to \cF^n(N\tensor{R}M)^*$  are the canonical projections.

\begin{lemma}\label{lemma:filtrintersection}
Let $R$ be a ring and $V,W$ be decreasingly filtered $R$-bimodules such that $W/F_nW$ are finitely generated and projective as left $R$-modules for all $n\in\N$. Then
\begin{equation*}
\cF_{n}\left(V\tensor{R} W\right) :\,=\, \sum_{p+q=n}F_{p}V \tensor{R} F_{q}W\,=\, \bigcap_{p+q=n+1}\ker{\pi_{\Sscript{p}}^{\Sscript{V}}\tensor{R} \pi_{\Sscript{q}}^{\Sscript{W}}}
\end{equation*}
where $\pi_{\Sscript{p}}^{\Sscript{V}}\colon V\to V/F_pV$ and $\pi_{\Sscript{q}}^{\Sscript{W}}\colon W\to W/F_qW$ are the canonical projections. In particular, for $M$ and $N$ $R$-bimodules such that $N$ is locally finitely generated and projective on the right, we have
$$\cF_n(M^*\tensor{R}N^*)=\bigcap_{p+q=n-1}\ker{\tau_p^*\tensor{R}\tau_q^*}.$$
\end{lemma}

The following proposition gives the desired comparison between  the linear topologies on the filtered  bimodules $\what{M^*\tensor{R}N^*}={M^*\cmptens{R}N^*}$ and  $(N\tensor{R}M)^*$.

\begin{proposition}\label{prop:MRN}
Let $M$ and $N$ be two $R$-bimodules, locally finitely generated and projective  as right $R$-modules. Then the natural transformation $\phi_{\Sscript{M,N}}$ of equation \eqref{Eq:eva} induces an homeomorphism (in fact, a filtered isomorphism) ${M^*\cmptens{R}N^*} \cong (N\tensor{R}M)^*$ such that the following diagram is commutative:
$$
\xymatrix@R=20pt@C=45pt{ \left(M^*\right)^{}_{\Sscript{R}}\tensor{R}{^{}_{\Sscript{R}}\left(N^*\right)} \ar@{->}^-{\phi_{\Sscript{M,N}}}[rr]  \ar@{->}_-{\gamma_{\Sscript{M^*\tensor{R}N^*}}}[rd]  & & (N_{\Sscript{R}}\tensor{R}{_{\Sscript{R}}M})^*   \\ & {\left(M^*\right)^{}_{\Sscript{R}}\cmptens{R}{^{}_{\Sscript{R}}\left(N^*\right)}}   \ar@{-->}_-{\what{\phi_{\Sscript{M,N}}}}^-{\cong}[ru]  & }
$$
\end{proposition}
\begin{proof}
We know that $N\tensor{R}M$ is a locally finitely generated and projective right $R$-module, whence $(N\tensor{R}M)^*$ is a complete $R$-bimodule with respect to the filtration $\mathsf{Ann}(F^k(N\tensor{R}M))=F_{k+1}(N\tensor{R}M)^*$ (see \S\ref{ssec:DLF}).  In view of \eqref{eq:phifilt}, for all $m+n=k$ we have that
\begin{equation*}
\left(\ntau{m}{M}\tensor{R}\ntau{n}{N}\right)^*\left(\phi_{M,N}\left(\cF_{k+1}\left(M^*\tensor{R}N^*\right)\right)\right)=\phi_{m,n}\left(\left(\left(\ntau{m}{M}\right)^*\tensor{R}\left(\ntau{n}{N}\right)^*\right)\left(\cF_{k+1}\left(M^*\tensor{R}N^*\right)\right)\right)=0.
\end{equation*}
In particular,  there exists a unique $R$-bilinear morphism $\sigma_{\Sscript{m,n}}:{M^*\tensor{R}N^*}/{\cF_{k+1}\left(M^*\tensor{R}N^*\right)} \to (  F^nN\tensor{R}F^mM )^*$ such that $\sigma_{\Sscript{m,n}}\,\greekn{\pi}{k+1}{M^*\tensor{R}  N^*}=\left(\ntau{n}{N}\tensor{R}\ntau{m}{M}\right)^*\,\phi_{\Sscript{M,N}}$.
Notice that the completion $\what{\phi_{\Sscript{M,N}}}$ of the filtered morphism $\phi_{M,N}$ fits into the following commutative diagram
\begin{equation}\label{eq:diagrams}
\xymatrix@R=20pt@C=30pt{ M^*\cmptens{R}  N^* \ar@{->}^-{\what{\phi_{\Sscript{M,N}}}}[r] \ar@{->}_-{\mathfrak{p}_{k+1}}[d]  &   (N\tensor{R}M)^*  \ar@{->}^-{\left(\ntau{n}{N}\tensor{R}\ntau{m}{M}\right)^*}[d]  \\  \frac{M^*\tensor{R}N^*}{\cF_{k+1}\left(M^*\tensor{R}N^*\right)} \ar^-{\sigma_{\Sscript{m,n}}}[r]  & \Big(  F^nN\tensor{R}F^mM \Big)^*    }
\end{equation}

Our next aim is to construct explicitly a filtered inverse for $\what{\phi_{\Sscript{M,N}}}$. Set $\gamma:=\gamma_{\Sscript{M^*\tensor{R}N^*}}: M^{*}\tensor{R}N^{*} \to \what{M^{*}\tensor{R}N^{*}}$. For all $m+n=k$ consider the composition $\Pi_{m,n}:=\phi^{-1}_{m,n}\circ \sigma_{m,n}\circ \mathfrak{p}_{k+1}$ which gives an $R$-bilinear morphism $\Pi_{m,n}:M^*\cmptens{R}N^*\to F^mM^*\tensor{R}F^nN^*$. It satisfies $\Pi_{m,n}\circ\gamma=\left(\ntau{m}{M}\right)^*\tensor{R}\left(\ntau{n}{N}\right)^*$ and
$$
F_{k+1}\left(M^*\cmptens{R}N^*\right):=\ker{\mathfrak{p}_{k+1}}=\bigcap_{m+n=k}\ker{\Pi_{m,n}}.
$$
Now, by considering the $R$-bilinear maps $\phi^{-1}_{\Sscript{m,n}}\, \sigma_{\Sscript{m,n}}:M^*\tensor{R}N^*/\cF_{k+1}\left(M^*\tensor{R}N^*\right) \to F^mM^*\tensor{R}F^nN^* $ and
\begin{equation}\label{eq:}
\xi_{m,n}:F^mM^*\tensor{R}F^nN^*\to \frac{M^*\tensor{R}N^*}{\cF_{h+1}\left(M^*\tensor{R}N^*\right)}; \qquad \left(f\tensor{R}g\mapsto \theta_{\Sscript{m}}^*(f)\tensor{R}\thetan^*(g)+\cF_{h+1}\left(M^*\tensor{R}N^*\right)\right), \, \footnote{Notice that we cannot perform the tensor product $\theta_{\Sscript{m}}^*\tensor{R}\thetan^*$ as the maps $\thetan^*$ are just left $R$-linear. Nevertheless, the stated morphism is well-defined. It is a consequence of Remark \ref{rem:vartheta} and of the fact that $\ker{\left(\ntau{p}{M}\right)^*\tensor{R}\left(\ntau{q}{N}\right)^*} \subseteq \cF_{m+1}\left(M^*\tensor{R}N^*\right)$,  $m=\min(p,q)$.}
\end{equation}
for all $m+n=k$ and $h=\min(m,n)$, one can show that $M^*\cmptens{R}  N^*$ together with the family of morphisms $\left\{\Pi_{\Sscript{m,n}}\mid n,m\geq0\right\}$ is isomorphic to the inverse limit of the projective system $F^mM^*\tensor{R}F^nN^*$ with structure maps $\left(\ntau{p,m}{M}\right)^* \tensor{R}\left(\ntau{q,n}{N}\right)^*$ for all $p\leq m$ and $q\leq n$. Now, by definition of $\Pi_{\Sscript{m,n}}$ we have that
\begin{equation}\label{eq:phiPi}
\phi_{\Sscript{m,n}}\circ \Pi_{\Sscript{m,n}}\stackrel{(\text{def})}{=}\sigma_{\Sscript{m,n}}\circ \mathfrak{p}_{k+1}\stackrel{\eqref{eq:diagrams}}{=}\left(\ntau{n}{N}\tensor{R}\ntau{m}{M}\right)^*\circ \what{\phi_{\Sscript{M,N}}},
\end{equation}
whence $\what{\phi_{\Sscript{M,N}}}$ is also the unique morphism induced by the map of projective systems $\phi_{\Sscript{m,n}}$. By considering $\phi_{\Sscript{m,n}}^{-1}$ instead, one deduces that there exists a unique morphism $\psi_{\Sscript{M,N}}:\left(N\tensor{R} M\right)^*\to M^*\cmptens{R}  N^*$ such that $\Pi_{\Sscript{m,n}}\circ \psi_{\Sscript{M,N}}=\phi_{\Sscript{m,n}}^{-1} \circ \left(\ntau{n}{N}\tensor{R}\ntau{m}{M}\right)^*$. It is not difficult now to see that $\what{\phi_{\Sscript{M,N}}}$ and $\psi_{\Sscript{M,N}}$ are filtered morphisms which are mutually inverses.
\end{proof}

\begin{remark}
Given $z\in  M^*\cmptens{R}  N^*$, we already know that $z=\limn{\left(\mathfrak{p}_n(z)\right)}$, up to a choice of a representative in $M^*\tensor{R}  N^*$ for each element $\mathfrak{p}_n(z)$. Fix $n\geq 0$, for all $h,k\geq n$ such that $n=\min(h,k)$, it turns out from the previous proof that
\begin{align*}
\left(\xi_{h,k}\circ \Pi_{h,k}\right)(z) & \stackrel{(\text{def})}{=} \left( \xi_{h,k} \circ  \phi^{-1}_{h,k} \circ  \sigma_{h,k}\circ  \mathfrak{p}_{h+k+1} \right)(z) = \left( \xi_{h,k} \circ  \phi^{-1}_{h,k} \circ  \sigma_{h,k}\circ  \pi_{h+k+1} \right)(x) \\
 & = \left( \xi_{h,k} \circ  \phi^{-1}_{h,k} \circ  \left(\ntau{k}{N}\tensor{R}\ntau{h}{M}\right)^*\circ \phi_{\Sscript{M,N}} \right)(x)  \stackrel{\eqref{eq:phifilt}}{=} \left( \xi_{h,k} \circ   \left( \left(\ntau{h}{M}\right)^*\tensor{R}\left(\ntau{k}{N}\right)^* \right) \right)(x) = \pi_{n+1}(x)
\end{align*}
for some $x\in M^*\tensor{R}  N^*$ and where $\pi_i:M^*\tensor{R}  N^*\to M^*\tensor{R}  N^*/\cF_i\left(M^*\tensor{R}  N^*\right)$ is the canonical projection, for all $i\geq 0$. Now, $\mathfrak{p}_{h+k+1}(z)=\pi_{h+k+1}(x)=\mathfrak{p}_{h+k+1}(\gamma(x))$ implies that $z-\gamma(x)\in\ker{\mathfrak{p}_{h+k+1}}\subseteq \ker{\mathfrak{p}_{n+1}}$, since $n\leq h+k$. Thus $\pi_{n+1}(x)=\mathfrak{p}_{n+1}(\gamma(x))=\mathfrak{p}_{n+1}(z)$ as well and hence $\xi_{h,k}\circ \Pi_{h,k} = \mathfrak{p}_{\min(h,k)+1}$. In particular,
\begin{equation}\label{eq:zlim}
z=\limn{\left(\mathfrak{p}_{n+1}(z)\right)}=\limn{\left(\left(\xi_{n,n}\circ \Pi_{n,n}\right)(z)\right)}.
\end{equation}
\end{remark}

\bigskip
\textbf{Acknowledgement.} 
Both authors would like to thank the referee for the careful reading and the helpful comments on a previous version of the paper.
Moreover, the second author would like to thank the members of the campus in Ceuta for their friendship and the warm hospitality during his stays there.

\newcommand{\etalchar}[1]{$^{#1}$}
\def\cprime{$'$} \def\polhk#1{\setbox0=\hbox{#1}{\ooalign{\hidewidth
  \lower1.5ex\hbox{`}\hidewidth\crcr\unhbox0}}} \def\cprime{$'$}

\end{document}